\documentclass[11pt]{article}
\usepackage{amssymb, amsmath, amsthm, srcltx, oldgerm}
\usepackage{graphicx}
\usepackage{color}
\usepackage{comment}
\usepackage{fancyhdr}
\usepackage{titletoc}
\usepackage[linktocpage=true]{hyperref}
\newtheorem{thm}{Theorem}[section]
\newtheorem{lem}[thm]{Lemma}
\newtheorem{cor}[thm]{Corollary}
\newtheorem{prop}[thm]{Proposition}
\theoremstyle{definition}
\newtheorem{assn}[thm]{Assumption}
\newtheorem{defn}[thm]{Definition}
\theoremstyle{remark}
\newtheorem{rem}[thm]{Remark}
\numberwithin{equation}{section}
\newcommand{\Tr}{\mathrm{Tr}\;\!}

\newcommand{\Real}{\mathbb R}

\newcommand{\E}{\mathbf{E}\,}
\newcommand{\F}{\mathbf{F}}            
\newcommand{\I}{\mathbf{I}}            
\newcommand{\J}{\mathbf{J}}            
\newcommand{\M}{\mathbf{M}}
\newcommand{\R}{\mathbf{R}}
\newcommand{\U}{\mathbf{U}}
\newcommand{\V}{\mathbf{V}}
\newcommand{\W}{\mathbf{W}}
\newcommand{\X}{\mathbf{X}}
\newcommand{\Y}{\mathbf{Y}}
\newcommand{\Z}{\mathbf{Z}}
\newcommand{\re}{\mathrm{Re}\;\!}
\newcommand{\im}{\mathrm{Im}\;\!} 

\newenvironment{Proof of}{\removelastskip\par\medskipstraightforward
\noindent{\em Proof of} \rm}{\penalty-20\null\hfill$\square$\par\medbreak}

\makeatletter
\renewcommand\tableofcontents{\paragraph{Table of Contents.} \ \@starttoc{toc}}
\makeatother

\begin{document}
\date{August 13, 2014}
\title{\bf Asymptotic Spectra of Matrix-Valued Functions \\
of Independent Random Matrices \\ and Free Probability}

\author{{\bf F. G\"otze$^1$}\\{\small Faculty of Mathematics}
\\{\small University of Bielefeld}\\{\small Germany}
\and {\bf H. K\"osters$^1$}\\{\small Faculty of Mathematics}
\\{\small University of Bielefeld}\\{\small Germany}
\and {\bf A. Tikhomirov}$^{1,2}$\\{\small Department of Mathematics
}\\{\small Komi Research Center of Ural Branch of RAS,}\\{\small Syktyvkar State University}
\\{\small Russia}
}

\maketitle
 \footnote{$^1$Partially supported by CRC 701 ``Spectral Structures and Topological Methods in Mathematics'', Bielefeld.
 $^2$Partially supported by RFBR, grant N 14-01-00500 and by Program of Fundamental Research
 Ural Division of RAS 12-P-1-1013.}

\title{}
\author{}

\maketitle

\begin{abstract}
We investigate the universality of singular value and eigenvalue distributions 
of matrix valued functions of independent random matrices and apply 
these general results in several examples. In particular we determine 
the limit distribution and prove universality under general conditions  
for singular value and eigenvalue distributions of products 
of independent matrices from spherical ensembles.
\end{abstract}

\titlecontents{section}[1.5em]{\addvspace{1pt}}{\contentslabel{1.5em}}{}{\titlerule*[0.72pc]{.}\contentspage}


\section{Introduction} \label{sec:notation}

One of the main questions studied in Random Matrix Theory 
is the asymptotic universality, meaning the dependence
on a few global characteristics of the distribution of the matrix entries,  
of the distribution of spectra of random matrices 
when their dimension goes to infinity.
This holds for the spectra of Hermitian random matrices 
with independent entries (up to symmetry),
first proved by Wigner in 1955 \cite{Wigner:55}.
Another well studied case is that of sample covariance matrices 
(i.e.\@ $\W=\X\X^*$, where $\X$ is a matrix 
with independent entries), first studied in \cite{MP:67} by Marchenko--Pastur. 
The spectrum of non Hermitian random matrices with 
independent identically distributed entries is universal as well.
The limiting complex spectrum of this \emph{Ginibre--Girko Ensemble}
is the circular law (i.e.\@ the uniform distribution on the unit circle 
in the complex plane).
The universality here was first proved in \cite{Girko:84} by Girko. 
In the last years different models of random matrices 
which were derived from Wigner and Ginibre--Girko matrices were studied.
For instance, in \cite{AGT:10}, \cite{AGT:10a}  
the universality of the singular value distribution 
of powers of Ginibre--Girko matrices was shown.
In \cite{GT:12} and \cite{SR:10} the universality of the spectrum of products of 
independent random matrices from the Ginibre--Girko Ensemble was proved.
Moreover, more recently, the local properties of the spectrum have also been investigated
in the Gaussian case; see e.g.\@ \cite{Akemann-Ipsen-Kieburg} and \cite{Kuijlaars-Zhang}.

In this paper we describe a general approach to prove 
the universality of singular value and eigenvalue distributions
of {\it matrix-valued functions} of independent random matrices. 
More precisely, we consider random matrices of the form
$$
\F = \mathbb{F}(\X^{(1)},\hdots,\X^{(m)}) \,,
$$
where $\X^{(1)},\hdots\X^{(m)}$ are independent \emph{non-Hermitian} random matrices 
with independent entries and $\mathbb{F}$ is a matrix-valued function.
Our approach is based on the Lindeberg principle of replacing
matrix entries with arbitrary distributions by matrix entries with Gaussian distributions. 
This approach has proved to be fruitful and is used by many authors in random matrix theory;
see e.g.\@ \cite{Chaterj}, \cite{TaoVu:10}, \cite{PasShcherb}, \cite{GT:12}. 
To prove the universality of singular value distributions,
we assume a Lindeberg-type condition for the matrix entries
and a certain rank condition as well as certain smoothness conditions
for the matrix-valued function;
see Equations \eqref{lind}, \eqref{rank} and \eqref{mat3} -- \eqref{mat32a}
in Section \ref{sec:universality-SVD}.
To prove the universality of eigenvalue distributions,
we use Girko's principle of Hermitization (see \cite{Girko:84}),
according to which there is a close connection between 
the eigenvalue distribution of the (square) matrix $\F$
and the family of the singular value distributions 
of all \emph{shifted matrices} $\F - \alpha \I$,
with $\alpha \in \mathbb{C}$.
Here we need some assumptions on the large and small singular values of the shifted matrices;
see Conditions $(C0)$, $(C1)$, $(C2)$ in Section \ref{sec:universality-EVD}.

Furthermore, we introduce a general approach to identify the limiting eigenvalue distribution
of the (square) matrix $\F$. 
Our main results here show how to derive the density of the limiting eigenvalue distribution
of the matrix $\F$ from (the $S$-transform of) its limiting singular value distribution.
This derivation can be divided into two major steps:

In a first step, we derive equations for the Stieltjes transforms $g(z,\alpha)$
of the (symmetrized) singular value distributions of the shifted matrices $\F - \alpha \I$
via the $S$-transform $S(z)$ of the (symmetrized) singular value distribution
of the unshifted matrix $\F$. The key system of equations here reads
\begin{align}\label{eq:intro-1}
w(z,\alpha)&=z+\frac{\widetilde R_{\alpha}(-g(z,\alpha))}{g(z,\alpha)},\notag\\
g(z,\alpha)&=(1+w(z,\alpha)g(z,\alpha))S(-(1+w(z,\alpha)g(z,\alpha))),
\end{align}
where $w(z,\alpha)$ is an unknown auxiliary function and $\widetilde R_\alpha(z)$ is a known function.
To derive this system of equations, we use the asymptotic freeness of the matrices
\begin{align}\label{eq:intro-2}
\begin{bmatrix}&\mathbf O&\F&\\&{\F}^*&\mathbf O&\end{bmatrix}
\qquad\text{and}\qquad
\begin{bmatrix}&\mathbf O&-\alpha\I&\\&-\overline \alpha\I&\mathbf O&\end{bmatrix}
\end{align}
as well as the calculus for $R$-transforms and $S$-transforms.
Furthermore, we show that it~is possible take the limit $z \to 0$ in \eqref{eq:intro-1}.
Since we are working in a quite general framework,
the investigation of the existence of this limit 
as well as its analytic properties require some work.

In a second step, we identify the density $f$ of the limiting eigenvalue distribution 
of the random matrix $\F$ using logarithmic potential theory.
The main observation here is that the function $\psi(\alpha) := - w(0,\alpha) g(0,\alpha)$
is closely related to the partial derivatives of the logarithmic potential
of the limiting eigenvalue distribution. Thus, under regularity assumptions,
we obtain the relation
 \begin{equation}\label{eq:intro-3}
 f(u,v)=\frac{1}{2\pi |\alpha|^2}  \left(u\frac{\partial \psi}{\partial u}+v\frac{\partial \psi}{\partial v}\right),
\end{equation}
where $u$ and $v$ denote the real and imaginary part of $\alpha$, respectively.

Let us emphasize that this identification of the limiting eigenvalue distribution is quite general.
In principle, we only need the $S$-transform of the limiting singular value distribution
and the asymptotic freeness of the matrices in \eqref{eq:intro-2}.

%
%

In Section \ref{sec:applications} we give several examples for applications of our main universality results 
(Theorems \ref{singularvalueuniversality} and \ref{eigenvalueuniversality}).
The guiding principle here is (i) to establish universality
and (ii) to compute the~limits in the Gaussian case,
using tools from free probability theory.
\linebreak Here we focus on a special class of matrix-valued functions, 
namely \emph{products} of matrices or powers and inverses thereof.
Although our framework should, in principle, cover more general functions as well,
products of independent matrices represent a convenient class of examples
in which the assumptions of our main results can be checked.
For instance, the~conditions $(C0)$, $(C1)$, $(C2)$ on the large and small singular values
can be deduced from existing results by Tao and Vu \cite{TaoVu:10}
and G\"otze and Tikhomirov \cite{GT:10aop}, \cite{GT:12} here.
Moreover, once universality is proved, it suffices to~identify
the limiting eigenvalue and singular value distributions in the \emph{Gaussian} case.
But if the random matrices $\X^{(1)},\hdots,\X^{(m)}$ 
have independent standard Gaussian entries,
their distributions are invariant under rotations,
and the $S$-transforms of the limiting singular value distributions
of their products are readily obtained using tools from free probability theory,
see e.g.\@ Voiculescu \cite{Voiculescu:98} or Hiai and Petz \cite{Petz-1}.
From here it is possible to obtain the limiting singular value distributions and,
as we have seen, the limiting eigenvalue distributions.


Our~examples illustrate that our main results provide a unifying framework
to derive old and new results for products of independent random matrices.
In particular, we determine the limiting singular value and eigenvalue distributions
for products of independent random matrices from the so-called \emph{spherical ensemble}
(see e.g.\@ \cite{Mays}),
i.e.\@ for products of the form $\X^{(1)} (\X^{(2)})^{-1} \cdots \X^{(2m-1)} (\X^{(2m)})^{-1}$,
where $\X^{(1)},\ldots,\X^{(2m)}$ are independent Girko--Ginibre matrices.
 
\tableofcontents

\bigskip

\section{General Framework}

We now introduce our main assumptions and notation.
Generalizations and specializations will be indicated 
at the beginnings of later sections.

Let $m \geq 1$ be fixed.
Let $\mathcal M_{n\times p}$ denote the space of $n\times p$ matrices. 
Let $\mathbb F=(f_{jk})$, $1\le j\le n$, $1\le k\le p$ be a map 
from the space of $m$-tuples of $n_0 \times n_1$, $n_1\times n_2,\ldots,n_{m-1}\times n_m$ 
matrices $\mathcal M_{n_0\times n_1}\times\cdots\times\mathcal M_{n_{m-1}\times n_m}$ to 
$\mathcal M_{n\times p}$. Here we assume that $n_0=n$ and $n_m=p$.

\pagebreak[2]

In order to study the spectral asymptotics of sequences of such matrix tuples we shall
make a so-called {\it dimension shape assumption}, meaning that $n_q=n_q(n)$, 
and that for any $q=1,\ldots,m$,
\begin{equation}\label{samplesize}
 \lim_{n\to\infty}\frac n{n_q(n)}=y_q>0.
\end{equation}

Let $\X = (\X^{(1)},\ldots, \X^{(m)})$ 
be an $m$-tuple of independent random matrices of dimensions
$n_0\times n_1,\ldots,n_{m-1}\times n_m$, respectively, 
with independent entries. More precisely, we assume that 
$$
\X^{(q)} = (\tfrac{1}{\sqrt{n_q}} X^{(q)}_{jk}),
$$
where the $X^{(q)}_{jk}$ are independent complex random variables
such that for all $q=1,\hdots,m$ and $j=1,\hdots,n_{q-1};\,k=1,\hdots,n_{q}$,
we have $\E X_{jk}^{(q)}=0$ and $\E|X_{jk}^{(q)}|^2=1$.

Furthermore, let $\Y = (\Y^{(1)},\ldots,\Y^{(m)})$ 
be an $m$-tuple of independent random matrices of dimensions
$n_0\times n_1,\ldots,n_{m-1}\times n_m$, respectively,
with independent {\it Gaussian} entries. 
More~precisely, we assume that
$$
\Y^{(q)} = (\tfrac{1}{\sqrt{n_q}} Y^{(q)}_{jk}) \,,
$$
where the $Y^{(q)}_{jk}$ are independent complex random variables
such that for all $q=1,\hdots,m$ and $j=1,\hdots,n_{j-q};\,k=1,\hdots,n_{q}$,
$(\re Y^{(q)}_{jk},\im Y^{(q)}_{jk})$ has a bivariate Gaussian distribution
with the same first and second moments as $(\re X^{(q)}_{jk},\im X^{(q)}_{jk})$.
By this we mean that
\begin{align}
\label{eq:secondmomentstructure-1}
\E \re Y^{(q)}_{jk} = \E \re X^{(q)}_{jk} &\,,\
\E \im Y^{(q)}_{jk} = \E \im X^{(q)}_{jk} \,,\
\nonumber\\
\E |\re Y^{(q)}_{jk}|^2 = \E |\re X^{(q)}_{jk}|^2 &\,,\
\E |\im Y^{(q)}_{jk}|^2 = \E |\im X^{(q)}_{jk}|^2 \,,\
\nonumber\\
\E (\re Y^{(q)}_{jk} \im Y^{(q)}_{jk}) &= \E (\re X^{(q)}_{jk} \im X^{(q)}_{jk}) \,.
\end{align}
In particular, $\E Y_{jk}^{(q)}=0$ and $\E|Y_{jk}^{(q)}|^2=1$.

In Section~\ref{sec:applications}, 
when we determine the limiting singular value and eigenvalue distributions
in the Gaussian case, we will impose the stronger assumption that the $Y^{(q)}_{jk}$ 
are \emph{standard} real or complex Gaussian random variables.
By Eq.~\eqref{eq:secondmomentstructure-1}, this entails some restrictions
on the second moments of the $X^{(q)}_{jk}$.
%
%

We shall also assume that the random matrices 
$\X^{(1)},\ldots,\X^{(m)}$ and $\Y^{(1)},\ldots,\Y^{(m)}$ 
are defined on the same probability space and that
$\Y^{(1)},\ldots,\Y^{(m)}$ are independent of $\X^{(1)},\ldots,\X^{(m)}$. 
Finally, for any $m$-tuple $\Z = (\Z^{(1)},\hdots,\Z^{(m)})$
in $\mathcal M_{n_0\times n_1}\times\cdots\times\mathcal M_{n_{m-1}\times n_m}$,
we~set 
\begin{equation}\label{eq:F-definition}
\F_{\Z} := \mathbb F(\Z^{(1)},\hdots,\Z^{(m)}) \,.
\end{equation}
Note that since we are interested in asymptotic singular value and eigenvalue distributions,
we are actually dealing with \emph{sequences} of matrix tuples of increasing dimension.
However, the dependence on $n$ is usually suppressed in our notation.

Throughout this paper, we use the following notation.
For a matrix $\mathbf A = (a_{jk})$ $\in \mathcal M_{n \times p}$,
we write $\|\mathbf A\|$ for the \emph {operator norm} of $\mathbf A$
and $\|\mathbf A\|_2 := (\sum_{j=1}^{n} \sum_{k=1}^{p} |a_{jk}|^2)^{1/2}$
for the \emph{Frobenius norm} of $\mathbf A$.
The singular values of $\mathbf A$ are the square-roots
of the eigenvalues of the $n \times n$ matrix $\mathbf A \mathbf A^*$.
Finally, unless otherwise indicated, $C$ and $c$ denote
sufficiently large and small positive constants, respectively,
which may change from step to step.

\section{Universality of Singular Value Distributions of Functions \\ of Independent Random Matrices}
\label{sec:universality-SVD}

We start with the singular value distribution of functions of independent random matrices.
Let $\X = (\X^{(1)},\hdots,\X^{(m)})$ be an $m$-tuple 
of independent random matrices with independent entries
as in Section \ref{sec:notation},
and let $\F_{\X}=\mathbb F(\X^{(1)},\ldots,\X^{(m)})$
be a matrix-valued function of~$\X$.
We are interested in the empirical distribution of the singular values of $\F_\X$,
i.e.\@ of the square-roots of the eigenvalues of $\F_\X^{} \F_\X^*$.
  
We shall assume that the random variables $X_{jk}^{(q)}$, 
for $q=1,\ldots,m$, $j=1,\ldots,n_{q-1};\,k=1,\ldots,n_q$,
satisfy the following {\it Lindeberg condition}, 
i.~e.
\begin{equation}\label{lind}
 \text{for any $\tau > 0$,} \quad
 L_n(\tau):=\frac1{n^2}\sum_{q=1}^m\sum_{j=1}^{n_{q-1}}\sum_{k=1}^{n_q}
 \E|X_{jk}^{(q)}|^2 \mathbb I\{|X_{jk}^{(q)}|>\tau\sqrt n\}\to0
\quad \text{ as }n\to \infty.
\end{equation}

We shall assume as well that the function $\mathbb F$ satisfies a so-called {\it rank condition}, 
i.~e.\@ for any $m$-tuples $({\mathbf A}^{(1)},\ldots,{\mathbf A}^{(m)}),({\mathbf B}^{(1)},\ldots,{\mathbf B}^{(m)})
\in \mathcal M_{n_0\times n_1}\times\cdots\times\mathcal M_{n_{m-1}\times n_{m}}$ 
we have
\begin{equation}\label{rank}
\text{\rm rank}\{\mathbb F(\mathbf A^{(1)},\ldots,\mathbf A^{(m)})
-\mathbb F(\mathbf B^{(1)},\ldots,\mathbf B^{(m)})\}
\le
C(\mathbb F)\sum_{q=1}^m
\text{\rm rank}
\{\mathbf A^{(q)}-\mathbf B^{(q)}\}.
\end{equation}

We now define truncated matrices. 
Note that by \eqref{lind} there exists a sequence $(\tau_n)$ such that
\begin{equation}\label{asumplind}
 \tau_n\to0 \qquad\text{and}\quad L_n(\tau_n)\tau_n^{-4}\to 0 \qquad \text{ as } n\to\infty.
\end{equation}
Clearly, we may additionally require that $\tau_n \geq n^{-1/3}$ for all $n$.
We fix such a sequence and consider the matrix tuple
$\widehat{\X} = (\widehat{\X}^{(1)},\hdots,\widehat{\X}^{(m)})$
consisting of the matrices
$\widehat {\X}^{(q)} = (\tfrac{1}{\sqrt{n_q}} \widehat X_{jk}^{(q)})$,
$q=1,\hdots,m$, where
\begin{equation}\notag
 \widehat{X}_{jk}^{(q)}=X_{jk}^{(q)}\mathbb I\{|X_{jk}^{(q)}|\le\tau_n\sqrt n\}.
\end{equation}
Let $\mathbf B$ be a non-random matrix of order $n\times p$,
let $\F_{\X}$ and $\F_{\widehat{\X}}$
be defined as in \eqref{eq:F-definition}, and let 
$s_1({\X})\ge\ldots\ge s_n({\X})$
and 
$s_1({\widehat{\X}})\ge\ldots\ge s_n({\widehat{\X}})$ 
denote the singular values of the matrices
$\F_{\X}+\mathbf B$ and $\F_{\widehat{\X}}+\mathbf B$,
respectively.
Let $\mathcal F_{\X}(x)$ (resp. $\mathcal F_{\widehat{\X}}(x)$)
denote the empirical distribution function of the \emph{squared} singular values 
of the matrix $\F_{\X}+\mathbf B$ 
(resp. $\F_{\widehat{\X}}+\mathbf B$), i.e.
\begin{equation}\notag
\mathcal F_{\X}(x)=\frac1n\sum_{j=1}^n\mathbb I\{s_j^2(\X)\le x\},\quad 
\mathcal F_{\widehat{\X}}(x)=\frac1n\sum_{j=1}^n\mathbb I\{s_j^2(\widehat{\X})
\le x\},
\quad x \in \mathbb R \,.
\end{equation}
The corresponding Stieltjes transforms  of these empirical distributions
are denoted by $m_{\X}(z)$ and $m_{\widehat{\X}}(z)$, i.e.
\begin{equation}\notag
 m_{\X}(z)=\frac1n\sum_{j=1}^n\frac1{s_j^2(\X)-z},\quad 
m_{\widehat{\X}}(z)=\frac1n\sum_{j=1}^n\frac1{s_j^2(\widehat{\X})-z},
\quad z \in \mathbb C_+ \,.
\end{equation}

\pagebreak[2]
First we prove the following
\begin{lem}
\label{lem:truncation}
Assume that the conditions \eqref{lind} and \eqref{rank} hold.
Then
\begin{equation}\notag
\E \sup_x|\mathcal F_{\X}(x)-\mathcal F_{\widehat{\X}}(x)|\le C\tau_n^2,
\end{equation}
and, for any $z = u + iv$ with $v > 0$,
\begin{equation}\notag
 \E|m_{\X}(z)-m_{\widehat{\X}}(z)|\le Cv^{-1}\tau_n^2.
\end{equation}
\end{lem}

\begin{proof}
 By the rank inequality of Bai, see \cite{BS:10}, Theorem A.44, we have
\begin{equation}\label{trunc1}
\E \sup_x|\mathcal F_{\X}(x)-\mathcal F_{\widehat{\X}}(x)|
\le \frac1{n}\E{\text{\rm rank}
\{\F_{\X}-\F_{\widehat{\X}}\}},
\end{equation}
and, by integration by parts,
\begin{equation}\label{trunc2}
 \E|m_{\X}(z)-m_{\widehat{\X}}(z)|\le\frac{\pi}{nv}\E{\text{\rm rank}
\{\F_{\X}-\F_{\widehat{\X}}\}}.
\end{equation}
By condition \eqref{rank}, we have
\begin{equation}\label{trunc3}
 \text{\rm rank}
\{\F_{\X}-\F_{\widehat{\X}}\}
\le 
C(\mathbb F)\sum_{q=1}^m\text{\rm rank}\{\X^{(q)}-{\widehat{\X}}^{(q)}\}.
\end{equation}
Furthermore,
\begin{align}\label{trunc4}
\sum_{q=1}^m\E\text{\rm rank}\{\X^{(q)}-{\widehat{\X}}^{(q)}\}
&\le \sum_{q=1}^m\sum_{j=1}^{n_{q-1}}\sum_{k=1}^{n_q}
\E\mathbb I\{|X_{jk}^{(q)}|\ge \tau_n\sqrt n\}\notag\\&
\le \frac1{n\tau_n^2}\sum_{q=1}^m \sum_{j=1}^{n_{q-1}}\sum_{k=1}^{n_q}
\E|X_{jk}^{(q)}|^2\mathbb I\{|X_{jk}^{(q)}|\ge \tau_n\sqrt n\}\notag\\
&=\frac{nL_n(\tau_n)}{\tau_n^2}. 
\end{align}
Inequalities \eqref{trunc1}--\eqref{trunc4} and assumption \eqref{asumplind} 
together complete the proof of the Lemma.
\end{proof}

\emph{Remark.}
Lemma \ref{lem:truncation} is about the distribution functions of the squared singular values,
$\mathcal F_{\X}(x)=\frac1n\sum_{j=1}^n\mathbb I\{s_j^2(\X)\le x\}$ ($x>0$).
Similar results hold for the distribution functions of the non-squared singular values,
$\mathcal F_{\X}(x^2)=\frac1n\sum_{j=1}^n\mathbb I\{s_j(\X)\le x\}$ ($x>0$),
as well as for their symmetrizations,
$\widetilde{\mathcal F}_{\X}(x):=\frac12(1+\text{sign}(x) \mathcal F_{\X}(x^2))$ ($x \ne 0$). 
It is this consequence of~Lemma \ref{lem:truncation}
that will be used below.

\medskip

Let $\Y = (\Y^{(1)},\hdots,\Y^{(m)})$ be an $m$-tuple 
of independent random matrices with independent {\it Gaussian} entries
as in Section \ref{sec:notation},
and let $\widehat{\Y} = (\widehat{\Y}^{(1)},\hdots,\widehat{\Y}^{(m)})$
denote the $m$-tuple consisting of the matrices
$\widehat{\Y}^{(q)} = (\tfrac{1}{\sqrt{n_q}} \widehat{Y}^{(q)}_{jk})$,
where
$$
 \widehat{Y}_{jk}^{(q)}=Y_{jk}^{(q)}\mathbb I\{|Y_{jk}^{(q)}|\le\tau_n\sqrt n\}.
$$
for $q=1,\hdots,m$ and $j=1,\hdots,n_{q-1};\,k=1,\hdots,n_{q}$.
Set 
$$
 \widetilde L_n(\tau_n):=\frac1{n^2}\sum_{q=1}^m\sum_{j=1}^{n_{q-1}}\sum_{k=1}^{n_q}
 \E|Y_{jk}^{(q)}|^2 \mathbb I\{|Y_{jk}^{(q)}|>\tau_n\sqrt n\} \,.
$$
Then, using the relation $\tau_n \geq n^{-1/3}$ and the special properties of the Gaussian distribution,
it is easy to check that we also have
\begin{align}
\label{asumplind2}
\widetilde{L}_n(\tau_n) \tau_n^{-4} \to 0 \qquad \text{ as } n \to \infty.
\end{align}
Furthermore, note that for the truncated random variables, 
the moment identities \eqref{eq:secondmomentstructure-1}
need not hold anymore. However, we have the relations
\begin{equation}
\label{eq:momentsAfterTruncation-11}
|\E \widehat X_{jk}^{(q)}| = \left| \E X_{jk}^{(q)} \mathbb I\{|X_{jk}^{(q)}|>\tau_n\sqrt n\} \right|
\le \frac{1}{\tau_n\sqrt{n}} \E |X_{jk}^{(q)}|^2 \mathbb I\{|X_{jk}^{(q)}|>\tau_n\sqrt n\} \,,
\end{equation}
\begin{equation}
\label{eq:momentsAfterTruncation-12}
     \left| \E (\re \widehat{X}_{jk}^{(q)})^2-\E (\re X_{jk}^{(q)})^2 \right| 
\leq \E|X_{jk}^{(q)}|^2 \mathbb I\{|X_{jk}^{(q)}|>\tau_n\sqrt n\} \,,
\end{equation}
\begin{equation}
\label{eq:momentsAfterTruncation-13}
     \left| \E (\im \widehat{X}_{jk}^{(q)})^2-\E (\im X_{jk}^{(q)})^2 \right| 
\leq \E|X_{jk}^{(q)}|^2 \mathbb I\{|X_{jk}^{(q)}|>\tau_n\sqrt n\} \,,
\end{equation}
\begin{equation}
\label{eq:momentsAfterTruncation-14}
     \left| \E (\re \widehat{X}_{jk}^{(q)} \im \widehat{X}_{jk}^{(q)})-\E (\re X_{jk}^{(q)} \im X_{jk}^{(q)}) \right| 
\leq \E|X_{jk}^{(q)}|^2 \mathbb I\{|X_{jk}^{(q)}|>\tau_n\sqrt n\} \,,
\end{equation}
as well as the analogous relations for the r.v.'s $\widehat Y_{jk}^{(q)}$,
which imply that
\begin{equation}
\label{eq:momentsAfterTruncation-16}
\frac1{n^{3/2}}\sum_{q=1}^m\sum_{j=1}^{n_{q-1}}\sum_{k=1}^{n_q} \left( | \E \widehat X_{jk}^{(q)} | + | \E \widehat Y_{jk}^{(q)} | \right) \le \frac{L(\tau_n) + \widetilde L(\tau_n)}{\tau_n} \,, 
\end{equation}
\begin{equation}
\label{eq:momentsAfterTruncation-17}
\frac1{n^2}\sum_{q=1}^m\sum_{j=1}^{n_{q-1}}\sum_{k=1}^{n_q} \left| \E (\re \widehat X_{jk}^{(q)})^2 - \E (\re \widehat Y_{jk}^{(q)})^2 \right| \le L(\tau_n) + \widetilde L(\tau_n) \,, 
\end{equation}
\begin{equation}
\label{eq:momentsAfterTruncation-18}
\frac1{n^2}\sum_{q=1}^m\sum_{j=1}^{n_{q-1}}\sum_{k=1}^{n_q} \left| \E (\im \widehat X_{jk}^{(q)})^2 - \E (\im \widehat Y_{jk}^{(q)})^2 \right| \le L(\tau_n) + \widetilde L(\tau_n) \,, 
\end{equation}
\begin{equation}
\label{eq:momentsAfterTruncation-19}
\frac1{n^2}\sum_{q=1}^m\sum_{j=1}^{n_{q-1}}\sum_{k=1}^{n_q} \left| \E (\re \widehat X_{jk}^{(q)})(\im \widehat X_{jk}^{(q)}) - \E (\re \widehat Y_{jk}^{(q)})(\im \widehat Y_{jk}^{(q)}) \right| \le L(\tau_n) + \widetilde L(\tau_n) \,. 
\end{equation}

\pagebreak[2]

For the rest of this section, we use the following notation. 
For any matrix tuple $\X = (\X^{(1)},\hdots,\X^{(m)})$
and any $n \times p$ matrix $\mathbf B$, we~introduce the matrix
$\F_\X$ as in \eqref{eq:F-definition},
the Hermitian matrix
\begin{equation}\notag
  \V_{\X}:=\begin{bmatrix}&\mathbf O&\F_{\X}+\mathbf B&\\&
  ({\F}_{\X}+\mathbf B)^*&\mathbf O&\end{bmatrix},
\end{equation}
as well as the corresponding resolvent matrix
\begin{equation}\notag
\R_{\X}:= \R_{\X}(z)=(\V_{\X}-z\I)^{-1}.
\end{equation}
Furthermore, let $s_1({\X})\ge\ldots\ge s_n({\X})$
denote the singular values of the matrix $\F_{{\X}}+\mathbf B$.
Note that, apart from a fixed number of zero eigenvalues,
the eigenvalues of the matrix $\V_{\X}$ are given by
$\pm s_1({\X}),\ldots\,\pm s_n({\X})$.
The corresponding Stieltjes transform will be denoted by 
\begin{equation}\notag
m_n(z,\mathbf{X}) := \frac1{2n} \left( \Tr \R_{\X} + \frac{p-n}{z} \right).
\end{equation}

For $0\le\varphi\le \frac{\pi}2$ and $q=1,\ldots,m$, let
\begin{equation}\label{reprsimple}
\Z^{(q)}(\varphi)=
(\widehat{\X}^{(q)}\cos\varphi +\widehat{\Y}^{(q)}\sin\varphi),
\end{equation}
and $\Z(\varphi) := (\Z^{(1)}(\varphi),\hdots,\Z^{(m)}(\varphi))$.
For abbreviation, we shall write
$\mathbf{F}(\varphi)$, $\V(\varphi)$, $\mathbf{R}(\varphi)$, and $m_n(z,\varphi)$
instead of 
$\mathbf{F}_{\Z(\varphi)}$, $\V_{\Z(\varphi)}$, $\mathbf{R}_{\Z(\varphi)}$, and $m_n(z,\Z(\varphi))$.
With this notation we~have 
$\F_{\widehat{\X}}=\F(0)$, $\F_{\widehat{\Y}} =\F(\frac{\pi}{2})$,
$m_n(z,\widehat{\X})=m_n(z,0)$ and $m_n(z,\widehat{\Y})=m_n(z,\frac{\pi}2)$.
Also, we~may~write 
\begin{equation}\notag
 m(z,\tfrac{\pi}2)-m_n(z,0)
 =\int_0^{\frac{\pi}2}\frac{\partial m_n(z,\varphi)}{\partial\varphi}d\varphi.
\end{equation}
The representation of type \eqref{reprsimple} has been used
for sums of random variables, for instance, in \cite{Bentkus} 
(second relation on page 367). 
For random matrices \eqref{reprsimple} has been used, for example, 
by Pastur and Lytova in \cite{PasLut} (see Equation (60)).

A simple computation shows that
\begin{equation}\notag
 \frac{\partial m_n(z,\varphi)}{\partial\varphi}
 =-\frac1{2n}\Tr\frac{\partial\V(\varphi)}{\partial\varphi}\R^2(\varphi).
\end{equation}
Furthermore, using this relation we get
\begin{align}\label{mat0}
 \frac{\partial m_n(z,\varphi)}{\partial\varphi}
 =&-\frac1{2n}\sum_{q=1}^m\sum_{j=1}^{n_{q-1}}\sum_{k=1}^{n_{q}}
\frac1{\sqrt{n_q}}\Big({-}\re \widehat X_{jk}^{(q)}\sin\varphi +\re \widehat Y_{jk}^{(q)}\cos\varphi
\Big)\Tr\frac{\partial \V}{\partial \re Z_{jk}^{(q)}}\R^2\notag\\&-
\frac i{2n}\sum_{q=1}^m\sum_{j=1}^{n_{q-1}}\sum_{k=1}^{n_{q}}
\frac1{\sqrt{n_q}}\Big({-}\im \widehat X_{jk}^{(q)}\sin\varphi +\im \widehat Y_{jk}^{(q)}\cos\varphi
\Big)\Tr\frac{\partial \V}{\partial \im Z_{jk}^{(q)}}\R^2.
\end{align}
We denote by
\begin{align}\label{def+}
g_{jk}^{(q)}&:= g_{jk}^{(q)}(\Z^{(1)}(\varphi),\ldots,\Z^{(m)}(\varphi))=
\Tr\frac{\partial \V}{\partial \re Z_{jk}^{(q)}}\R^2,\notag\\
\widehat g_{jk}^{(q)}
&:=\widehat  g_{jk}^{(q)}(\Z^{(1)}(\varphi),\ldots,\Z^{(m)}(\varphi))=
\Tr\frac{\partial \V}{\partial \im Z_{jk}^{(q)}}\R^2.
\end{align}
Let $g_{jk}^{(q)}(\theta)$ denote the function obtained from $g_{jk}^{(q)}$ 
by replacing the indeterminate $Z_{jk}^{(q)}$ with $\theta Z_{jk}^{(q)}$.


\begin{thm}\label{singularvalueuniversality} 
Assume that the Lindeberg condition \eqref{lind}
and the rank condition \eqref{rank} hold.
Furthermore suppose that there exist constants $A_0>0$, $A_1>0$ and $A_2>0$ 
such that for any random~variable $\theta$ which is uniformly distributed 
on the interval $[0,1]$ and independent of the r.v.'s $X_{jk}^{(q)}$ and $Y_{jk}^{(q)}$,
the following conditions hold:
\begin{align}
\label{mat3}
&\sup_{j,k,q} \Big\|\E\Big\{g_{jk}^{(q)}(\theta)\Big|X_{jk}^{(q)}, Y_{jk}^{(q)}\Big\}\Big\|_{\infty}
\le A_0,\\
\label{mat3a}
&\sup_{j,k,q} \Big\|\E\Big\{\widehat g_{jk}^{(q)}(\theta)\Big|X_{jk}^{(q)}, Y_{jk}^{(q)}\Big\}\Big\|_{\infty}
\le A_0,\\
\label{mat31}
&\sup_{j,k,q}\max\Big\{\Big\|\E\Big\{\frac{\partial g_{jk}^{(q)}(\theta)}
{\partial {\re Z_{jk}^{(q)}}}
\Big|X_{jk}^{(q)}, Y_{jk}^{(q)}\Big\}
\Big\|_{\infty},\Big\|\E\Big\{\frac{\partial g_{jk}^{(q)}(\theta)}{\partial {\im Z_{jk}^{(q)}}}
\Big|X_{jk}^{(q)}, Y_{jk}^{(q)}\Big\}\Big\|_{\infty}\Big\}
\le A_1 ,\\
\label{mat31a}
&\sup_{j,k,q}\max\Big\{\Big\|\E\Big\{\frac{\partial \widehat g_{jk}^{(q)}(\theta)}
{\partial {\re Z_{jk}^{(q)}}}
\Big|X_{jk}^{(q)}, Y_{jk}^{(q)}\Big\}\Big\|_{\infty},\Big\|
\E\Big\{\frac{\partial \widehat g_{jk}^{(q)}(\theta)}{\partial {\im Z_{jk}^{(q)}}}
\Big|X_{jk}^{(q)}, Y_{jk}^{(q)}\Big\}\Big\|_{\infty}\Big\}
\le A_1 ,\displaybreak[2]\\
&\sup_{j,k,q}\max\Big\{\Big\|\E\Big\{\frac{\partial^2 g_{jk}^{(q)}(\theta)}
{\partial {\re Z_{jk}^{(q)}}^2}\Big|X_{jk}^{(q)}, Y_{jk}^{(q)}\Big\}
\Big\|_{\infty},
\Big\|\E\Big\{\frac{\partial^2 g_{jk}^{(q)}(\theta)}{\partial {\im Z_{jk}^{(q)}}^2}
\Big|X_{jk}^{(q)}, Y_{jk}^{(q)}\Big\}
\Big\|_{\infty},\notag\\
\label{mat32}&\qquad\qquad\qquad\qquad\qquad\qquad\qquad
\Big\|\E\Big\{\frac{\partial^2 g_{jk}^{(q)}(\theta)}
{\partial {\re Z_{jk}^{(q)}} \partial\im Z_{jk}^{(q)}}\Big|X_{jk}^{(q)}, Y_{jk}^{(q)}\Big\}
\Big\|_{\infty}\Big\}\le A_2,\\
&\sup_{j,k,q}\max\Big\{\Big\|\E\Big\{\frac{\partial^2 \widehat g_{jk}^{(q)}(\theta)}
{\partial {\re Z_{jk}^{(q)}}^2}\Big|X_{jk}^{(q)}, Y_{jk}^{(q)}\Big\}
\Big\|_{\infty},
\Big\|\E\Big\{\frac{\partial^2\widehat  g_{jk}^{(q)}(\theta)}
{\partial {\im Z_{jk}^{(q)}}^2}\Big|X_{jk}^{(q)}, Y_{jk}^{(q)}\Big\}
\Big\|_{\infty},\notag\\
\label{mat32a}
&\qquad\qquad\qquad\qquad\qquad\qquad\qquad
\Big\|\E\Big\{\frac{\partial^2\widehat  g_{jk}^{(q)}(\theta)}
{\partial {\re Z_{jk}^{(q)}}\partial\im Z_{jk}^{(q)}}\Big|X_{jk}^{(q)}, Y_{jk}^{(q)}\Big\}
\Big\|_{\infty}\Big\}\le A_2.
\end{align}
Then, for any $z=u+iv$ with $v>0$,
\begin{equation}\notag
 \lim_{n\to\infty}(m_n(z,\Y)-m_n(z,\X))=0 \qquad \text{in probability} \,.
\end{equation}
\end{thm}

\emph{Remark.} 
It follows from the conclusion of the theorem 
and basic properties of the Stieltjes transform
that if the singular value distributions of the matrices $\F_\Y+\mathbf B$
are weakly convergent in probability to some limit $\nu$,
then so are the singular value distributions of the matrices $\F_\X+\mathbf B$.
In~this sense Theorem \ref{singularvalueuniversality} proves 
the \emph{universality of singular value distributions}.

\begin{proof}[Proof of Theorem \ref{singularvalueuniversality}]
By Lemma \ref{lem:truncation} and the subsequent remark, 
it is sufficient to prove the claim with $m_n(z,\widehat{\X})$ and $m_n(z,\widehat{\Y})$ 
instead of $m_n(z,\X)$ and $m_n(z,\Y)$.
Furthermore, according to Lemma \ref{variance} in the Appendix, it is enough to prove that
\begin{equation}
\label{eq:convergenceOfExpectation-1}
 \lim_{n\to\infty}\E(m(z,\tfrac{\pi}2)-m_n(z,0))=0.
\end{equation}
%
%
%
Using Taylor's formula in the form
\begin{align}\notag
f(x,y)=f(0,0)+xf'_x(0,0)&+yf'_y(0,0)+x^2\E_{\theta}(1-\theta)
f_{xx}''(\theta x,\theta y)\notag\\&+2xy\E_{\theta}(1-\theta)f_{xy}''(\theta x,\theta y)+
 y^2\E_{\theta}(1-\theta)f_{yy}''(\theta x,\theta y),
\end{align}
where $\theta$ is a random variable which is uniformly distributed on the unit interval,
we get
\begin{align}\label{mat1}
 g_{jk}^{(q)}&=g_{jk}^{(q)}(0)+
\frac1{\sqrt{n_{q}}}(\re \widehat X_{jk}^{(q)}\cos\varphi+\re \widehat Y_{jk}^{(q)}\sin\varphi)
\frac{\partial g_{jk}^{(q)}}
{\partial \re Z_{jk}^{(q)}}(0)\notag\\&
+\frac1{\sqrt{n_{q}}}(\im \widehat X_{jk}^{(q)}\cos\varphi+\im \widehat Y_{jk}^{(q)}\sin\varphi)
\frac{\partial g_{jk}^{(q)}}
{\partial \im Z_{jk}^{(q)}}(0)\notag\\&+\frac{1}{{n_{q}}}(\re \widehat X_{jk}^{(q)}
\cos\varphi+\re \widehat Y_{jk}^{(q)}\sin\varphi)^2
\E_{\theta}(1-\theta)\frac{\partial^2 g_{jk}^{(q)}}
{\partial\re {Z_{jk}^{(q)}}^2}(\theta)\notag\\&+\frac{2}{{n_{q}}}(\re \widehat X_{jk}^{(q)}
\cos\varphi+\re \widehat Y_{jk}^{(q)}\sin\varphi)
(\im \widehat X_{jk}^{(q)}\cos\varphi+\im \widehat Y_{jk}^{(q)}\sin\varphi)\notag\\&
\qquad\qquad\qquad\qquad\qquad\qquad\qquad\times
\E_{\theta}(1-\theta)\frac{\partial^2 g_{jk}^{(q)}}
{\partial\re {Z_{jk}^{(q)}}\partial\im {Z_{jk}^{(q)}}}(\theta)\notag\\&
+\frac{1}{{n_{q}}}
(\im \widehat X_{jk}^{(q)}\cos\varphi+\im \widehat Y_{jk}^{(q)}\sin\varphi)^2
\E_{\theta}(1-\theta)\frac{\partial^2 g_{jk}^{(q)}}
{\partial{\im Z_{jk}^{(q)}}^2}(\theta).
\end{align}
Here $\E_{\theta}$ denotes the expectation with respect to the r.v. $\theta$ 
conditioning on all other r.v.'s.
Inserting (\ref{mat1}) into (\ref{mat0}), we get
\begin{align}\notag
\E(m(z,\tfrac{\pi}2)-m_n(z,0))= - \sum_{j=1}^{12} \Sigma_j, 
\end{align}
where
\begin{align}\notag
 \Sigma_1=&\frac1{2n}\sum_{q=1}^m\sum_{j=1}^{n_{q-1}}\sum_{k=1}^{n_{q}}
\int_0^{\frac{\pi}2}\frac1{\sqrt{n_{q}}}\E\bigg((-\re \widehat X_{jk}^{(q)}\sin\varphi 
+\re \widehat Y_{jk}^{(q)}\cos\varphi)g_{jk}^{(q)}(0,0)\bigg) d\varphi,
\notag\\
\Sigma_2=&
\frac1{2n}\sum_{q=1}^m\sum_{j=1}^{n_{q-1}}\sum_{k=1}^{n_{q}}
\int_0^{\frac{\pi}2}\frac1{n_{q}}\E\bigg((-\re \widehat X_{jk}^{(q)}\sin\varphi +\re \widehat Y_{jk}^{(q)}
\cos\varphi)\notag\\&\qquad\qquad\qquad
\times(\re \widehat X_{jk}^{(q)}\cos\varphi+\re \widehat Y_{jk}^{(q)}
\sin\varphi)\frac{\partial g_{jk}^{(q)}}
{\partial \re Z_{jk}^{(q)}}(0,0)\bigg) d\varphi,\notag\\
\Sigma_3=&
\frac1{2n}\sum_{q=1}^m\sum_{j=1}^{n_{q-1}}\sum_{k=1}^{n_{q}}
\int_0^{\frac{\pi}2}\frac1{n_{q}}\E\bigg((-\re \widehat X_{jk}^{(q)}\sin\varphi +\re \widehat Y_{jk}^{(q)}
\cos\varphi)\notag\\&\qquad\qquad\qquad\times(\im \widehat X_{jk}^{(q)}\cos\varphi
+\im \widehat Y_{jk}^{(q)}
\sin\varphi)\frac{\partial g_{jk}^{(q)}}
{\partial \im Z_{jk}^{(q)}}(0,0)\bigg) d\varphi,\notag
,\displaybreak[2]\\
\Sigma_4=&\frac1{2n}\sum_{q=1}^m\sum_{j=1}^{n_{q-1}}\sum_{k=1}^{n_{q}}
\int_0^{\frac{\pi}2}\frac1{n_{q}\sqrt{n_{q}}}\E\bigg((-\re \widehat X_{jk}^{(q)}\sin\varphi 
+\re \widehat Y_{jk}^{(q)}\cos\varphi)
\notag\\&\qquad\qquad\qquad\times(\re \widehat X_{jk}^{(q)}\cos\varphi+\re \widehat Y_{jk}^{(q)}
\sin\varphi)^2
(1-\theta_{jk}^{(q)})
\frac{\partial^2 g_{jk}^{(q)}}
{\partial {\re Z_{jk}^{(q)}}^2}(\theta_{jk}^{(q)})\bigg) d\varphi,\notag\\
\Sigma_5=&\frac1{2n}\sum_{q=1}^m\sum_{j=1}^{n_{q-1}}\sum_{k=1}^{n_{q}}
\int_0^{\frac{\pi}2}\frac1{n_{q}\sqrt{n_{q}}}\E\bigg((-\re \widehat X_{jk}^{(q)}\sin\varphi 
+\re \widehat Y_{jk}^{(q)}\cos\varphi)
\notag\\&\qquad\qquad\qquad\times(\im \widehat X_{jk}^{(q)}\cos\varphi+\im \widehat Y_{jk}^{(q)}
\sin\varphi)^2
(1-\theta_{jk}^{(q)})
\frac{\partial^2 g_{jk}^{(q)}}
{\partial {\im Z_{jk}^{(q)}}^2}(\theta_{jk}^{(q)})\bigg) d\varphi,\notag\\
\Sigma_6=&\frac1{n}\sum_{q=1}^m\sum_{j=1}^{n_{q-1}}\sum_{k=1}^{n_{q}}
\int_0^{\frac{\pi}2}\frac1{n_{q}\sqrt{n_{q}}}\E\bigg((-\re \widehat X_{jk}^{(q)}\sin\varphi 
+\re \widehat Y_{jk}^{(q)}\cos\varphi)
\notag\\&\times(\im \widehat X_{jk}^{(q)}\cos\varphi+\im \widehat Y_{jk}^{(q)}\sin\varphi)(\re \widehat X_{jk}^{(q)}
\cos\varphi+\re \widehat Y_{jk}^{(q)}\sin\varphi)
\notag\\&\qquad\qquad\qquad\qquad\qquad\qquad\times(1-\theta_{jk}^{(q)})
\frac{\partial^2 g_{jk}^{(q)}}
{\partial {\im Z_{jk}^{(q)}}\partial {\re Z_{jk}^{(q)}}}(\theta_{jk}^{(q)})\bigg) d\varphi,
\end{align}
and $\Sigma_{7},\hdots,\Sigma_{12}$ denote similar terms
coming from the second line in \eqref{mat0}.
Since $\Sigma_{7},\hdots,\Sigma_{12}$ can be treated
in the same way as $\Sigma_{1},\hdots,\Sigma_{6}$,
we provide the details for the latter only.

Since $g_{jk}^{(q)}(0,0)$ and $X_{jk}^{(q)}$, $Y_{jk}^{(q)}$ are independent,
it follows from \eqref{eq:momentsAfterTruncation-16} that 
\begin{equation}\label{mat7}
 |\Sigma_1|\le \frac{C A_0 (L_n(\tau_n) + \widetilde L_n(\tau_n))}{\tau_n}\le C\tau_n^3.
\end{equation}
Using again that the random variables $\frac{\partial g_{jk}^{(q)}}
{\partial Z_{jk}^{(q)}}(0,0)$ and $X_{jk}^{(q)}$, $Y_{jk}^{(q)}$ 
are independent, we get
\begin{align}\notag
 \E\bigg((-\re \widehat X_{jk}^{(q)}\sin\varphi +\re \widehat Y_{jk}^{(q)}\cos\varphi)(\re \widehat X_{jk}^{(q)}
 \cos\varphi+\re \widehat Y_{jk}^{(q)}
\sin\varphi)\frac{\partial g_{jk}^{(q)}}{\partial\re Z_{jk}^{(q)}}(0,0)\bigg)\quad\qquad\notag\\=
\E\bigg((-\re \widehat X_{jk}^{(q)}\sin\varphi +\re \widehat Y_{jk}^{(q)}\cos\varphi)
(\re \widehat X_{jk}^{(q)}\cos\varphi+\re \widehat Y_{jk}^{(q)}\sin\varphi)\bigg)
\E\bigg(\frac{\partial g_{jk}^{(q)}}{\partial\re  Z_{jk}^{(q)}}(0,0)\bigg).\notag
\end{align}
Since $\re \widehat X_{jk}^{(q)}$ and $\re \widehat Y_{jk}^{(q)}$
are independent, we have
\begin{align}\notag
 \E&(-\re \widehat X_{jk}^{(q)}\sin\varphi +\re \widehat Y_{jk}^{(q)}\cos\varphi)(\re \widehat X_{jk}^{(q)}
 \cos\varphi+\re \widehat Y_{jk}^{(q)}\sin\varphi)\notag\\
&=\E\re \widehat X_{jk}^{(q)}\E\re \widehat Y_{jk}^{(q)}(\cos^2\varphi-\sin^2\varphi)
  -(\E|\re \widehat X_{jk}^{(q)}|^2 - \E|\re \widehat Y_{jk}^{(q)}|^2) \cos\varphi \sin\varphi\notag
\end{align}
and therefore, by \eqref{eq:momentsAfterTruncation-11},
\begin{align}\notag
\Big|\E&(-\re \widehat X_{jk}^{(q)}\sin\varphi +\re \widehat Y_{jk}^{(q)}\cos\varphi)
(\re \widehat X_{jk}^{(q)}\cos\varphi+\re \widehat Y_{jk}^{(q)} \sin\varphi)\Big|\notag\\
&\le |\E\re \widehat X_{jk}^{(q)}|^2 + |\E\re \widehat Y_{jk}^{(q)}|^2+\Big|\E|\re \widehat X_{jk}^{(q)}|^2 - \E|\re \widehat Y_{jk}^{(q)}|^2\Big| \notag\\
&\le \frac1{n\tau_n^2}(\E|X_{jk}^{(q)}|^2\mathbb I\{|X_{jk}^{(q)}| \ge \tau_n\sqrt n\})^2+\frac1{n\tau_n^2}(\E|Y_{jk}^{(q)}|^2\mathbb I\{|Y_{jk}^{(q)}| \ge \tau_n\sqrt n\})^2\notag\\&\qquad\,\,+\Big|\E|\re \widehat X_{jk}^{(q)}|^2 - \E|\re \widehat Y_{jk}^{(q)}|^2\Big| \notag\\
&\le \frac1{n\tau_n^2} \E|X_{jk}^{(q)}|^2\mathbb I\{|X_{jk}^{(q)}| \ge \tau_n\sqrt n\}   +\frac1{n\tau_n^2} \E|Y_{jk}^{(q)}|^2\mathbb I\{|Y_{jk}^{(q)}| \ge \tau_n\sqrt n\}   \notag\\&\qquad\,\,+\Big|\E|\re \widehat X_{jk}^{(q)}|^2 - \E|\re \widehat Y_{jk}^{(q)}|^2\Big|. \notag
\end{align}
By \eqref{eq:momentsAfterTruncation-17},
the~last inequality implies that
\begin{equation}\label{mat5}
 |\Sigma_2|\le C A_1\Big(\frac{L_n(\tau_n)+\widetilde L_n(\tau_n)}{n\tau_n^2}+L_n(\tau_n)+\widetilde L_n(\tau_n)\Big)\le C\Big(\frac{\tau_n^2}{n}+\tau_n^4\Big).
\end{equation}
Similarly, using \eqref{eq:momentsAfterTruncation-11} and \eqref{eq:momentsAfterTruncation-19}, we~get
\begin{equation}\label{mat5a}
 |\Sigma_3|\le C A_1\Big(\frac{L_n(\tau_n)+\widetilde L_n(\tau_n)}{n\tau_n^2}+L_n(\tau_n)+\widetilde L_n(\tau_n)\Big)\le C\Big(\frac{\tau_n^2}{n}+\tau_n^4\Big).
\end{equation}
Also, note that 
\begin{align}\notag
 \frac1{\sqrt n}\bigg|\E\bigg((-\re \widehat X_{jk}^{(q)}\sin\varphi+\re \widehat Y_{jk}^{(q)}\cos\varphi)
&(\re \widehat X_{jk}^{(q)}\cos\varphi+\re \widehat Y_{jk}^{(q)}\sin\varphi)^2\notag\\&
\times(1-\theta_{jk}^{(q)})
\frac{\partial^2 g_{jk}^{(q)}}
{\partial {\re Z_{jk}^{(q)}}^2}(\theta_{jk}^{(q)})\bigg)\bigg|\notag\\
\le \frac{C}{\sqrt n}
\E \bigg( (|\widehat X_{jk}^{(q)}|^3+|\widehat Y_{jk}^{(q)}|^3) \bigg| &\E\Big\{\frac{\partial^2 g_{jk}^{(q)}}{\partial {\re Z_{jk}^{(q)}}^2}(\theta_{jk}^{(q)}) \Big| \widehat X_{jk}^{(q)},\widehat Y_{jk}^{(q)}\Big\} \bigg| \bigg)\notag\\
\le \frac {C A_2}{\sqrt n} \E \Big( |\widehat X_{jk}^{(q)}|^3+|\widehat Y_{jk}^{(q)}|^3 \Big)
&\le CA_2 \tau_n \,. \notag
\end{align}
It is simple to check now that
\begin{equation}\label{mat6}
|\Sigma_4| \le CA_2\tau_n. 
\end{equation}
Analogously we show that
\begin{equation}\label{mat6a}
 \max\{|\Sigma_5|,|\Sigma_6|\}\le CA_2\tau_n. 
\end{equation}
Combining the preceding estimates, we obtain \eqref{eq:convergenceOfExpectation-1}.
%
%
Thus, Theorem \ref{singularvalueuniversality} is proved.
\end{proof}

\section{Universality of Eigenvalue Distributions of Functions \\ of Independent Random Matrices}
\label{sec:universality-EVD}

We now turn to the eigenvalue distribution of functions of independent random matrices.
We use the assumptions and the notation from Section~\ref{sec:notation},
but throughout this section we~assume additionally that $n = p$, 
so that ${\F}_{\X}$ and ${\F}_{\Y}$ are square matrices.

Let $\mu$ a probability measure on the complex plane with compact support.
Define the logarithmic potential of the measure $\mu$ as 
\begin{equation}\notag
U_{\mu}(\alpha)=-\int_{\mathbb C}\log|\alpha-\zeta|d\mu(\zeta).
\end{equation}
Let $\mu_{\X}$ (resp. $\mu_{\Y}$) denote the empirical spectral 
measure of the matrix $\F_{\X}$
(resp. $\F_{\Y}$), \linebreak i.e. $\mu_{\X}$ 
(resp. $\mu_{\Y}$) is the uniform distribution on the eigenvalues 
$\{\lambda_1(\X),\ldots,\lambda_n(\X)\}$
(resp. $\{\lambda_1(\Y),\ldots\lambda_n(\Y)\}$) of the matrix 
$\F_{\X}$ (resp. $\F_{\Y}$).
Then
\begin{align}
U_{\X}(\alpha)&=-\int_{\mathbb C}\log|\alpha-\zeta|d\mu_{\X}(\zeta)=
-\frac1n\sum_{j=1}^n\log|\lambda_j(\X)-\alpha|,\notag\\
U_{\Y}(\alpha)&=-\int_{\mathbb C}\log|\alpha-\zeta|d\mu_{\Y}(\zeta)=
-\frac1n\sum_{j=1}^n\log|\lambda_j(\Y)-\alpha|.\notag
\end{align}
%
%
%

Let $\alpha \in \mathbb{C}$,
and let $s_1(\F_{\X}-\alpha\I)\ge \cdots\ge s_n(\F_{\X}-\alpha\I)$ 
and $s_1(\F_{\Y}-\alpha\I)\ge \cdots\ge s_n(\F_{\Y}-\alpha\I)$
denote the singular values of the matrices 
$\F_{\X}-\alpha\I$ and $\F_{\Y}-\alpha\I$, 
respectively.
Note that we have the representations
\begin{align}\label{eq:altRep}
U_{\X}(\alpha)=-\frac1n\sum_{j=1}^n\log s_j(\F_{\X}-\alpha \I),\quad
U_{\Y}(\alpha)=-\frac1n\sum_{j=1}^n\log s_j(\F_{\Y}-\alpha \I).
\end{align}

\begin{defn}
  If there exists some $p>0$ such that the quantity
  \begin{equation}\notag
   \frac1n\sum_{k=1}^ns_k^p(\F_{\X})
  \end{equation}
  is bounded in probability as $n \to \infty$, 
  we say that the matrices  $\F_{\X}$ satisfy condition $(C0)$.
\end{defn}

\begin{defn}
 If, for any fixed $\alpha \in \mathbb{C}$,
 there exists some $Q>0$ such that the relation
 \begin{equation}\notag
  \lim_{n\to\infty}\Pr\{s_n(\F_{\X} - \alpha \I)\le n^{-Q}\}=0
 \end{equation}
 holds, we say that the matrices  $\F_{\X}$ satisfy  condition $(C1)$.
\end{defn}

\begin{defn}
 If, for any fixed $\alpha \in \mathbb{C}$,
 there exists some $0<\gamma<1$ such that
 for \emph{any} sequence $\delta_n \to 0$,
 \begin{equation}\label{singval}
 \lim_{n\to\infty}\Pr\bigg\{\frac1n\sum_{{n_1}\le j\le{ n_2} }|
 \log s_j(\F_{\X}-\alpha\I)|>\varepsilon\bigg\}=0
\quad\text{for all $\varepsilon > 0$},
\end{equation}
 with $n_1=[n-n\delta_n]+1$, $n_2=[n-n^{\gamma}]$,
 we say that the matrices  $\F_{\X}$ satisfy  condition $(C2)$.
\end{defn}

We now prove the \emph{universality of eigenvalue distributions}.

\begin{thm}\label{eigenvalueuniversality}
Assume that the  matrices $\F_{\X}$ and $\F_{\Y}$ 
satisfy the conditions $(C0)$, $(C1)$ and $(C2)$.
Assume additionally that the conditions \eqref{lind} and \eqref{rank} hold
and that, for any fixed $\alpha \in \mathbb{C}$,
the conditions \eqref{mat3} -- \eqref{mat32a} of Theorem \ref{singularvalueuniversality} hold
with $\mathbf B = \alpha \I$.
Then the empirical distributions of the eigenvalues of the matrices 
$\F_{\X}$ and $\F_{\Y}$ have the same limit distribution in probability
in the sense that
\begin{equation}\notag
 \lim_{n\to\infty}\Pr\Big\{\Big|\int_{\mathbb C} fd\mu_{\X}
 -\int_{\mathbb C} fd\mu_{\Y}\Big|>\varepsilon\Big\}=0
\end{equation}
for any bounded continuous function $f$ and any $\varepsilon>0$.
\end{thm}

\begin{proof} 
The proof of Theorem \ref{eigenvalueuniversality} is based 
on the ``replacement principle'' by Tao and Vu (see \cite{TaoVu:10}, Theorem 2.1)
which builds upon a sort of inversion formula
for the logarithmic potential that goes back to Girko \cite{Girko:84}
and that was also investigated by Bai \cite{Bai},~\cite{BS:10}.

We prove that, for any fixed $\alpha \in \mathbb{C}$,
\begin{equation}\label{potential}
 \lim_{n\to\infty}\Big(\frac1n\log\Big|\det\{\F_{\X}-\alpha\I\}\Big|
 -\frac1n\log\Big|\det\{\F_{\Y}-\alpha\I\}\Big|\Big)=0
\quad\text{in probability.}
\end{equation}
This is equivalent to
\begin{equation}\notag
 \lim_{n\to\infty}(U_{\X}(\alpha)-U_{\Y}(\alpha))=0
\quad\text{in probability.}
\end{equation}
Note that by condition (C1) the determinants in \eqref{potential}
are not zero with probability $1-o(1)$.

\pagebreak[1]

Let $L(F,G)$ denote the L\'evy distance between two distribution functions $F$ and $G$. 
Recall that
\begin{equation}\notag
 L(F,G)=\inf\{\varepsilon>0: F(x-\varepsilon)-\varepsilon\le G(x)\le F(x+\varepsilon) +\varepsilon \text{ for all } x \in \Real \} \in [0,1].
\end{equation}
Let $\mathcal G_{\X}(x,\alpha)$ and $\mathcal G_{\Y}(x,\alpha)$ denote 
the distribution functions of the singular values of the matrices 
$\F_{\X}-\alpha\I$ and $\F_{\Y}-\alpha\I$, respectively.
Let $\varkappa_n=L(\mathcal G_{\X}(\cdot,\alpha),\mathcal G_{\Y}(\cdot,\alpha))$. 
According to Theorem \ref{singularvalueuniversality} (with $\mathbf B = \alpha \I$), we have
\begin{equation}\label{leviprob}
 \lim_{n\to\infty}\varkappa_n=0
\quad\text{in probability}.
\end{equation}
We introduce the events
\begin{equation}\notag
 \Omega_0=\{s_n(\F_{\X}-\alpha\I)\ge n^{-Q}\} \,\cap\, \{s_n(\F_{\Y}-\alpha\I)\ge n^{-Q}\},
\quad
 \Omega_1=\{\varkappa_n\le (\E\varkappa_n)^{\frac12}\}. 
\end{equation}
Note that by \eqref{leviprob} and Markov's inequality, we have
\begin{equation}\label{omega1}
 \Pr\{\Omega_1^{c}\}\le (\E\varkappa_n)^{\frac12} \to 0 \text{ as }n\to\infty.
\end{equation}
Furthermore, put $\eta_n=\max \{ (\E\varkappa_n)^{\frac13}, (\log n)^{-1} \}$,
and introduce the event
\begin{equation}\notag
 \Omega_2=\{\text{there exists } a\in[\eta_n,2\eta_n]:\, 
 \mathcal G_{\Y}(a+\varkappa_n,\alpha)-\mathcal G_{\Y}(a-\varkappa_n,\alpha)
 \le \varkappa_n^{\frac12}\}.
\end{equation}
It is straightforward to check that
\begin{align}\label{omega2}
    \Pr\{\Omega_2^c\}
\le \Pr\{\mathcal G_{\Y}(2\eta_n,\alpha)-\mathcal G_{\Y}(\eta_n,\alpha) \ge \varkappa_n^{\frac12}\Big[\frac{\eta_n}{2\varkappa_n}\Big]\} 
\le \Pr\{\frac{\eta_n}{2\varkappa_n^{\frac12}}\le 2\}
\le \frac{16\E\varkappa_n}{\eta_n^2}\le 16\eta_n.
\end{align}
Let $\delta_n := 1/|\log 2\eta_n| \to 0$,
let $n_1 = [n-n\delta_n]+1$ and $n_2 = [n-n^\gamma]$
be defined as in condition (C2), 
and introduce the event
$$
\Omega_3 = \{ s_{n_1}(\F_{\X}-\alpha \I) \geq 2\eta_n \} \cap \{ s_{n_1}(\F_{\Y}-\alpha \I) \geq 2\eta_n \} \,.
$$
Note that on the set $\{ s_{n_1}(\F_{\X}-\alpha \I) < 2\eta_n \}$,
we~have, for large enough $n$,
\begin{multline*}
    \frac1n \sum_{n_1 \leq j \leq n_2} |\log s_j(\F_{\X} - \alpha \I)|
\ge \frac1n (n_2 - n_1 + 1) |\log 2\eta_n|
\ge \frac1n (n\delta_n-n^\gamma-1) |\log 2\eta_n| \\
  = (\delta_n - n^{\gamma-1}-n^{-1}) |\log 2\eta_n|
\ge \tfrac12 \delta_n |\log 2\eta_n|
  = \tfrac12 \,,
\end{multline*}
so that
$$
\Pr\{ s_{n_1}(\F_{\X}-\alpha \I) < 2\eta_n \} 
\le
\Pr\{ \frac1n \sum_{n_1 \leq j \leq n_2} |\log s_j(\F_{\X} - \alpha \I)| \geq \tfrac12 \}
\to 
0
$$
by condition (C2).
Since a similar estimate holds with $\F_{\Y}$ instead of $\F_{\X}$,
it follows that
\begin{equation}\label{omega3}
\Pr\{\Omega_3^c\} 
\leq
\Pr\{ s_{n_1}(\F_{\X}-\alpha \I) < 2\eta_n \} 
+ 
\Pr\{ s_{n_1}(\F_{\Y}-\alpha \I) < 2\eta_n \} 
\to 0
\quad \text{as $n \to \infty$.}
\end{equation}

\pagebreak[2]

\noindent{}Furthermore, on the set $\Omega_3$ we have, for any $a \in [\eta_n,2\eta_n]$,
\begin{align}
  \Big|\frac1n\log&\Big|\det\{\F_{\X}-\alpha\I\}\Big|
  -\frac1n\log\Big|\det\{\F_{\Y}-\alpha\I\}\Big|\Big|\notag\\&\le 
  \frac1n\sum_{k=n_2}^n|\log s_k(\F_{\X}-\alpha\I)|+
\frac1n\sum_{k=n_2}^{n}|\log s_k(\F_{\Y}-\alpha\I)|\notag\\&+\frac1n\sum_{k=n_1}^{n_2-1}|
\log s_k(\F_{\X}-\alpha\I)|+\frac1n\sum_{k=n_1}^{n_2-1}|\log s_k(\F_{\Y}-\alpha\I)|\notag\\&
+ \Big|\int_{a}^{a^{-1}} \log x\; d(\mathcal G_{\X}(x,\alpha)-\mathcal G_{\Y}(x,\alpha)) \Big|\notag\\&
+ \int_{(2\eta_n)^{-1}}^{\infty} |\log x| \;d(\mathcal G_{\X}(x,\alpha)+\mathcal G_{\Y}(x,\alpha)).\notag
\end{align}
Now fix $\varepsilon > 0$. The preceding inequality implies that
\begin{align}\label{big}
 \Pr\Big\{\Big|\frac1n\log&\Big|\det\{\F_{\X}-\alpha\I\}\Big|
 -\frac1n\log\Big|\det\{\F_{\Y}-\alpha\I\}\Big|\Big|\ge\varepsilon\Big\}
 \notag\\&\le\Pr\{\Omega_0^c\}
 +\Pr\{\Omega_1^c\}+\Pr\{\Omega_2^c\}+\Pr\{\Omega_3^c\}\notag\\&+\Pr\Big\{\frac1n\sum_{k=n_2}^n|
 \log s_k(\F_{\X}-\alpha\I)|+
\frac1n\sum_{k=n_2}^{n}|\log s_k(\F_{\Y}-\alpha\I)|
\ge\varepsilon/4;\Omega_0\Big\}\notag\\&+
\Pr\Big\{\frac1n\sum_{k=n_1}^{n_2-1}|\log s_k(\F_{\X}-\alpha\I)|+\frac1n\sum_{k=n_1}^{n_2-1}|
\log s_k(\F_{\Y}-\alpha\I)|\ge
\varepsilon/4\Big\}\notag\\&+\Pr\Big\{\Big|\int_{a}^{a^{-1}}\log x \;d
(\mathcal G_{\X}(x,\alpha)-\mathcal G_{\Y}(x,\alpha))
\Big|\ge\varepsilon/4;\Omega_2\cap\Omega_1\Big\}
\notag\\
&+\Pr\Big\{\int_{(2\eta_n)^{-1}}^{\infty} |\log x| \; d(\mathcal G_{\X}(x,\alpha)+\mathcal G_{\Y}(x,\alpha))
\ge\varepsilon/4\Big\}.
\end{align}
By condition $(C1)$ and inequalities \eqref{omega1} -- \eqref{omega3},
\begin{equation}\label{vip1}
 \Pr\{\Omega_0^c\}
 +\Pr\{\Omega_1^c\}+\Pr\{\Omega_2^c\}+\Pr\{\Omega_3^c\}\to 0 \quad\text{as}\quad n\to\infty.
\end{equation}
By definition of $\Omega_0$ and by condition $(C2)$, we have, for $n\to\infty$,
\begin{align}
\Pr&\{\frac1n\sum_{k=n_2}^n|\log s_k(\F_{\X}-\alpha\I)|+
\frac1n\sum_{k=n_2}^{n}|\log s_k(\F_{\Y}-\alpha\I)|\ge\varepsilon/4;\Omega_0\}
\notag\\&+\Pr\{\frac1n\sum_{k=n_1}^{n_2}|\log s_k(\F_{\X}-\alpha\I)|
+\frac1n\sum_{k=n_1}^{n_2}|\log s_k(\F_{\Y}-\alpha\I)|\ge
\varepsilon/4\}\to0.
\end{align}
Moreover, using that for any $p>0$,
the function $x^{-p}\log x$ is decreasing in $x$
for $x\ge \text{\rm e}^{1/p}$, 
we get, for~large enough $n$,
\begin{multline*}
    \int_{(2\eta_n)^{-1}}^{\infty} |\log x| \; d(\mathcal G_{\X}(x,\alpha)+\mathcal G_{\Y}(x,\alpha))
\le C\eta_n^{p}|\log\eta_n|\int_{0}^{\infty}x^{p}d(\mathcal G_{\X}(x,\alpha)+\mathcal G_{\Y}(x,\alpha))\notag\\
\le C\eta_n^p|\log\eta_n|\bigg(\frac1n\sum_{k=1}^ns_k^p(\F_\X-\alpha\I)+\frac1n\sum_{k=1}^ns_k^p(\F_\Y-\alpha\I)\bigg).\notag
\end{multline*}
By the inequality $s_k(\F-\alpha\I) \le s_k(\F)+|\alpha|$ and condition $(C0)$, 
this quantity converges to~zero in probability.

Thus, it remains  to bound the last but one summand in \eqref{big}.
Recall that $a \in [\eta_n,2\eta_n]$.
Integrating by parts, we~have
\begin{align}\label{integr1}
 \Big|\int_{a}^{a^{-1}}\log xd(\mathcal G_{\X}(x,\alpha)-\mathcal G_{\Y}(x,\alpha))\Big|
&\le 
 |\log a||\mathcal G_{\X}(a,\alpha)-\mathcal G_{\Y}(a,\alpha)|\notag
\\&\qquad+
 |\log a||\mathcal G_{\X}(a^{-1},\alpha)-\mathcal G_{\Y}(a^{-1},\alpha)|\notag
\\& \qquad+ \bigg|\int_{a}^{a^{-1}}\frac{\mathcal G_{\X}(x,\alpha)-\mathcal G_{\Y}(x,\alpha)}{x}dx\bigg|. 
\end{align}
Recall that we need to bound this expression for $\omega \in \Omega_2 \cap \Omega_1$ only
and that $\varkappa_n \leq \eta_n$ for~such~$\omega$.
By Chebyshev's inequality, we have
\begin{equation}\notag
 \max\{1-\mathcal G_{\X}(M,\alpha),1-\mathcal G_{\Y}(M,\alpha)\}
 \le \frac1{M^p}\bigg(\frac1n\sum_{k=1}^ns_k^p(\F_\X-\alpha\I)+
\frac1n\sum_{k=1}^ns_k^p(\F_\Y-\alpha\I)\bigg)
\end{equation}
for any $M > 0$.
It therefore follows from condition $(C0)$ that 
the second term on the r.h.s. of \eqref{integr1} 
converges to zero in probability. 
Furthermore, by the definition of $\varkappa_n$,
we have the following bound 
\begin{equation}\label{levi2}
 |\mathcal G_{\X}(x,\alpha)-\mathcal G_{\Y}(x,\alpha)|\le\varkappa_n
 + \mathcal G_{\Y}(x+\varkappa_n,\alpha)-\mathcal G_{\Y}(x-\varkappa_n,\alpha),
\quad{x \in \mathbb R.}
\end{equation}
Thus, for~$\omega\in\Omega_2 \cap \Omega_1$ we may find $a\in [\eta_n,2\eta_n]$ 
such that
\begin{equation}\notag
    |\mathcal G_{\X}(a,\alpha)-\mathcal G_{\Y}(a,\alpha)|\notag   
\le \varkappa_n + \mathcal G_{\Y}(a+\varkappa_n,\alpha)-\mathcal G_{\Y}(a-\varkappa_n,\alpha)
\le \varkappa_n + \varkappa_n^{\frac12}
\le 2\varkappa_n^{\frac12}
\le 2\eta_n^{\frac12} \,.
\end{equation}
Because $\eta_n^{\frac12} |\log \eta_n| \to 0$ as $n\to\infty$,
it follows that the first term on the r.h.s of \eqref{integr1} 
converges to zero in probability.
Finally, using inequality \eqref{levi2} again, we obtain
\begin{multline*}
    \bigg|\int_{a}^{a^{-1}}\frac{\mathcal G_{\X}(x,\alpha)-\mathcal G_{\Y}(x,\alpha)}{x}dx\bigg|
\le \int_{\eta_n}^{\eta_n^{-1}} \frac{|\mathcal G_{\X}(x,\alpha)-\mathcal G_{\Y}(x,\alpha)|}{x}dx \\
\le C\varkappa_n|\log\eta_n|+\int_{\eta_n}^{\eta_n^{-1}}
 \frac{\mathcal G_{\Y}(x+\varkappa_n,\alpha)-\mathcal G_{\Y}(x-\varkappa_n,\alpha)}{x}dx.
\end{multline*}
It is straightforward to check that for any $0<\varepsilon<a<b$ and any distribution 
function $F(x)$ the following inequality
\begin{equation}\notag
\int_a^b \frac{F(x+\varepsilon)-F(x-\varepsilon)}{x} \, dx
\le 2\varepsilon(\frac1{a-\varepsilon}+\frac1{b-\varepsilon})
\le \frac{4\varepsilon}{a-\varepsilon}
\end{equation}
holds.
Applying this inequality, we get
\begin{equation}\notag
\int_{\eta_n}^{\eta_n^{-1}}\frac{\mathcal G_{\Y}(x+\varkappa_n,\alpha)-\mathcal G_{\Y}(x-\varkappa_n,\alpha)}{x}dx 
\le \frac{C\varkappa_n}{\eta_n-\varkappa_n}.
\end{equation}
Therefore, for $\omega\in\Omega_2\cap\Omega_1$ we obtain, for large enough $n$,
\begin{equation}\notag
\bigg|\int_{a}^{a^{-1}}\frac{\mathcal G_{\X}(x,\alpha)-\mathcal G_{\Y}(x,\alpha)}{x}dx\bigg|
\le C \left( \varkappa_n |\log \eta_n| + \frac{\varkappa_n}{\eta_n - \varkappa_n} \right)
\le C\eta_n^{1/2}.
\end{equation}
This implies that
\begin{equation}\notag
 \lim_{n\to\infty}\Pr\Big\{\Big|\int_{a}^{a^{-1}}\log x \; d(\mathcal G_{\X}(x,\alpha)-\mathcal G_{\Y}(x,\alpha))\Big|\ge\varepsilon/4;\Omega_2\cap\Omega_1\Big\}=0.
\end{equation}
Thus relation \eqref{potential} is proved. 

\pagebreak[1]

We may now apply the ``replacement principle'' by Tao~and~Vu;
see \cite{TaoVu:10}, Theorem 2.1.
Note that this theorem is based on two assumptions (i) and (ii).
Assumption (ii) is just a reformulation of relation \eqref{potential}.
Assumption (i) is only needed to show that the probability measures $\mu_\X$ and $\mu_\Y$
are tight in probability (see Equations (3.3) and (3.4) in \cite{TaoVu:10}),
and may be replaced with our assumption ($C0$).
It~therefore follows that $\mu_{\X}-\mu_{\Y}$ converges weakly to~zero in probability,
i.e.\@ for any bounded continuous function $f$ and any $\varepsilon>0$, we have
\begin{equation}\notag
 \lim_{n\to\infty}\Pr\bigg\{\bigg|\int_{\mathbb C} fd\mu_{\X}
 -\int_{\mathbb C} fd\mu_{\Y}\bigg|>\varepsilon\bigg\}=0.
\end{equation}
Thus Theorem \ref{eigenvalueuniversality} is proved.
\end{proof}


\begin{rem}\label{spher0}
It follows from the proof of Theorem \ref{eigenvalueuniversality} 
that it suffices to assume, instead of the conditions \eqref{lind}, \eqref{rank}, 
and \eqref{mat3} -- \eqref{mat32a} of Theorem \ref{singularvalueuniversality},
that for any $\alpha\in\mathbb C$,
\begin{equation}\label{newcondition}
  \lim_{n\to\infty}L(\mathcal G_{\X}(\cdot,\alpha),\mathcal G_{\Y}(\cdot,\alpha))=0
  \quad\text{in probability.}
\end{equation}
\end{rem}

\section{Asymptotic Freeness of Random Matrices}\label{freeprob}
In this section we consider the asymptotic freeness of random matrices with special structure.
Before that, we recall the definition of Voiculescu's asymptotic freeness as well as some basic notions
from free probability theory. See also the survey by Speicher \cite{Speicher}
and the lecture notes by Voiculescu \cite{Voiculescu:98}.
\begin{defn}
A pair $(\mathcal A,\varphi)$ consisting of a unital algebra $\mathcal A$ 
and a linear functional $\varphi:\,\mathcal A\rightarrow \mathbb C$ 
with $\varphi(1)=1$ is called a \emph{non-commutative probability space}.
Elements of $\mathcal A$ are called \emph{random variables}, the numbers 
$\varphi(a_{i(1)}\cdots a_{i(n)})$ for such random variables 
$a_1,\ldots,a_k\in\mathcal A$ are called \emph{moments}, and the collection of 
all moments is called the \emph{joint distribution} of $a_1,\ldots,a_k$.
Equivalently, we may say that the joint distribution of $a_1,\hdots,a_k$
is given by the linear functional 
$\mu_{a_1,\hdots,a_k} : \mathbb C \langle X_1,\hdots,X_k \rangle \to \mathbb{C}$
with 
$\mu_{a_1,\hdots,a_k}(P(X_1,\hdots,X_k)) = \varphi(P(a_1,\hdots,a_k))$,
where $\mathbb C \langle X_1,\hdots,X_k \rangle$ denotes the algebra 
of all polynomials in $k$ non-commuting indeterminates $X_1,\hdots,X_k$.

If, for a given element $a \in \mathcal{A}$, there exists 
a unique probability measure $\mu_a$ on $\mathbb R$ such~that 
$\int t^k \, d\mu_a(t) = \varphi(a^k)$ for all $k \in \mathbb N$,
we identify the distribution of $a$ with the~probability measure $\mu_a$.
\end{defn}

\begin{defn}
Let $(\mathcal A,\varphi)$ be a non-commutative probability space.

$1)$ 
Let $(\mathcal A_i)_{i \in I}$ be a family of unital sub-algebras of $\mathcal A$.
The sub-algebras $\mathcal A_i$ are called \emph{free} or \emph{freely independent},
if, for any positive integer $k$,
$\varphi(a_{1}\cdots a_k)=0$ 
whenever the following set of conditions holds:
$a_j\in\mathcal A_{i(j)}$ (with $i(j)\in I$) for all $j=1,\ldots,k$, 
$\varphi(a_j)=0$ for all $j=1,\ldots,k$,
and neighbouring elements are from taken different sub-algebras, 
i.e.\@ $i(1)\ne i(2),\,i(2)\ne i(3),\,\ldots,\, i(k-1)\ne i(k)$.

2)
Let $(\mathcal A_i')_{i \in I}$ be a family of subsets of $\mathcal A$.
The subsets $\mathcal A_i'$ are called \emph{free} or \emph{freely independent},
if their generated unital sub-algebras are free, 
i.e. if $(\mathcal A_i)_{i\in I}$ are free, where, 
for each $i\in I$, $\mathcal A_i$ is the smallest unital sub-algebra of $\mathcal A$ 
which contains $\mathcal A_i'$.

3) 
Let $(a_i)_{i \in I}$ be a family of elements from $\mathcal A$.
The elements $a_i$ are called \emph{free} or \emph{freely independent},
if the subsets $\{ a_i \}$ are free.
\end{defn}

Consider two random variables $a$ and $b$ which are free. 
Then the distributions of $a+b$ and $ab$ 
(in the sense of linear functionals)
depend only on the distributions of $a$ and $b$
(see~e.g.\@ \cite{VoiculescuNica}, Chapter 2),
and we can make the following definition:

\begin{defn} 
For free random variables $a$ and $b$,
the distributions of $a+b$ and $ab$ are called 
the \emph{free additive convolution} and the \emph{free multiplicative convolution}
of $\mu_a$ and $\mu_b$ and are denoted by
\begin{equation}\notag
\mu_{a+b}=\mu_a \boxplus \mu_b
\quad\text{and}\quad
\mu_{ab}=\mu_a\boxtimes \mu_b
\end{equation}
respectively.
\end{defn}

It can be shown (see e.g.\@ \cite{VoiculescuNica}, Chapter~3)
that this defines commutative and associative operations.
Furthermore, it is well known 
that if $\mu_a$ and $\mu_b$ are compactly supported 
probability measures on $\mathbb R$, then $\mu_a \boxplus \mu_b$ 
is also a compactly supported probability measure on $\mathbb R$.
Similarly, if $\mu_a$ and $\mu_b$ are compactly supported 
probability measure on $\mathbb R_{+}$, then $\mu_a \boxtimes \mu_b$ 
is also a compactly supported probability measure on $\mathbb R_{+}$.

\pagebreak[2]

We shall now define the {\it $R$-transform} and the {\it $S$-transform} 
of a random variable \linebreak $a \in \mathcal A$ 
as~introduced by Voiculescu (see e.g.\@ \cite{Voiculescu:98}, p. 296 and p. 310).
Let $a \in \mathcal A$ be \linebreak a random variable with distribution $\mu_a$ on $\Real$.
Let $M_a(z)$ denote the generic formal moment generating function of $\mu_a$, 
namely $M_a(z)=\sum_{k=1}^{\infty} m_k z^k$, where $m_k=\int_{-\infty}^{\infty}x^k \, d\mu_a(x)=\varphi(a^k)$. 
Let $G_a(z) := \frac{1}{z} (1 + M_a(\tfrac{1}{z}))$,
and define the {\it $R$-transform} of $a$ by
\begin{equation}\label{R-transform}
R_{a}(z) := G_a^{-1}(z) - \frac{1}{z} \,,
\end{equation}
where $G_a^{-1}(z)$ denotes the inverse of $G_a(z)$ w.r.t.\@ composition of functions.
Moreover, when~$\varphi(a) \ne 0$,
define the {\it $S$-transform} of $a$ by
\begin{equation}\label{S-transform}
S_{a}(z) := \frac{z+1}{z} M_a^{-1}(z) \,,
\end{equation}
where $M_a^{-1}(z)$ denotes the inverse of $M_a(z)$ w.r.t.\@ composition of functions.
These defini\-tions have to be understood at the level of formal series. 
However, when $a$ has a distribution on $\Real$ with compact support,
$R_a(z)$ and $S_a(z)$ may be regarded as analytic functions
in a certain neighborhood of the origin.

Then, for free random variables $a$~and~$b$,
we have
\begin{equation}\label{freesum}
R_{a+b}(z)=R_{a}(z)+R_{b}(z)
\end{equation}
and, when $\phi(a)\ne 0$ and $\phi(b) \ne 0$, 
\begin{equation}\label{freeproduct}
S_{ab}(z)=S_{a}(z) S_{b}(z);
\end{equation}
see e.g.\@ \cite{Voiculescu:98}, p. 296 and p. 310.

For a generalization of the $S$-transform to the case 
where $\varphi(a)=0$ and $\varphi(a^2) \ne 0$,
see Rao and Speicher \cite{SpeicherRao:2007}
as well as Arizmendi and P\'erez-Abreu \cite{Perez}.
Here the $S$-transform is a~formal power series in $\sqrt{z}$,
and there exist two branches of the $S$-transform.
We will use this generalization only in the case where 
$a$ has a \emph{symmetric} distribution $\mu_a \ne \delta_0$ on $\Real$,
in which we adopt the approach by Arizmendi and P\'erez-Abreu \cite{Perez};
see definition \eqref{eq:ST-symmetric} below.

Let $a$ have a compactly supported distribution $\mu_a \ne \delta_0$ on $\Real$,
and define the function $\widetilde R_a(z):=z R_a(z)$. 
Then
\begin{equation}\label{nica}
z S_a(z)={\widetilde R}_a^{-1}(z),
\end{equation}
where ${\widetilde R}_a^{-1}(z)$ denotes the inverse
of ${\widetilde R}_a(z)$ w.r.t.\@ composition of functions.
See for instance Nica \cite{Nica}, Equation (21) in Chapter 13.
For clarity, let us emphasize that we call $\widetilde{R}_a(z)$
what is called $R_a(z)$ in \cite{Nica}.
Furthermore, let us note that the argument in \cite{Nica}
requires that $\varphi(a) \ne 0$.
However, when $\varphi(a) = 0$ and $\varphi(a^2) \ne 0$,
one can use similar arguments as in Rao and Speicher \cite{SpeicherRao:2007}
to show that, similarly as for the $S$-transform, 
there exist two~branches of ${\widetilde R}_a^{-1}(z)$
and that \eqref{nica} continues to hold 
with an appropriate choice of these branches.
We shall always take the branches such that
\begin{align}
\label{eq:ChoiceOfBranch}
\text{$\im {\widetilde R}_a^{-1}(z) \le 0$ \ and \ $\im S_a(z) \ge 0$ \ for \ $z \approx 0$, $z \not\in \mathbb R_{+}$.}
\end{align}

\pagebreak[2]

\noindent{}For example, when the distribution of $a$ is the semi-circular law, we take
$\widetilde R_a^{-1}(z)=-\sqrt z$ and $S_a(z)=-\frac1{\sqrt z}$,
where the branch of $\sqrt z$ is chosen such that 
$\im\sqrt z\ge 0$ for $z\in\mathbb C \setminus \mathbb R_{+}$.

\begin{rem}
The $R$-transform and the $S$-transform (as well as the other transforms)
can also be introduced for a compactly supported probability measure $\mu$ 
on $\mathbb R$ and will then be denoted by $R_\mu$ and $S_\mu$, respectively.
For the $S$-transform, we require that $\mu \ne \delta_0$.
\end{rem}

We will also need the additive and multiplicative convolution 
for probability measures with \emph{unbounded} support;
see Bercovici and Voiculescu \cite{BercVoic} for details.
For our purposes, the following definitions are convenient:
For any probability measures $\mu,\nu$ on $\Real$,
we set
\begin{align}
\label{eq:boxplus-1}
\mu \boxplus \nu := \lim_{n \to \infty} (\mu_n \boxplus \nu_n) \,,
\end{align}
where $(\mu_n),(\nu_n)$ are sequences of probability measures on $\Real$
with bounded support which converge to $\mu,\nu$ w.r.t.\@ L\'evy distance $d_L$.
Similarly, for any probability measures $\mu,\nu$ on~$\Real_+$,
we set
\begin{align}
\label{eq:boxtimes-1}
\mu \boxtimes \nu := \lim_{n \to \infty} (\mu_n \boxtimes \nu_n) \,,
\end{align}
where $(\mu_n),(\nu_n)$ are sequences of probability measures on $\Real_+$
with bounded support which converge to $\mu,\nu$ w.r.t.\@ Kolmogorov distance $d_K$.
It can be shown that these operations are well-defined and continuous
in the sense that
\begin{align}
\label{eq:boxplus-2}
d_L(\mu_1 \boxplus \nu_1,\mu_2 \boxplus \nu_2) \le d_L(\mu_1,\mu_2) + d_L(\nu_1,\nu_2)
\end{align}
and
\begin{align}
\label{eq:boxtimes-2}
d_K(\mu_1 \boxtimes \nu_1,\mu_2 \boxtimes \nu_2) \le d_K(\mu_1,\mu_2) + d_K(\nu_1,\nu_2) \,,
\end{align}
cf. Propositions 4.13 and 4.14 in \cite{BercVoic}.

The above transforms also have extensions to unbounded probability measures.
Let us provide the details for the $S$-transform;
cf. Section~6 in \cite{BercVoic}.
For a probability measure $\nu$ on $(0,\infty)$, define the function
$$
\psi_\nu(z) := \int_0^\infty  \frac{tz}{1-tz} \, d\nu(t) \,.
$$
By Proposition 6.2 in \cite{BercVoic}, this function is univalent in the left half-plane $i \mathbb{C}^+$,
with $\psi_\nu(i \mathbb{C}^+) \cap \Real = (-1,0)$. 
The $S$-transform of $\nu$ is the function on $\psi_\nu(i \mathbb{C}^+)$ defined by
\begin{align}
\label{eq:ST-unbounded}
S_\nu = \frac{z+1}{z} \psi_\nu^{-1}(z) \,,
\end{align}
where $\psi_\nu^{-1}$ denotes the inverse of $\psi_\nu$.

\pagebreak[2]

Finally, we also need the $S$-transform for \emph{symmetric} probability measures on $\Real \setminus \{ 0 \}$.
Here we follow Arizmendi and P\'erez-Abreu \cite{Perez}.
Note that the function
\begin{align}
\label{eq:Q-definition}
\begin{array}{ccccc}
\mathcal Q &:& \Big\{ \text{symmetric p.m.'s on $\Real \setminus \{ 0 \}$} \Big\} &\to& \Big\{ \text{p.m.'s on $(0,\infty)$} \Big\} \\[+5pt]
&& \mu &\mapsto& \text{induced measure of $\mu$} \\ &&&& \text{under the mapping $x \mapsto x^2$}
\end{array}
\end{align}
is one--to--one. 
For a symmetric probability measure $\mu$ on $\Real \setminus \{ 0 \}$,
we define the $S$-transform by
\begin{align}
\label{eq:ST-symmetric}
S_\mu(z) = \sqrt{\frac{z+1}{z} \, S_{\mathcal Q(\mu)}(z)} \,,  
\end{align}
where the branch of the square-root is such that $\im S_\mu(z) \ge 0$ for $z \in (-1,0)$.
Of~course, for~probability measures with compact support (and the respective additional properties),
the new definitions \eqref{eq:ST-unbounded} and \eqref{eq:ST-symmetric} are consistent with the previous ones.
Also, let us mention that the assumption $\mu(0) = 0$ is more restrictive than necessary
(it~could be relaxed to $\mu(0) < 1$), but sufficient for our purposes.

Let $\mu,\mu_1,\mu_2,\mu_3,\dots$ be a family either of probability measures on $(0,\infty)$
or of symmetric probability measures on $\Real \setminus \{ 0 \}$. Then we have the following results:
\begin{align}
\label{ST-uniqueness} &\text{The $S$-transform $S_\mu$ determines the measure $\mu$.} \\
\label{ST-continuity} &\text{If $\mu_n \Rightarrow \mu$ then $S_{\mu_n} \to S_{\mu}$ locally uniformly on $(-1,0)$.}
\end{align}

\medskip

Next, we~recall Voiculescu's definition of asymptotic freeness for random matrices.
Note that for any (classical) probability space $(\Omega,\mathcal F,\Pr)$,
the set $\mathcal A := \mathcal M_{n \times n}(L^{\infty-}(\Omega))$
\linebreak[1] of $n \times n$ matrices with entries in $L^{\infty-}(\Omega)$
endowed with the functional $\varphi(\mathbf A) := \E\tfrac1n\Tr(\mathbf A)$
is a non-commutative probability space.

\begin{defn}\label{def:asymptoticfreeness} \
%
%
%
%
Let $I$ be an index set, and let $I = I_1 \cup \cdots \cup I_l$ be a partition of $I$.
For each $n \in \mathbb N$, let $(\mathbf A_{n,i})_{i \in I}$ a family
of random~matrices in $\mathcal{M}_{n \times n}(L^{\infty-}(\Omega))$.
The~families $\{ \mathbf A_{n,i} : i \in I_j \}$, $j=1,\hdots,l$, are called \emph{asymptotically free}
if there exists a family $(a_i)_{i \in I}$ of non-commutative random variables
in some non-commutative probability space $(\mathcal A, \varphi)$ such~that
for all $k \in \mathbb N$ and all $i(1),\hdots,i(k) \in I$,
$$
\lim_{n \to \infty} \frac1n\E\Tr (\mathbf A_{n,i(1)} \cdots \mathbf A_{n,i(k)}) = \varphi(a_{i(1)} \cdots a_{i(k)})
$$
and the families $\{ a_i : i \in I_j \}$, $j=1,\hdots,l$, are free.
\end{defn}


\begin{rem} \label{rem:application-RMT}
Let $(\mathbf A_n)_{n \in \mathbb N}$ and $(\mathbf B_n)_{n \in \mathbb N}$ be sequences 
of self-adjoint random matrices 
with $\mathbf A_n,\mathbf B_n$ $\in \mathcal{M}_{n \times n}(L^{\infty-}(\Omega))$
such~that 
the mean eigenvalue distributions of $\mathbf A_n$ and $\mathbf B_n$ 
converge in moments to compactly supported p.m.'s $\mu_{\mathbf A}$ and $\mu_{\mathbf B}$,
respectively, 
and $(\mathbf A_n)_{n \in \mathbb N}$ and $(\mathbf B_n)_{n \in \mathbb N}$ are asymptotically free.
Then the mean eigenvalue distribution of $\mathbf A_n + \mathbf B_n$ converges in moments
to the p.m.\@ $\mu_{\mathbf A} \boxplus \mu_{\mathbf B}$.
Moreover, if $\mathbf A_n$~and~$\mathbf B_n$ are additionally positive semi-definite,
the mean eigenvalue distribution of $\mathbf A_n \mathbf B_n$ converges in moments
to the p.m.\@ $\mu_{\mathbf A} \boxtimes \mu_{\mathbf B}$.
\end{rem}

\begin{rem}
It is easy to see that 
two sequences $(\mathbf A_n)_{n\in\mathbb N}$ and $(\mathbf B_n)_{n\in\mathbb N}$
with $\mathbf A_n,\mathbf B_n$ $\in \mathcal{M}_{n \times n}(L^{\infty-}(\Omega))$
are \emph{asymptotically free} if and only if for all $k \ge 1$ and all $j_1,l_1,\ldots,j_k,l_k$ $\ge 1$,
the~following relation holds, assuming that all limits involved exist,
\begin{align}\label{asfree}
 \lim_{n\to\infty}\E\frac1n\Tr\Big(
 \Big(\mathbf A_n^{j_1}-(\lim_{\nu\to\infty} \frac1\nu\E\Tr \mathbf A_\nu^{j_1})\I_n\Big)
 \Big(\mathbf B_n^{l_1}-(\lim_{\nu\to\infty} \frac1\nu\E\Tr \mathbf B_\nu^{l_1})\I_n\Big)\cdots&
\notag\\ 
 \Big(\mathbf A_n^{j_k}-(\lim_{\nu\to\infty} \frac1\nu\E\Tr \mathbf A_\nu^{j_k})\I_n\Big)
 \Big(\mathbf B_n^{l_k}-(\lim_{\nu\to\infty} \frac1\nu\E\Tr \mathbf B_\nu^{l_k})\I_n\Big)\Big)&=0.
\end{align}
Here $\I_n$ denotes the identity matrix of dimension $n \times n$.
\end{rem}

\medskip

Consider a sequence of random matrices $(\Y_n)_{n\in\mathbb N}$
with $\Y_n \in \mathcal M_{n \times n}(L^{\infty-}(\Omega))$, $n \in \mathbb N$.
Suppose that the matrix $\Y_n$ is \emph{bi-unitary invariant},
i.\,e.\@ the joint distribution of the entries of $\Y_n$ is equal 
to that of the entries of $\V_1\Y_n\V_2$, 
for any unitary $n \times n$ matrices $\V_1$~and~$\V_2$.
(For instance, this is the case if $\Y_n$ is a (non-self-adjoint) matrix
with independent standard complex Gaussian entries.) 
Introduce the matrices 
$$
\mathbf A_n=\begin{bmatrix}&\mathbf O&\Y_n&\\&{\Y}_n^*&\mathbf O&\end{bmatrix}
$$
and, for $\alpha=u+iv$ with real $u,v$,
$$
\mathbf B_n=\J_n(\alpha)=\begin{bmatrix}&\mathbf O&-\alpha\I_n&\\&-\overline \alpha\I_n&\mathbf O&\end{bmatrix}.
$$
We shall investigate the asymptotic freeness of the matrices
$(\mathbf A_n)_{n\in\mathbb N}$ and $(\mathbf B_n)_{n\in\mathbb N}$.

\begin{prop}
\label{prop:asympfree}
Let $(\mathbf A_n)_{n\in\mathbb N}$ and $(\mathbf B_n)_{n\in\mathbb N}$ 
be sequences of matrices as above, where for each $n \in \mathbb{N}$, 
$\Y_n$ is a bi-unitary invariant matrix, and $\alpha$ is a fixed complex number.
Moreover, suppose that the eigenvalue distribution of $\mathbf A_n$
converges weakly in probability (as $n \to \infty$)
to a compactly supported probability measure $\mu$ on $\mathbb R$
and that for any $k \in \mathbb N$, $\sup_{n \in \mathbb N} \tfrac1n \E \Tr A_n^{2k} < \infty$.
Then the sequences $(\mathbf A_n)_{n\in\mathbb N}$ and $(\mathbf B_n)_{n\in\mathbb N}$
are asymptotically~free.
\end{prop}

\begin{proof}
We check Equation \eqref{asfree}. Note that
\begin{equation}\label{jz}
 \mathbf B_n^l=\J_n^l(\alpha)=\begin{cases}|\alpha|^{2p}\,\I_{2n},&\text{ if } 
 l=2p,\\|\alpha|^{2p}\,\J_n(\alpha),&\text{ if }l=2p+1.\end{cases}
\end{equation}
From here it follows immediately that
$$
\J_n^{2p}(\alpha)-(\lim_{\nu \to \infty}\frac1{2\nu}\E\Tr\J_\nu^{2p}(\alpha))\I_{2n}=\mathbf O.
$$
This means that relation (\ref{asfree}) holds 
if at least one of the $l_1,l_2,\ldots, l_k$ is even.
Hence, suppose that $l_1,\ldots,l_k$ are all odd.
In this case we~may reduce relation (\ref{asfree}) to
\begin{align}\label{freeas1}
 \lim_{n\to\infty}\E\frac1n\Tr\Big( \widetilde{\mathbf A_n^{j_1}}\J(\alpha)\cdots\widetilde{\mathbf A_n^{j_k}}\J(\alpha)\Big)=0,
\end{align}
where
$$
\widetilde{\mathbf A_n^j} = \mathbf A_n^j - (\lim_{\nu\to\infty} \frac1{2\nu}\E\Tr \mathbf A_\nu^j)\I_{2n}.
$$

\pagebreak[2]

Note that
\begin{equation}\notag
 \mathbf A_n^{j}=\begin{cases}
 	\begin{bmatrix}&(\Y_n\Y_n^*)^p&\mathbf O&\\&\mathbf O&(\Y_n^*\Y_n)^p&\end{bmatrix},&\text{ if }j=2p,\\ \\
	\begin{bmatrix}&\mathbf O&(\Y_n\Y_n^*)^p\Y_n&\\&(\Y_n^*\Y_n)^p\Y_n^*&\mathbf O&\end{bmatrix},&\text{ if }j=2p+1.
 \end{cases}  
\end{equation}

In order to complete the proof, we proceed similarly as in Hiai and Petz \cite{Petz-1}.
Using the bi-unitary invariance and the singular value decomposition
of the matrix $\Y_n$, we may represent the matrix $\Y_n$
as $\U_n \mathbf \Delta_n \V_n^*$,
where $\U_n$, $\mathbf \Delta_n$, $\V_n$ are independent,
$\U_n$~and~$\V_n$ are random unitary matrices
(with Haar distribution),
and $\mathbf \Delta_n$ is a random diagonal matrix
whose diagonal elements are the singular values of $\Y_n$,
\emph{but with random signs} (chosen uniformly at random and independently from everything else).
Note that
$$
\widetilde{\mathbf A_n^j}
=
\begin{cases}\begin{bmatrix}&\U_n (\mathbf \Delta_n^{2p} - \int x^{2p} d\mu \I) \U_n^*&\mathbf O&\\&
 \mathbf O&\V_n (\mathbf \Delta_n^{2p} - \int x^{2p} d\mu \I) \V_n^*&\end{bmatrix},&\text{ if }j=2p,\\ \\
\begin{bmatrix}&\mathbf O&\U_n \mathbf \Delta_n^{2p+1} \V_n^*&\\&
\V_n \mathbf \Delta_n^{2p+1} \U_n^*&\mathbf O&\end{bmatrix},&\text{ if }j=2p+1.
\end{cases}  
$$
Thus, the non-zero $n \times n$ blocks in the matrix
$
\widetilde{\mathbf A_n^{j_1}} \J(\alpha) \cdots \widetilde{\mathbf A_n^{j_k}} \J(\alpha)
$
in \eqref{freeas1} are products of the matrices
$\U_n (\mathbf \Delta_n^{2p} - \int x^{2p} d\mu \I) \U_n^*$,
$\V_n (\mathbf \Delta_n^{2p} - \int x^{2p} d\mu \I) \V_n^*$,
$\U_n \mathbf \Delta_n^{2p+1} \V_n^*$,
$\V_n \mathbf \Delta_n^{2p+1} \U_n^*$,
as~well~as certain powers of $\alpha$ and $\overline\alpha$,
such that each $\U_n^*$ is followed by a $\V_n$,
and each $\V_n^*$ is followed by a $\U_n$.

\pagebreak[2]

Now, by our assumptions, it is easy to see that
the eigenvalue distribution of $\mathbf \Delta_n$ converges weakly in probability to $\mu$.
Furthermore, since the matrices $\mathbf A_n^2$ and $\mathbf \Delta_n^2$ have the same eigenvalue distributions,
we have $\sup_{n \in \mathbb N} \tfrac1n \E \Tr \mathbf \Delta_n^{2k} < \infty$ for each $k \in \mathbb N$.
Starting from these observations, it is straightforward to show that
$$
\lim_{n \to \infty} \tfrac1n \Tr \mathbf \Delta_n^{k} = \int x^k \, d\mu(x) \qquad \text{in probability}
$$
for each $k \in \mathbb N$.
Therefore, by a diagonalization argument, we may find a subsequence $(\mathbf \Delta_{n_l})$ such that
$$
\lim_{l \to \infty} \tfrac{1}{n_l} \Tr \mathbf \Delta_{n_l}^{k} = \int x^k \, d\mu(x) \qquad \text{almost surely}
$$
for each $k \in \mathbb N$. 
Now, applying Theorem 2.1 in \cite{Petz-1} conditionally on the sequence $(\mathbf \Delta_{n})$,
we obtain that the families
$\{ \U_{n_l},\U_{n_l}^* \}$, $\{ \V_{n_l},\V_{n_l}^* \}$ and $\{ \mathbf \Delta_{n_l} \}$
are asymptotically free almost surely.
(Strictly speaking, we may not apply Theorem 2.1 directly, but we first have to replace
the matrices $\mathbf \Delta_{n_l}$ with matrices of uniformly bounded operator norm,
as in the proofs of Theorems 3.2 and 4.3 in \cite{Petz-1}.)
Since we may repeat the preceding argument for any subsequence of the original sequence $(\mathbf \Delta_n)$,
we come to the conclusion that the families
$\{ \U_{n},\U_{n}^* \}$, $\{ \V_{n},\V_{n}^* \}$ and $\{ \mathbf \Delta_{n} \}$
are also asymptotically free almost surely. \linebreak
Finally, as $\sup_{n \in \mathbb N} \tfrac1n \E \Tr \mathbf \Delta_n^{2k} < \infty$ for each $k \in \mathbb N$,
these families are asymptotically free in the ordinary sense as well.

But this implies that the limit in \eqref{freeas1} is equal to zero. 
This completes the~proof of the asymptotic free\-ness of $\mathbf A_n$ and $\mathbf B_n$.
\end{proof}

\begin{rem} \label{rem:asympfree}
In our applications of Proposition \ref{prop:asympfree}, the matrices $\Y_n$ 
will be products of Gaussian matrices with independent entries 
and / or matrices of uniformly bounded operator norm.
It is easy to see that in this case the assumption that
$\sup \tfrac1n \E \Tr \mathbf A_n^{2k} < \infty$
for each $k \in \mathbb N$ is satisfied.
\end{rem}

\medskip

\begin{rem} \label{BOI}
The notion of \emph{bi-unitary invariance} is relevant for computing
the limiting spectral distributions for random matrices with i.i.d.\@ standard complex Gaussian entries.
To compute the limiting distributions for random matrices with i.i.d.\@ standard real Gaussian entries
we need the notion of \emph{bi-orthogonal invariance}.
According to a side-remark in Hiai and Petz \cite{Petz-1}, 
the results for this case are analogous.

Moreover, using bi-orthogonal invariance, 
it is even possible to treat random matrices with i.i.d.\@ entries
with a common \emph{bivariate} real Gaussian distribution.
(Thus, we may allow for correlations between the real and imaginary parts, for example.)
Indeed, suppose~that the matrices $\Y^{(1)},\hdots,\Y^{(q)}$
have i.i.d.\@ Gaussian entries such that $\E Y_{jk}^{(q)} = 0$, $\E |Y_{jk}^{(q)}|^2 = 1$ and 
\begin{align}
\label{eq:secondmomentstructure-4}
\E|\re Y^{(q)}_{jk}|^2 = \sigma_1^2,\quad
\E|\im Y^{(q)}_{jk}|^2 = \sigma_2^2,\quad
\E(\re Y^{(q)}_{jk} \im X^{(q)}_{jk}) = \varrho \sigma_1 \sigma_2
\end{align}
for all $q=1,\hdots,m$, $j=1,\hdots,n_{q-1}$, $k=1,\hdots,n_q$,
where $\sigma_1^2,\sigma_2^2 \in [0,1]$ with $\sigma_1^2 + \sigma_2^2 = 1$ 
and $\varrho \in [-1,+1]$.
Then it~is easy to see that the matrices $\Y^{(j)}$ may be represented 
in the form $\Y^{(j)} = u (\tau_1 \Y'_j + i \tau_2 \Y''_j)$,
where $u \in \mathbb{C}$, $|u| = 1$, $\tau_1,\tau_2 \in [0,1]$, $\tau_1^2 + \tau_2^2 = 1$,
and the matrices $\Y'_j$ and $\Y''_j$ contain independent standard real Gaussian entries.
Note that the factor $u$ has no influence on the~singular value distribution,
and, as long as it~is rotationally invariant, on the~eigen\-value distribution.
Thus, it suffices to determine the limiting distribution
for products of the matrices $\tau_1 \Y'_j + i \tau_2 \Y''_j$.
Note that these matrices are bi-orthogonal invariant.

Since the matrices $\tau_1 \Y'_j + i \tau_2 \Y''_j$ have independent complex entries
with mean $0$ and variance $1$, their limiting singular value distribution
is well-known by the Marchenko--Pastur theorem; see e.g.\@ Theorem 3.7 in \cite{BS:10}.
Also, from independence and bi-orthogonal invariance,
it follows that for each $q=m-1,\hdots,1$, the matrices
$$
(\tau_1 \Y'_q + i \tau_2 \Y''_q)^* (\tau_1 \Y'_q + i \tau_2 \Y''_q)
\quad\text{and}\quad
\bigg(\prod_{j=q+1}^{m} (\tau_1 \Y'_j + i \tau_2 \Y''_j)\bigg)\bigg(\prod_{j=q+1}^{m} (\tau_1 \Y'_j + i \tau_2 \Y''_j)\bigg)^*
$$
are asymptotically free. Thus, the limiting singular value distribution of the product
$\prod_{j=1}^{m} (\tau_1 \Y'_j + i \tau_2 \Y''_j)$ may be found by repeated application
of Lemma \ref{rectang} in the appendix. Along these lines, 
many of the results in Section~\ref{sec:applications} may be extended to random matrices 
with a more general second moment structure as in \eqref{eq:secondmomentstructure-4}.
\end{rem}

\section{Stieltjes Transforms of Spectral Limits of Shifted Matrices}
\label{sec:stieltjes}

In this section, we assume that $n = p$, i.e.\@ $\F_\Y$ is a square matrix,
and that the matrices $\Y^{(q)}$ have independent \emph{standard} 
complex Gaussian entries, up to normalization.

Our aim is to describe the limit of the singular value distributions
of the ``shifted'' matrices $\F_\Y - \alpha \I$ ($\alpha \in \mathbb C$)
in terms of the limit of the singular value distributions of $\F_Y$.
It will be convenient to consider 
instead of the singular value distributions of the matrices $\F_\Y$ and $\F_\Y - \alpha \I$
the eigenvalue distributions of the (Hermitian) matrices 
$$
\V=\V_{\Y}=\begin{bmatrix}&\mathbf O&\F_{\Y}&\\&{\F_{\Y}^*}&\mathbf O&\end{bmatrix}
$$
and
$$
\V(\alpha)=\V_\Y(\alpha)=\V_\Y+\J(\alpha),
$$
where $\J(\alpha)$ is defined in Equation \eqref{jdef} below.
More precisely, we will show that, under appropriate conditions,
the mean eigenvalue distributions of the matrices $\V(\alpha)$
converge in moments to probability measures with compact support.
Note that this implies the weak convergence
of the mean eigenvalue distributions,
and hence the weak convergence in probability
of the eigenvalue distributions,
by the variance estimate from Section \ref{sec:variance}.

\medskip

Recall the $2n\times2n$ block matrix
\begin{equation}\label{jdef}
\J(\alpha)=\begin{bmatrix}&\mathbf O& -\alpha\I_n&\\&-\overline 
\alpha\I_n& \mathbf O&\end{bmatrix}
\end{equation}
from the previous section.
This matrix has spectral distribution 
$T(\alpha)=\frac12\delta_{+|\alpha|}+\frac12\delta_{-|\alpha|}$, 
where $\delta_a$ denotes the unit atom in the point $a$. 
We now calculate the $R$-transform of the distribution $T(\alpha)$.
It is straightforward to check that 
for the distribution $T(\alpha)$, we have
\begin{equation}\label{mz}
M_\alpha(z)=\frac{|\alpha|^2z^2}{1-|\alpha|^2z^2}
\end{equation}
and
\begin{equation}\label{gz}
G_\alpha(z)=\frac{z}{z^2-|\alpha|^2}.
\end{equation}
(Recall that $M(z)$ and $G(z)$ have been introduced 
above Equation \eqref{R-transform}.)
From \eqref{gz} it~follows that
\begin{equation}
G_\alpha^{-1}(z)=\frac{1\pm\sqrt{1+4|\alpha|^2z^2}}{2z}.
\end{equation}
Here we consider the principal branch of the square root.
In order to obtain a function $R_\alpha(z)$ that is analytic at zero,
we must take the plus sign. Therefore, \eqref{R-transform} yields
\begin{equation}\label{r}
R_{\alpha}(z):=\frac{-1+\sqrt{1+4|\alpha|^2z^2}}{2z}
\end{equation}
and
\begin{equation}\label{rtilde}
\widetilde{R}_{\alpha}(z):=\frac{-1+\sqrt{1+4|\alpha|^2z^2}}{2} \,.
\end{equation}

\pagebreak[2]
\medskip

\noindent\emph{Remark.} 
Similarly, we find that the $S$-transform of the distribution $T(\alpha)$ is~given~by
\begin{equation}\label{s}
  S_\alpha(z) 
= \frac{1+z}{z} M_\alpha^{-1}(z)
= \sqrt{\frac{1+z}{z}} \frac{1}{|\alpha|} \,.
\end{equation}
Here we take the branch of the square root such that $\im S_\alpha(z) \ge 0$
for $z \approx 0$, $z \not\in \Real_{+}$.

\medskip

We can now state the main result of this section.

\begin{thm}\label{freer}
Assume that the mean eigenvalue distributions of the matrices $\V_n$ 
converge in moments (as $n \to \infty$)
to a probability distribution $\mu_{\V} \ne \delta_0$ with compact support,
with corresponding $S$-transform $S_{\V}(z)$.
Assume also that the sequences $\V_n$ and $\J_n(\alpha)$ are asymptotically free. 
Then the mean eigenvalue distributions of the matrices $\V_n(\alpha)$ 
converge in moments (as $n \to \infty$)
to the probability measure $\mu_{\V(\alpha)} := \mu_{\V} \boxplus T(\alpha)$,
and, for $z \in \mathbb C^{+}$ with $|z|$ sufficiently large,
the Stieltjes transform $g(z,\alpha)$ of $\mu_{\V(\alpha)}$ satisfies 
the following system of equations,
\begin{align}\label{mainequation}
w(z,\alpha)&=z+\frac{\widetilde R_{\alpha}(-g(z,\alpha))}{g(z,\alpha)},\notag\\
g(z,\alpha)&=(1+w(z,\alpha)g(z,\alpha))S_{\V}(-(1+w(z,\alpha)g(z,\alpha))).
\end{align}
\end{thm}

\begin{rem}
Note that the measure $\mu_\V$ is symmetric with respect to the origin,
and recall that in this case we choose the branch of $S$-transform $S_\V(z)$
as in \eqref{eq:ChoiceOfBranch}.
\end{rem}

\begin{proof}[Proof of Theorem \ref{freer}]
By asymptotic freeness, the matrices $\V(\alpha)$ converge in moments
to the probability measure $\mu_{\V(\alpha)} := \mu_{\V} \boxplus T(\alpha)$;
see the remark below Definition \ref{def:asymptoticfreeness}.
Let~$R_{\V(\alpha)}$ and $\widetilde R_{\V(\alpha)}$ 
denote the $R$-transform and $\widetilde R$-transform
of $\mu_{\V(\alpha)}$, respectively.
Then, by the additivity of the $\widetilde R$-transform
(see Eq.~\eqref{freesum}), we have
\begin{equation}\label{rtransform}
\widetilde R_{\V(\alpha)}(z)=\widetilde R_{\V}(z)+\widetilde R_{\alpha}(z).
\end{equation}
Now we rewrite equation (\ref{rtransform}) in terms of Stieltjes transforms.
Let $\mu$ be a symmetric probability measure on $\Real$ with compact support,
and let
$$
g_\mu(z) := \int \frac{1}{t-z} d\mu(t)
$$
denote the Stieltjes transform of $\mu$.
It is well-known that $g_\mu$ maps $\mathbb C_{+}$ to $\mathbb C_{+}$
and that $g_\mu(z) = -z^{-1} + o(z^{-2})$ as $z \to \infty$.
\pagebreak[2]
Suppose that $z \in \mathbb C_{+}$ with $|z|$ sufficiently large.
Then, with $M_\mu(z)$ and $G_\mu(z)$ as in Section~\ref{freeprob},
we~have $g_\mu(z) = - G_\mu(z) = - \tfrac{1}{z} (1 + M_\mu(\tfrac{1}{z}))$.
It~there\-fore follows from \eqref{S-transform} that
\begin{equation}\notag
 S_{\mu}(-(1+zg_\mu(z)))=\frac{g_\mu(z)}{1+zg_\mu(z)}.
\end{equation}
Using equation \eqref{nica}, we get
\begin{equation}\notag
{\widetilde R}_{\mu}^{-1}(-(1+zg_\mu(z)))=-g_\mu(z),
\end{equation}
or
\begin{equation}\notag
\widetilde  R_{\mu}(-g_\mu(z))=-(1+zg_\mu(z)).
\end{equation}
Denote by $g(z,\alpha)$ the Stieltjes transform 
of the limiting spectral measure of the matrices $\V(\alpha)$.
Now replace $z$ with $-g(z,\alpha)$ in equation (\ref{rtransform}),
for $z \in \mathbb C_+$ with $|z|$ sufficiently large.
We get
\begin{equation}\notag
 -(1+zg(z,\alpha))=\widetilde R_{\V}(-g(z,\alpha))+\widetilde R_{\alpha}(-g(z,\alpha)).
\end{equation}
 We may rewrite this equation as follows
$$
-g(z,\alpha)=\widetilde R_{\V}^{-1}\left(-\left(1+g(z,\alpha)
\Big(z+\frac{\widetilde R_{\alpha}(-g(z,\alpha))}{g(z,\alpha)}\Big)\right)\right).
$$
Using the relation \eqref{nica} again, we finally get
\begin{align}\label{mainequation1}
 w&=z+\frac{\widetilde R_{\alpha}(-g(z,\alpha))}{g(z,\alpha)},\notag\\
g(z,\alpha)&=(1+wg(z,\alpha))S_{\V}(-(1+wg(z,\alpha))).
\end{align}
Thus, Theorem \ref{freer} is proved completely.
\end{proof}

\pagebreak[2]
\medskip

Henceforward, we assume that $\mu_\V$ is a symmetric probability measure on $\mathbb R \setminus \{ 0 \}$
(not~necessarily with compact support). Write $S_\V(z)$ for the corresponding $S$-transform.
For $\alpha \in \mathbb C \setminus \{ 0 \}$,
let $\mu_\alpha := T(\alpha)$ and $\mu_{\V(\alpha)} := \mu_{\V} \boxplus \mu_\alpha$.
Write $g_\V(z)$, $g_\alpha(z)$ and $g(z,\alpha) := g_{\V(\alpha)}(z)$
for the Stieltjes transforms of $\mu_\V$, $\mu_\alpha$ and $\mu_{\V(\alpha)}$.
Then, for $z = iy$ with $y > 0$ large enough,
the Stieltjes trans\-form $g(z,\alpha)$ still satisfies
the system \eqref{mainequation}.
This follows from the proof of Theorem \ref{freer}
if the probability measure $\mu_\V$ has bounded support,
and by the approximations mentioned in Section \ref{freeprob}
if the probability measure $\mu_\V$ has unbounded support.


In the next section, we will consider the system \eqref{mainequation}
with $z = iy$, where $y \geq 0$ is \emph{any} non-negative real number.
The next results show that this is possible under appropriate conditions.

\begin{lem}
\label{lem:continuation}
Let $\mu_\V$, $\mu_\alpha$ and $\mu_{\V(\alpha)}$ be defined as above.
\begin{enumerate}
\item[(i)]
For any symmetric probability measure $\mu$ on $\mathbb R$,
we have $g_\mu(i y) \in i \mathbb R^+$ for all $y > 0$.
\item[(ii)]
The function $y \mapsto \im g_\alpha(iy)$ is strictly increasing on $(0,|\alpha|]$
and strictly decreasing on $[|\alpha|,\infty)$, with $g_\alpha(i|\alpha|) = \frac{i}{2|\alpha|}$.
\item[(iii)]
For all $y > 0$, $g(iy,\alpha) \in i [0,\tfrac{1}{2|\alpha|}]$. 
\item[(iv)]
The function $z \mapsto \widetilde R_{\alpha}(-g(z,\alpha))$
has an analytic continuation to an open neighborhood $U$
of the upper imaginary half-axis. \pagebreak[2]
\end{enumerate}
By abuse of notation, we still write $\widetilde R_{\alpha}(-g(z,\alpha))$ 
for the analytic continuation in part~(iv),
and we then define $w = w(z,\alpha)$ as in \eqref{mainequation}.
\begin{enumerate}
\item[(v)]
For $z = iy$ with $y > 0$, we~have the representation
\begin{align}
\label{eq:extrtilde}
\widetilde R_\alpha(-g(iy, \alpha)) = -\tfrac12 + \tfrac12 \sqrt{1 + 4 |\alpha|^2 g(iy,\alpha)^2} \,,
\end{align}
where the branch of the square root is determined by the analytic continuation
\linebreak in part (iv). (Thus, the square root may be positive or negative!)
In particular, $\widetilde R_\alpha(-g(iy, \alpha)) \in [-1,0]$.
\item[(vi)]
Suppose that $\mu_\V$ is a probability measure on $\Real \setminus \{ 0 \}$.
Then $1 + w(i y,\alpha) g(i y,\alpha) \in (0,1)$ for all $y > 0$,
and \eqref{mainequation} holds for all $z = iy$ with $y > 0$.
\end{enumerate}
\end{lem}


\begin{proof}
Let us introduce some more notation.
For any probability measure $\mu$ on $\Real$,
we~will~use the Stieltjes transform $g_\mu(z)$,
the Cauchy transform $G_\mu(z)$, 
and the reciprocal Cauchy transform $F_\mu(z) := -1/g_\mu(z)$.
It is well-known that free additive convolution
can be analyzed using \emph{subordinating functions}.
Let $\mathcal F$ denote the class of all functions $F : \mathbb C^+ \to \mathbb C^+$
that arise as reciprocal Cauchy transforms of probability measures on $\mathbb R$.
Then, given $\mu_{\V}$ and $\mu_{\alpha}$,
there exist unique functions $Z_1$ and $Z_2$ in $\mathcal F$ such~that
\begin{align}
\label{eq:subordination}
z = Z_1(z) + Z_2(z) - F_{\V(\alpha)}(z) 
\quad\text{and}\quad
F_{\V(\alpha)}(z) = F_{\V}(Z_1(z)) = F_{\alpha}(Z_2(z))
\end{align}
for all $z \in \mathbb C^+$; see e.\,g.\@ Chistyakov and G\"otze \cite{CG}.
It is easy to see that if the measures $\mu_{\V}$ and $\mu_\alpha$ are symmetric, 
$\mu_{\V(\alpha)}$ is also symmetric, 
and the functions $F_{\V}$, $F_{\alpha}$, $F_{\V(\alpha)}$, $Z_1$ and $Z_2$ 
map $i \mathbb R^+$ to $i \mathbb R^+$.
Finally, let us mention that for any probability measure $\mu$ on $\Real$,
we have
\begin{align}
\label{eq:reci-ineq}
\im g_\mu(z) \leq \frac{1}{\im z}
\qquad\text{and}\qquad
\im F_\mu(z) \geq \im z
\end{align}
for all $z \in \mathbb C^+$, with equality only if $\mu$ is a Dirac measure.

(i) follows from a straightforward calculation.

(ii) follows by observing that
\begin{align}
\label{eq:g-alpha}
\im g_\alpha(iy) = \frac{y}{y^2 + |\alpha|^2}
\end{align}
and by using elementary calculus. \pagebreak[1]

For the proof of (iii), note that \eqref{eq:subordination} and (ii) imply
$$
\inf_{y>0} \im F_{\V(\alpha)}(i y) \ge \inf_{y>0} \im F_{\alpha}(i y) \ge 2|\alpha| \,.
$$

For the proof of (iv), recall that, for $z = iy$ with $y > 0$ large enough,
$$
\widetilde R_\alpha(-g(z,\alpha)) = \tfrac12 \big( {-} 1 + \sqrt{1 + 4 |\alpha|^2 g(z,\alpha)^2} \big) \,.
$$
Since the function $h(z) := 1 + 4 |\alpha|^2 g(z,\alpha)^2$
is a non-constant analytic function, there exists 
a simply connected open neighborhood $U$ of the imaginary axis in $\mathbb C^+$
such that $h(z) \ne 0$ for all $z \in U \setminus i \mathbb R^+$.
Moreover, if $z_0 = i y_0$ is a zero of $h(z)$ on the imaginary axis,
it~follows from (iii) that the (real-valued) function 
$y \mapsto h(i y)$ has a local minimum at the point $y_0$,
and this is only possible if $z_0 = i y_0$ is a zero of even order.
Thus, $h(z)$ is an analytic function on $U$
such that each zero is of even order,
and there exists an analytic branch of $\sqrt{h(z)}$ on $U$.
Changing the sign if necessary,
we may assume that $\sqrt{h(i y)} \in [0,1]$ for all sufficiently large $y$,
and then the desired analytic continuation of $\widetilde R_\alpha(-g(z,\alpha))$
is given by the~function $\tfrac12 \big( {-} 1 + \sqrt{h(z)} \big)$.

This also establishes Equation \eqref{eq:extrtilde}.
Since $h(i y)$ takes values in $[0,1]$ by part (iii),
the~rest of part (v) follows immediately.

We now prove (vi). It follows from our remarks 
around \eqref{eq:ST-unbounded} and \eqref{eq:ST-symmetric}
that since $\mu_\V$ is a symmetric probability measure on $\Real \setminus \{ 0 \}$,
$S_\V$ is analytic in an open set containing the interval $(-1,0)$.
Thus, it remains to show that
$$
1 + w(i y,\alpha) g(i y, \alpha) \in (0,1)
$$
for all $y > 0$, for it then follows by analytic continuation
that the second equation in \eqref{mainequation} holds for all $y > 0$.

\pagebreak[1]

By the definition of $w(iy,\alpha)$, we have
$$
  1 + w(i y,\alpha) g(i y, \alpha)
= 1 + i y g(i y, \alpha) + \widetilde R_\alpha(-g(i y, \alpha)) \,.
$$
Since $i y g(i y, \alpha) \in (-1,0)$ by part (i) and \eqref{eq:reci-ineq}
and $\widetilde R_\alpha(-g(i y, \alpha)) \in [-1,0]$ by part (v), 
it~follows that
$$
1 + w(i y,\alpha) g(i y, \alpha) \in (-1,+1) \,.
$$
Moreover, using that 
$$
\int x^2 d\mu_{\V(\alpha)}(x) = \int x^2 d\mu_\V(x) + |\alpha|^2 \,,
$$
it is straightforward to check that
\begin{align}
\label{eq:asymptapprox}
  \lim_{y \to \infty} y^2 \Big( 1 + i y g(i y, \alpha) + \widetilde R_\alpha(-g(i y, \alpha)) \Big)
= \int x^2 d\mu_\V(x)
> 0 \,.
\end{align}
Thus, we have
$$
1 + w(i y,\alpha) g(i y, \alpha) > 0
$$
for $y > 0$ large enough,
and it remains to show (by continuity) that 
$$
1 + w(i y, \alpha) g(i y, \alpha) \ne 0
$$
for all $y > 0$. 
Suppose by way of contradiction that 
$1 + w(i y_0,\alpha) g(i y_0,\alpha) = 0$ for some $y_0 > 0$.
By~\eqref{eq:asymptapprox}, 
we may assume without loss of generality 
that $y_0 > 0$ is maximal with this property.
Then $1 + w(i y, \alpha) g(i y, \alpha) \in (0,1)$
for all $y > y_0$, and by analytic continuation,
the second equation in \eqref{mainequation} holds
for all $y > y_0$. Letting $y \downarrow y_0$, we get
$$
-g(i y_0,\alpha) = \lim_{y \downarrow y_0} \Big( -(1+w(i y,\alpha) g(i y,\alpha)) S_\V(-(1+w(i y,\alpha) g(i y,\alpha))) \Big) = 0 \,,
$$
since $(-x) S_{\V}(-x) \to 0$ as $x \downarrow 0$.
But this is a contradiction to (i).
\end{proof}

\begin{lem}
\label{lem:limit}
Let $\mu_\V$, $\mu_\alpha$ and $\mu_{\V(\alpha)}$ be defined as above,
and suppose additionally that $\mu_\V$ is a symmetric probability measure on $\Real \setminus \{ 0 \}$,
but not a two-point distribution.
Then the limits 
$g(0,\alpha) := \lim_{y \downarrow 0} g(iy,\alpha)$ 
and 
$(wg)(0,\alpha) := \lim_{y \downarrow 0} (wg)(iy,\alpha)$ 
exist for all $\alpha \ne 0$.
Moreover, with the square root as in~\eqref{eq:extrtilde}, we have
\begin{align}
\label{mainequation-01}
(wg)(0,\alpha) &= \tfrac12\big( {-}1+\sqrt{1+4|\alpha|^2g(0,\alpha)^2} \,\big) \,, \nonumber\\
-g(0,\alpha) &= \widetilde S_{\V}\big( {-}(1+(wg)(0,\alpha)) \big) \,,
\end{align}
as well as
\begin{align}
\label{mainequation-02}
-g(0,\alpha) &= \widetilde S_{\V}\big({-}\tfrac12\big( 1+\sqrt{1+4|\alpha|^2g(0,\alpha)^2} \,\big)\big) \,,
\end{align}
where $\widetilde S_\V(z) := z S_\V(z)$ for $z \in (-1,0)$
and $\widetilde S_\V(z)$ is defined by continuous extension
for~$z \in \{ -1,0 \}.$
\end{lem}


\begin{proof}
We proceed by contradiction. Suppose that the limit 
$g(0,\alpha) := \lim_{y \downarrow 0} g(iy,\alpha)$ does~not exist.
Then, by Lemma \ref{lem:continuation} (iii) and continuity, 
the set of all accumulation points is a non-degenerate closed interval 
$I \subset i [0,\tfrac{1}{2|\alpha|}]$.
But, as a~con\-sequence of Lemma \ref{lem:continuation} (vi),
for each accumulation point $\widetilde{g} \in I \setminus \{ 0 \}$,
we have 
$$
{-}\widetilde{g}=\widetilde S_{\V}(-\tfrac12(1+\sqrt{1+4|\alpha|^2\widetilde{g}^2})) \,,
$$
where the square-root can be positive or negative.
It is easy to see that this implies that
$$
S_{\V}(z) = \sqrt{\frac{1+z}{z}} \frac{1}{|\alpha|} \,.
$$
In view of our remark above Theorem \ref{freer}, this means that $\mu_{\V} = T(\alpha)$,
in contradiction to our assumption that $\mu_\V$ is not a two-point distribution.
Thus, the limit $g(0,\alpha)$ exists, and \eqref{mainequation-02} holds.

The existence of the limit $(wg)(0,\alpha) := \lim_{y \downarrow 0} (wg)(iy,\alpha)$ 
as well as the relations \eqref{mainequation-01} are now simple consequences.
It is worth noting here that the sign of the square root
$\sqrt{1 + 4|\alpha|^2g(iy,\alpha)^2}$ can only change
when $g(iy,\alpha) = \frac{i}{2|\alpha|}$,
and hence must be constant for~$y \approx 0$
when $g(0,\alpha) \ne \frac{i}{2|\alpha|}$.
\end{proof}

\begin{rem}
\label{rem:limit}
A similar argument shows that under the additional assumption
that $g(iy,\alpha)$ is (jointly) continuous in $y$ and $\alpha$,
we have
$$
g(0,\alpha) = \lim_{y \downarrow 0, \beta \to \alpha} g(iy,\beta)
$$
for all $\alpha \ne 0$, and the resulting function in $\alpha$
is continuous on $\mathbb{C} \setminus \{ 0 \}$.
\end{rem}


Note that Equation \eqref{mainequation-02} has the ``trivial'' solution $g(0,\alpha) = 0$
when the sign of the square-root is negative. The next result gives a sufficient condition
for $g(0,\alpha) \ne 0$.

\begin{lem}
\label{lem:nonzero}
Let $\mu_\V$, $\mu_\alpha$ and $\mu_{\V(\alpha)}$ be defined as above,
and suppose additionally that $\mu_\V$ is a symmetric probability measure on $\Real \setminus \{ 0 \}$,
but not a two-point distribution.
Let~$\widetilde{S}_{\V}(z)$ be defined as in Lemma \ref{lem:limit}.
If $\liminf_{x \uparrow 1} |\widetilde{S}_{\V}(-x)| > 0$, 
then, for $\alpha \ne 0$ sufficiently close to zero, we have $g(0,\alpha) \ne 0$.
\end{lem}

\begin{proof}
On the one hand, it is easy to see that there exists a constant $C > 0$ such that
\begin{align}
\label{eq:tilde-S-limit-2}
|\widetilde{S}_\V(-r)| \le C \sqrt{r}
\end{align}
for all sufficiently small $r > 0$.
On the other hand, our assumption implies that
\begin{align}
\label{eq:tilde-S-limit-1}
|\widetilde{S}_\V(-1+r)| \ge c
\end{align}
for all sufficiently small $r > 0$.

Now fix $\alpha \ne 0$ and suppose by way of contradiction
that $g(0,\alpha) := \lim_{y \to 0} g(iy,\alpha) = 0$.
By \eqref{mainequation-02}, \eqref{eq:tilde-S-limit-1} and the fact 
that $\im \widetilde{S}_{\V}(z)$ is strictly negative for $z \in (-1,0)$, 
this implies that $\lim_{y \to 0} (1+(wg)(iy,\alpha)) = 0$,
or (equivalently)
$$
\lim_{y \to 0} (wg)(iy,\alpha) = \lim_{y \to 0} \left( iy g(iy,\alpha) - \tfrac12 + \tfrac12 \sqrt{1+4|\alpha|^2g(iy,\alpha)^2} \right) = -1 \,.
$$
Thus we find that the square-root must be negative for all sufficiently small $y > 0$.
Using Taylor expansion, it follows that
$$
1 + (wg)(iy,\alpha) = iy g(iy,\alpha) - |\alpha|^2 g^2(iy,\alpha) + o(g^2(iy,\alpha)) \qquad (y \downarrow 0).
$$
Recalling that $1 + (wg)(iy,\alpha)$ takes values in $[0,1]$,
we may conclude that for all sufficiently small $y > 0$,
$$
|1 + (wg)(iy,\alpha)| \leq 2 |\alpha|^2 g^2(iy,\alpha) \,.
$$
By \eqref{mainequation} and \eqref{eq:tilde-S-limit-2}, 
it follows that for all sufficiently small $y > 0$,
$$
|g(iy,\alpha)| = |\widetilde{S}_\V(-(1+(wg)(iy,\alpha)))| \leq \sqrt{2} C |\alpha| |g(iy,\alpha)| \,.
$$
For $|\alpha| < \frac{1}{\sqrt{2} C}$, this is a contradiction.
Consequently, our assumption that $g(0,\alpha) = 0$ is~wrong in this case,
and Lemma \ref{lem:nonzero} is proved.
\end{proof}

\section{Density of Limiting Spectral Distribution}\label{density} 

In this section, we compute the density of the limit distribution 
of the empirical spectral distributions of the matrices $\F_\Y$.
Here we assume that $n = p$, i.e.\@ $\F_\Y$ is a square matrix,
and that the matrices $\Y^{(q)}$ have independent \emph{standard} 
complex Gaussian entries, up to normalization.

\pagebreak[2]

To study the limiting distribution of the eigenvalue distributions
of the matrices $\F_\Y$, we use the~method of hermitization
which goes back to Girko \cite{Girko:84}.
This method may be summarized as follows:

\begin{thm}
\label{thm:hermitization}
For each $n \in \mathbb N$, let $\F_n$ be a random matrix of size $n \times n$
and $\V_n := \begin{bmatrix} 0 & \F_n \\ \F_n^* & 0 \end{bmatrix}$.
Suppose that for all $\alpha \in \mathbb{C}$,
the empirical spectral distributions $\nu_n(\,\cdot\,,\alpha)$ 
of the matrices $\V_n(\alpha) := \V_n + \J_n(\alpha)$
converge weakly in probability to a limit $\nu(\,\cdot\,,\alpha)$
and that the matrices $\F_n$ satisfy Conditions (C0), (C1) and (C2).
Then the empirical spectral distributions of the matrices $\F_n$
converge weakly in probability to a limit $\mu_\F$,
where $\mu_\F$ is the unique probability measure on the complex plane 
such that
\begin{align}
\label{logpotential-1}
U_\F(\alpha) := - \int \log|z-\alpha| \, d\mu_\F(z) = - \int \log|x| \, \nu(dx,\alpha)
\end{align}
for all $\alpha \in \mathbb{C}$.
\end{thm}

See e.g.\@ Lemma 4.3 in Bordenave and Chafa\"{i} \cite{BordChaf}.
Let us mention here that Conditions $(C0)$, $(C1)$ and $(C2)$
together with the assumption of weak convergence in probability imply
that the~function $\log|\,\cdot\,|$ is uniformly integrable in probability
for the measures $\nu_n(\,\cdot\,,\alpha)$
and that the integrals in \eqref{logpotential-1} are finite;
see also Lemma \ref{lem:logintegrability} in Appendix \ref{sec:TechnicalDetails}.

\medskip

Let us now suppose that the matrices $\F_n := \F_\Y$ are bi-unitary invariant,
that the empirical spectral distributions of the matrices 
$\V_n := \begin{bmatrix} 0 & \F_n \\ \F_n^* & 0 \end{bmatrix}$
converge weakly in prob\-ability to a compactly supported p.m.\@ $\mu_\V$, 
and that for each $k \in \mathbb N$, $\sup_{n \in \mathbb N} \E \tfrac1n \Tr \V_n^{2k} < \infty$.
\linebreak Then, by the results from Section \ref{freeprob},
the~empirical spectral distributions of the matrices $\V_n(\alpha) := \V_n + \J_n(\alpha)$
converge weakly in probability to the p.m.'s $\nu(\,\cdot\,,\alpha) := \mu_\V \boxplus T(\alpha)$, $\alpha \in \mathbb{C}$.
Hence, if the matrices $\F_\Y$ also satisfy Conditions $(C0)$, $(C1)$ and $(C2)$,
it follows from Theorem \ref{thm:hermitization} that the empirical spectral distributions 
of the matrices $\F_\Y$ converge weakly in probability to a limit $\mu_\F$,
where $\mu_\F$ is the unique probability measure on the complex plane 
such that
\begin{align}
\label{logpotential-2}
U_{\F}(\alpha) := - \int \log|z-\alpha| \, d\mu_{\F}(z) = - \int \log|x| \, \nu(dx,\alpha)
\end{align}
for all $\alpha \in \mathbb{C}$.
Moreover, by Theorem \ref{freer}, for each $\alpha \in \mathbb{C}$, 
the Stieltjes transform $g(z,\alpha)$ of the p.m. $\mu_\V \boxplus T(\alpha)$ 
satisfies the Equations \eqref{mainequation}.

\medskip

We now describe the density $f$ of the limiting spectral distribution $\mu_\F$
in terms of the $S$-transform of the measure $\mu_\V$.
In doing so, we will not use any special properties of random matrices,
but only the probability measures $\mu_\V$ and $\mu_\F$ 
and their properties stated below.


\medskip

For the rest of this section, we make the following assumptions.
%
%
We assume that $\mu_\V$ is a symmetric probability measure on $\mathbb R \setminus \{ 0 \}$
(not necessarily with compact support)
which is not a two-point distribution.
For each $\alpha \in \mathbb C$,
let $\nu(\,\cdot\,,\alpha) := \mu_{\V} \boxplus T(\alpha)$,
and assume that $\log|\,\cdot\,|$ is integrable w.r.t.\@ $\nu(\,\cdot\,,\alpha)$
and that the corresponding Stieltjes transform $g(z,\alpha)$
satisfies the Equations \eqref{mainequation} for all $z \in i \Real^{+}$.
Finally, we assume that $\mu_{\F}$ is a probability measure on $\mathbb C$ 
such that for all $\alpha \setminus \{ 0 \}$,
\begin{align}
\label{eq:logpotential-0}
U_{\F}(\alpha) := - \int \log|z-\alpha| \, d\mu_\F(z) = - \int \log|x| \, \nu(dx,\alpha) \,.
\end{align}
(In particular, we assume that the integrals are finite.)
%
%
In the sequel, we will often write $\alpha = u + iv$, with $u,v \in \Real$,
and regard functions in the complex variable $\alpha$
as functions in the real variables $u$ and $v$.

\pagebreak[2]

We shall additionally make the following assumptions:

\begin{assn}\label{continuity}
The function $g(iy,\alpha)$ is continuous and continuously differentiable
on~the set $(0,\infty) \times (\mathbb R^2 \setminus \{0\})$,
and the partial derivatives satisfy
\begin{equation}\label{eq:derivatives}
 \frac{\partial g(iy,\alpha)}{\partial u}
=\frac{2u(-i)g(iy,\alpha)}{\sqrt{1+4|\alpha|^2g^2(iy,\alpha)}}
 \frac{\partial g(iy,\alpha)}{\partial y} \,,
\end{equation}
where the square-root is the same as in \eqref{eq:extrtilde}.
Moreover, the function $g(iy,\alpha)$ admits a~continuous extension $g(0,\alpha)$ 
as $y \downarrow 0$.
\end{assn}

\begin{assn}\label{lipschitz}
For any compact set $K \subset \Real^2 \setminus \{ 0 \}$,
$$
\lim_{C\to\infty} \sup_{\alpha,\beta \in K} \frac{1}{|\alpha-\beta|} \left|\int_{-\infty}^{\infty}\log\left(1+\frac{y^2}{C^2}\right) \nu(dy,\alpha)-\int_{-\infty}^{\infty}\log\left(1+\frac{y^2}{C^2}\right) \nu(dy,\beta)\right|=0 \,.
$$
\end{assn}

The following lemma shows that Assumptions \ref{continuity} and \ref{lipschitz} are satisfied
if the probability measure $\mu_\V$ has compact support or, more generally, sufficiently small tails.
Since the~proof is rather technical, it is deferred to Appendix \ref{sec:TechnicalDetails}.

\begin{lem}\label{lem:support}
Assumptions \ref{continuity} and \ref{lipschitz} hold for probability measures $\mu_\V$ 
such that \linebreak $\mu_\V([-x,+x]^c) = \mathcal O(x^{-\eta})$ $(x \to \infty)$ for some $\eta > 0$.
\end{lem}

The logarithmic transform of the measures $\nu(\,\cdot\,\alpha)$ is defined by
\begin{align}
\label{eq:logtransform}
\Phi(\alpha):=\int_{-\infty}^{\infty}\log(|x|) \, \nu(dx,\alpha) \qquad (\alpha \ne 0) \,.
\end{align}
Note that this is exactly the integral on the right-hand side in \eqref{eq:logpotential-0}.
Similarly as above, we regard the function $\Phi$ as a function
of the real parameters $u$ and $v$.

\begin{lem}\label{lem:logtransform}
Suppose that Assumptions \ref{continuity} and \ref{lipschitz} hold.
Then the logarithmic transform $\Phi$ is differentiable on $\Real^2 \setminus \{ 0 \}$
with
\begin{align}
\label{eq:phi-derivative}
 \frac{\partial \Phi}{\partial u}(\alpha)
=\frac{u}{2|\alpha|^2} \left( 1-\sqrt{1+4|\alpha|^2g(0,\alpha)^2} \right)
 \qquad (\alpha \ne 0) \,,
\end{align}
where the function $g(0,\alpha)$ and the sign of the square-root are the same as in \eqref{mainequation-02}.
\end{lem}

\begin{proof}
For abbreviation, set $\varkappa(y,\alpha) := (-i) g(iy,\alpha)$.
Note that by Lemma \ref{lem:continuation} (iii), \linebreak
we have $\varkappa(y,\alpha) \in [0,\tfrac{1}{2|\alpha|}]$ for all $y \geq 0$.
Throughout this proof, $\alpha \in \Real^2 \setminus \{0\}$ is fixed,
and $h$ denotes a real number different from zero but sufficiently close to zero.

Introduce the integral
$$
B(C,\alpha)=\int_0^C \varkappa(y,\alpha) \, dy, \qquad C > 0,
$$
and observe that
\begin{align*}
 B(C,\alpha)
&=\int_0^C \int_{-\infty}^{+\infty} \frac{y}{y^2+x^2} \, \nu(dx,\alpha) \, dy
=\int_{-\infty}^{+\infty} \int_0^C \frac{y}{y^2+x^2}\, dy \, \nu(dx,\alpha) \\
&= -\int_{-\infty}^{\infty}\log(|x|) \, \nu(dx,\alpha)
+\frac12\int_{-\infty}^{\infty}\log(1+\frac{x^2}{C^2}) \, \nu(dx,\alpha)
+\log C \\
&= -\Phi(\alpha)
+\frac12\int_{-\infty}^{\infty}\log(1+\frac{x^2}{C^2}) \, \nu(dx,\alpha)
+\log C.
\end{align*}
Thus,
\begin{multline*}
  \frac{\Phi(\alpha+h)-\Phi(\alpha)}{h}
= - \frac{B(C,\alpha+h)-B(C,\alpha)}{h} \\
  + \frac{1}{2h}\left(\int_{-\infty}^{\infty}\log(1+\frac{x^2}{C^2}) \, \nu(dx,\alpha+h)-\int_{-\infty}^{\infty}\log(1+\frac{x^2}{C^2}) \, \nu(dx,\alpha)\right) \,.
\end{multline*}
Clearly, by Assumption \ref{lipschitz}, the expression in the second line 
can be made arbitrarily small by choosing $C$ sufficiently large.
Also, note that $\varkappa(C,\alpha) \to 0$ as $C \to \infty$.
Thus, to~complete the proof of the lemma, it remains to show that
for any $C > 0$,
\begin{align}
\label{eq:B-derivative}
 \frac{\partial B}{\partial u}(C,\alpha)
=-\frac{u}{2|\alpha|^2} \left( \sqrt{1-4|\alpha|^2\varkappa^2(C,\alpha)}-\sqrt{1-4|\alpha|^2\varkappa^2(0,\alpha)} \right) \,.
\end{align}
Let us mention here that the sign of the first square-root is positive for $C$ large enough,
whereas the sign of the second square-root may be positive or negative, 
as in Lemma \ref{lem:limit}.

To prove \eqref{eq:B-derivative}, note that by Assumption \ref{continuity}, we have
\begin{align*}
   \int_{c}^{C} \frac{\partial\varkappa}{\partial u}(y,\alpha) \, dy
&= \int_{c}^{C} \frac{2u\varkappa(y,\alpha)}{\sqrt{1-4|\alpha|^2\varkappa^2(y,\alpha)}} \frac{\partial \varkappa(y,\alpha)}{\partial y} \, dy \\
&= - \frac{u}{2|\alpha|^2} \left(\sqrt{1-4|\alpha|^2\varkappa^2(C,\alpha)}-\sqrt{1-4|\alpha|^2\varkappa^2(c,\alpha)}\,\right)
\end{align*}
for any $0 < c < C$. 
It therefore follows that  
\begin{multline*}
  \frac{1}{h} \int_c^C \! \Big( \varkappa(y,\alpha+h)-\varkappa(y,\alpha) \Big) \! \, dy
= \int_c^C \! \int_0^1 \frac{\partial\varkappa}{\partial u}(y,\alpha+th) \, dt \, dy 
= \int_0^1 \! \int_c^C \frac{\partial\varkappa}{\partial u}(y,\alpha+th) \, dy \, dt \quad\\\quad
= \int_0^1 - \frac{u+th}{2|\alpha+th|^2} \left(\sqrt{1-4|\alpha+th|^2\varkappa^2(C,\alpha\!+\!th)}-\sqrt{1-4|\alpha+th|^2\varkappa^2(c,\alpha\!+\!th)}\,\right) \, dt \,.
\end{multline*}
Thus, setting $c := |h|$ and letting $h \to 0$, we obtain 
\begin{multline*}
  \frac{B(C,\alpha\!+\!h)\!-\!B(C,\alpha)}{h}
= \frac{1}{h} \int_c^C \! \Big( \varkappa(y,\alpha+h)-\varkappa(y,\alpha) \Big) \! \, dy + \frac{1}{h} \int_0^c \! \Big( \varkappa(y,\alpha+h)-\varkappa(y,\alpha) \Big) \! \, dy \quad\\\quad
= - \frac{u}{2|\alpha|^2} \left(\sqrt{1-4|\alpha|^2\varkappa^2(C,\alpha)}-\sqrt{1-4|\alpha|^2\varkappa^2(0,\alpha)}\,\right) + o(1) \,,
\end{multline*}
where we have used the fact that the functions
$\varkappa(y,\alpha)$ and $\sqrt{1-4|\alpha|^2\varkappa^2(y,\alpha)}$
are bounded and continuous near the point $(0,\alpha)$.
This completes the proof of \eqref{eq:B-derivative}.
\end{proof}


\pagebreak[2]

\begin{thm}\label{denseigval}
 Suppose that $\mu_\V$ and $\mu_\F$ are as above
 and that Assumptions \ref{continuity} and \ref{lipschitz} hold.
 For $x > 0$ and $\alpha \ne 0$, introduce the functions
 $$
 \varkappa(x,\alpha) := (-i)g(ix,\alpha),
 \quad 
 \psi(x,\alpha) := (-i)g(ix,\alpha) \, (-i)w(ix,\alpha),
 $$
 where $w(z,\alpha)$ is defined as in Theorem \ref{freer},
 and their limits
 $$
 \varkappa(\alpha) := \varkappa(0,\alpha) := \lim_{x \downarrow 0} \varkappa(\alpha),
 \quad 
 \psi(\alpha) := \psi(0,\alpha) := \lim_{x \downarrow 0} \psi(x,\alpha).
 $$
 Then the functions $\varkappa(\,\cdot\,,\alpha)$ and $\psi(\,\cdot\,,\alpha)$ are real-valued
 with values in $[0,\tfrac{1}{2\alpha}]$ and $[0,1]$, respectively.
 Furthermore, set $\xi_\V(x) := i(-x)S_\V(-x)$ $(x \in [0;1])$.
 Then $\xi_\V \ge 0$, and we~have the relation
 \begin{equation}\label{algexp}
  \psi(\alpha)(1-\psi(\alpha))=|\alpha|^2 \, \big(\xi_\V(1-\psi(\alpha))\big)^2 \,.
 \end{equation}
 Alternatively, and more conveniently for applications, 
 we may rewrite Equation \eqref{algexp} in the form of two equations: 
 \begin{align}\label{mainequation10}
  \psi(\alpha)(1-\psi(\alpha))&=|\alpha|^2\varkappa(\alpha)^2,\notag\\
  \varkappa(\alpha)&=\xi_{\V}(1-\psi(\alpha)).
 \end{align} 

 Suppose additionally that there exists a \underline{finite} set $A$
 such that for $|\alpha| \not\in A$,
 the function $\psi(\alpha)$ is continuously differentiable at $\alpha$.
 Then, on the set $\{ (u,v) \in \mathbb{R}^2 \setminus \{ 0 \} : \sqrt{u^2+v^2} \not\in A \}$,
 the measure $\mu_\F$ has the Lebesgue density $f$ given by
 \begin{equation}\label{density1}
  f(u,v)=\frac1{2\pi}\Delta \Phi(\alpha)=\frac{1}{2\pi |\alpha|^2}  \left(u\frac{\partial \psi}{\partial u}+v\frac{\partial \psi}{\partial v}\right).
 \end{equation}
\end{thm}

\medskip

\begin{rem}
One nuisance is that the solution to \eqref{mainequation10} is not unique.
Indeed, the pair $(\varkappa,\psi) \equiv (0,1)$ is always a solution.
However, this trivial solution can be excluded using Lemma \ref{lem:nonzero}. \\
In typical applications, we will proceed as follows:
For given $\xi_\V$, solve the system \eqref{mainequation10}.
Using Lemma \ref{lem:nonzero}, argue that the solution is unique.
This is possible at least in some cases, and notably in all our applications.
Then check that the unique solution is continuously differentiable
(except on a finite number of rings) and compute $f$ using \eqref{density1}.
Finally, check that $f$ is indeed a probability density. \pagebreak[1] \\
\end{rem}

\begin{proof}[Proof of Theorem \ref{denseigval}]
Note that the measure $\nu(\cdot,\alpha)$ with corresponding Stieltjes transform $g(z,\alpha)$ 
is symmetric to the origin. 
Thus, by Lemma \ref{lem:continuation} (iii) and (vi),
we have 
$\varkappa(x,\alpha) \in [0,\tfrac{1}{2\alpha}]$
and 
$\psi(x,\alpha) \in [0,1]$
for all $x \geq 0$.
Also, the limits $\varkappa(0,\alpha)$ and $\psi(0,\alpha)$ exist
by Lemma \ref{lem:limit} and the subsequent remark.
Furthermore, by our conventions concerning the $S$-transform,
we~have $\xi_\V(x) \geq 0$ for all $x \in (0,1)$.

We now rewrite the Equations \eqref{mainequation} 
in terms of the real-valued functions $\varkappa$, $\psi$ and $\xi_\V$.
Using \eqref{mainequation} with $z=ix$, we~have
\begin{align}
\label{mainequation11}
\psi(x,\alpha)&=x\varkappa(x,\alpha)+\tfrac12-\tfrac12\sqrt{1-4|\alpha|^2\varkappa(x,\alpha)^2},\notag\\
\varkappa(x,\alpha)&=\xi_{\V}(1-\psi(x,\alpha)),
\end{align}
where the sign of the square-root is determined as in \eqref{eq:extrtilde}.
Letting $x \downarrow 0$, we~get
\begin{align}
\label{mainequation12}
\psi(\alpha)&=\tfrac12-\tfrac12\sqrt{1-4|\alpha|^2\varkappa(\alpha)^2},\notag\\
\varkappa(\alpha)&=\xi_{\V}(1-\psi(\alpha)),
\end{align}
where the sign of the square-root is determined by continuous extension.
Taking squares in the first equation and rearranging terms, we deduce that
\begin{align}\label{mainequation12.5}
 \psi(\alpha)(1-\psi(\alpha))&=|\alpha|^2\varkappa(\alpha)^2,\notag\\
 \varkappa(\alpha)&=\xi_{\V}(1-\psi(\alpha)).
\end{align}
Eliminating $\varkappa$ from these equations leads to the equivalent equation \eqref{algexp}. 

Suppose additionally that there exists a \underline{finite} set $A$
such that for $\alpha \not\in A$,
the function $\psi(\alpha)$ is continuously differentiable at $\alpha$.
Let $\alpha \not\in A$.
By Lemma \ref{lem:logtransform} and Equation \eqref{mainequation12}, we have
\begin{equation}\label{potential1}
 \frac{\partial \Phi(\alpha)}{\partial u}
=\frac{u}{2|\alpha|^2} \left( 1-\sqrt{1-4|\alpha|^2\varkappa(\alpha)^2} \right)
=\frac{u}{|\alpha|^2} \psi(\alpha) \,.
\end{equation}
Since $\psi$ is continuously differentiable w.r.t.~$u$, it~follows that
\begin{equation}\label{potential2}
 \frac{\partial^2 \Phi(\alpha)}{\partial u^2}=\frac1{|\alpha|^2}\psi-\frac{2u^2}
 {|\alpha|^4}\psi+\frac{u}{|\alpha|^2} \frac{\partial \psi}{\partial u} \,.
\end{equation}
Note that all functions depend on $|\alpha|$ only,
and are therefore symmetric with respect to $u$~and~$v$.
Thus we also have
\begin{equation}\label{potential3}
 \frac{\partial^2 \Phi(\alpha)}{\partial v^2}=\frac1{|\alpha|^2}\psi-\frac{2v^2}
 {|\alpha|^4}\psi+\frac{v}{|\alpha|^2} \frac{\partial \psi}{\partial v} \,.
\end{equation}
Summing \eqref{potential2} and \eqref{potential3}, we get, for $|\alpha| \not\in A$,
\begin{equation}\label{laplasian}
 \Delta \Phi(\alpha)=\frac{1}{|\alpha|^2}\Big(u\frac{\partial \psi}
 {\partial u}+v\frac{\partial \psi}{\partial v}\,\Big) \,.
\end{equation}
Moreover, it follows from the preceding discussion that on the open set where $|\alpha| \not\in A$,
$\Phi(\alpha)$ is twice continuously differentiable.

In view of relation \eqref{eq:logpotential-0}, this means that 
the log-potential $U_\F(\alpha)$ is twice continuously differentiable
on the open set where $|\alpha| \not\in A$.
It therefore follows by a well known result from potential theory
(see e.g.\@ \cite{saff}, Theorem II.1.3)
that the restriction of $\mu_\F$ to this set 
is absolutely continuous with Lebesgue density
$$
  f(u,v) 
= - \frac{1}{2\pi} \Delta U_\F(\alpha)
= \frac1{2\pi} \Delta \Phi(\alpha)
= \frac{1}{2\pi |\alpha|^2} \Big( u\frac{\partial \psi}{\partial u}+v\frac{\partial \psi}{\partial v} \Big).
$$
This completes the proof of Theorem \ref{denseigval}.
\end{proof}

\begin{rem}
Let us mention that in Chapter II.1 in \cite{saff}, it is assumed that 
the measure $\mu$ under consideration is finite and of compact support.
However, a closer inspection of the proof shows that the latter assumption
may be relaxed; it is sufficient to assume that the~function
$z \mapsto \log^+ |z|$ is integrable w.r.t.\@ $\mu$.
\end{rem}


\section{Applications}
\label{sec:applications}

We consider applications of Theorems \ref{singularvalueuniversality} and \ref{eigenvalueuniversality}.
These applications show that our main results allow old and new results 
on products of independent random matrices to~be~derived in a unified way.
In doing so, we shall always assume that either all random variables $X_{jk}^{(q)}$ are real with
\begin{align}
\label{eq:secondmomentstructure-8a}
\E X^{(q)}_{jk} = 0,\quad
\E |X^{(q)}_{jk}|^2 = 1,
\end{align}
or all random variables $X_{jk}^{(q)}$ are complex with
\begin{align}
\label{eq:secondmomentstructure-8b}
\E X^{(q)}_{jk} = 0,\quad
\E|\re X^{(q)}_{jk}|^2 = \E|\im X^{(q)}_{jk}|^2 = \tfrac12,\quad
\E(\re X^{(q)}_{jk} \im X^{(q)}_{jk}) = 0.
\end{align}
Clearly, under these assumptions, the corresponding random variables $Y_{jk}^{(q)}$
as in \eqref{eq:secondmomentstructure-1} have \emph{standard} real or complex Gaussian distributions,
respectively, and we may use the results from the preceding sections.
Let us note here that although we have stated these results only for the complex case,
there exist analogous results for the real case, as mentioned at the end of Section \ref{freeprob}.
Furthermore, let us emphasize that the assumptions 
\eqref{eq:secondmomentstructure-8a} and \eqref{eq:secondmomentstructure-8b}
are not needed to establish universality, but only to identify the limiting distributions.
Finally, let us mention that the assumptions
\eqref{eq:secondmomentstructure-8a} and \eqref{eq:secondmomentstructure-8b}
can be relaxed a bit; see the~Remark at the end of Section \ref{freeprob}.

\subsection{Applications of Theorem \ref{singularvalueuniversality}: Distribution of singular values} 
In this section we consider some applications of Theorem \ref{singularvalueuniversality}.
We start from the simplest case of Marchenko--Pastur law.
\subsubsection{Marchenko--Pastur Law} \label{MPLaw}
Let $m=1$ and let $\X^{(1)}=\X=\frac1{\sqrt p}(X_{jk})$ be an $n\times p$ matrix. 
We shall assume that $p=p(n)$ and $\lim_{n\to\infty}\frac n{p(n)}=y\in(0,1]$. 
We assume that $X_{jk}$, $j=1,\ldots,n$, $k=1,\ldots, p$, are independent random variables 
as in \eqref{eq:secondmomentstructure-8a} or \eqref{eq:secondmomentstructure-8b}.
Let $\mathbb F(\X) := \X$, and let $\mathcal G_n(x)$ denote the empirical
distribution function of the eigenvalues of the matrix $\W=\F_\X^{}\F_\X^*=\X\X^*$. 
Then we have the following result, cf.\@ Marchenko and Pastur \cite{MP:67}.
 
\begin{thm}\label{MarPas}
 Assume that the random variables $X_{jk}$, $j=1,\ldots,n$, $k=1,\ldots, p$ satisfy the Lindeberg 
 condition \eqref{lind}. Then
 \begin{equation}\notag
 \lim_{n\to\infty} \mathcal G_n(x)=G(x) \quad\text{in probability},
 \end{equation}
 where $G'(x)=\frac{\sqrt{(x-a)(b-x)}}{2\pi xy}$ with $a=(1-\sqrt y)^2$, $b=(1+\sqrt y)^2$.
\end{thm}

\begin{rem}
The probability distribution given by the distribution function $G(x)$ 
is called the \emph{Marchenko--Pastur distribution} with parameter $y$.
\end{rem}
   
\begin{proof}
For simplicity we shall consider only the case that the r.v.'s $X_{jk}$ are real.
Let~$\Y$ be a Gaussian matrix as in Section \ref{sec:notation}.
We prove only the universality of the singular value distribution of the matrix $\X$,
and then suppose that the limiting distribution of the singular values
of the Gaussian matrix $\Y$ is known.

To apply Theorem \ref{singularvalueuniversality}, we check conditions \eqref{rank}, 
\eqref{mat3}, \eqref{mat31}, and \eqref{mat32}. 
First we~note that in our case
\begin{equation}\label{marpas1}
 g_{jk}^{(1)}=g_{jk}=\Tr\frac{\partial \V}{\partial X_{jk}}\R^2,
\end{equation}
where $\V=\begin{bmatrix}&\mathbf O&\Z&\\ &\Z^*&\mathbf O&\end{bmatrix}$
and $\Z$ is defined as in \eqref{reprsimple}.
Let $\mathbf e_1,\ldots,\mathbf e_{n+p}$ be the standard orthonormal basis 
in $\mathbb R^{n+p}$. Then
\begin{equation}\label{marpas2}
 \frac{\partial \V}{\partial X_{jk}}=\mathbf e_j\mathbf e_{k+n}^T+\mathbf e_{k+n}
 \mathbf e_{j}^T, \text{ for }j=1,\ldots,n; 
k=1,\ldots,p.
\end{equation}
From \eqref{marpas1} and \eqref{marpas2} it follows that
\begin{equation}\label{marpas3}
 g_{jk}=[\R^2]_{j,k+n}+[\R^2]_{k+n,j}.
\end{equation}
Starting from \eqref{marpas3},
it is straightforward to check that condition \eqref{mat3} holds with constant 
$A_0=2v^{-2}$, condition \eqref{mat31} holds with constant
$A_1=8v^{-3}$ and condition \eqref{mat32} holds with constant $A_2=48v^{-4}$.

Furthermore, condition \eqref{rank} holds with constant $C(\mathbb F)=1$.
Thus, we have checked all conditions of Theorem \ref{singularvalueuniversality}. 
By Theorem \ref{singularvalueuniversality}, we obtain that the limit distribution of the singular values of the matrix $\X$ 
is the same as the limit distribution of the singular values of the Gaussian matrix $\Y$. 
It is well known that the limit distribution of the spectra of the matrices $\Y \Y^*$
is the Marchenko--Pastur distribution with parameter $y$. 
The density of this distribution is given by
\begin{equation}\label{MPD}
 p(x)=\frac{\sqrt{(x-a)(b-x)}}{2\pi xy} \mathbb I\{a\le x\le b\},
\end{equation}
where $a=(1-\sqrt y)^2$ and $b=(1+\sqrt y)^2$. 
Thus Theorem \ref{MarPas} is proved.
\end{proof}

\begin{rem}\label{MP}
The $S$-transform of the Marchenko--Pastur distribution 
with parameter $y \in (0,1]$, with density defined in \eqref{MPD}, 
is given by
$$
S(z)=\frac1{1+yz}.
$$
\end{rem}
\begin{proof}[Proof of Remark $\ref{MP}$] Let $g_y(z)$ denote the Stieltjes transform of the Marchenko--Pastur distribution. 
It is well-known that
\begin{equation*}
g_y(z)=-\frac1{z+y-1+yzg_y(z)}.
\end{equation*}
See for instance \cite{GT:04}, Equations (3.1) and (3.9).
Using this equation and the formal identity
\begin{equation*}
g(z)=-\frac1z(1+M(\frac1z)),
\end{equation*}
we get
\begin{equation*}
yzM^2(z)-(1-yz-z)M(z)+z=0.
\end{equation*}
Solving this equation with respect to $z$, we obtain
\begin{equation*}
M^{-1}(z)=\frac z{(1+yz)(1+z)}.
\end{equation*}
This equality immediately implies that
$$
S(z)=\frac1{1+yz},
$$
and Remark \ref{MP} is proved.
\end{proof}

\begin{rem}\label{inverse-MP}
Let $X$ be a r.v.\@ with Marchenko-Pastur distribution with parameter $y=1$,
and let $F_t$ denote the distribution of $(X+t)^{-1}$, $t \geq 0$.
Then $F_t \to F$ in Kolmogorov distance as $t \to 0$,
and the $S$-transform of the limit $X^{-1}$ is given by
$$
S(z)=-z.
$$
\end{rem}

\begin{proof}
The first part follows from the pointwise convergence of the corresponding densities,
which are easily calculated using \eqref{MPD}.
For the second part, we provide a formal proof. 
Recall that $S_X(z)=\frac1{z+1}$.
The corresponding Stieltjes transform is $g_X(z)=\frac12(-1+\sqrt{\frac{z-4}z})$.
Furthermore, we note that formally $M_{X^{-1}}(z)=zg_X(z)$,
where $M_{X^{-1}}(z)$ denotes the generic moment generating function of the distribution of $X^{-1}$.
This implies
\begin{align*} 
 M_{X^{-1}}(z)=\frac{-z+\sqrt{z(z-4)}}2.
\end{align*}
From this equality it follows that
\begin{align*} 
 M_{X^{-1}}^{-1}(z)=\frac{-z^2}{1+z},
\end{align*}
and therefore
\begin{align*} 
 S_{X^{-1}}(z)=-z.
\end{align*}
\end{proof}

\subsubsection{Product of Independent Rectangular Matrices}\label{products}
Let $m\ge 1$ be fixed. Let $n_0,\ldots,n_m$ denote integers 
depending on $n\ge 1$ such that $n_0=n$ and
\begin{equation}\notag
 \lim_{n\to\infty}\frac{n}{n_q(n)}=y_q\in(0,1], \quad q=1,\ldots,m.
\end{equation}
Consider independent random variables 
$X_{jk}^{(q)}$ for $q=1,\ldots,m$, $j=1,\ldots, n_{q-1}$, $k=1,\ldots, n_q$
as in \eqref{eq:secondmomentstructure-8a} or \eqref{eq:secondmomentstructure-8b}.
We introduce the matrices $\X^{(q)}=\frac1{\sqrt{n_q}}(X_{jk}^{(q)})$, 
$j=1,\ldots,n_{q-1}$, $k=1,\ldots,n_q$ for $q=1,\ldots,m$.
Let $\mathbb F(\X^{(1)},\ldots,\X^{(m)})=\prod_{q=1}^m\X^{(q)}$,
$\F=\F_\X$ and $\W=\F\F^*$.
Denote by $\mathcal G_n(x)$ the empirical spectral distribution function 
of matrix $\W$.
Then we~have the following result, see also M\"uller \cite{Mueller}
and Burda, Janik, Waclaw \cite{Burda1} for the Gaussian case 
and Alexeev, G\"otze, Tikhomirov \cite{AGT:10a}, \cite{AGT:10b}, \cite{AGT:12}
and Tikhomirov \cite{Tikh:12} for the general case.

\begin{thm} \label{productsthm}
Assume that the random variables $X_{jk}^{(q)}$, for $q=1,\ldots,m$ and $j=1,\ldots,n_{q-1}$; 
$k=1,\ldots,n_q$, satisfy the Lindeberg condition
\eqref{lind}. Then 
\begin{equation}\notag
 \lim_{n\to\infty}\mathcal G_n(x)=G^{(m)}(x) \quad\text{in probability},
\end{equation}
where the Stieltjes transform 
$s(z)=\int_{-\infty}^{\infty}\frac1{x-z}dG^{(m)}(x)$ 
of the distribution function $G^{(m)}(x)$ is determined by the equation
\begin{equation}
 1+zs(z)-s(z)\prod_{q=1}^m(1-y_q-y_qzs(z))=0.
\end{equation}
\end{thm}
\begin{proof}
For simplicity we shall assume that all r.v.'s $X^{(q)}_{jk}$ are real. 
Let $\Y^{(1)},\hdots,\Y^{(q)}$ be Gaussian matrices as in Section \ref{sec:notation}.

We now check that the function $\mathbb F$ satisfies conditions \eqref{rank}, 
\eqref{mat3}, \eqref{mat31} and \eqref{mat32}.
Condition \eqref{rank} follows from the obvious inequality
\begin{equation}\notag
 \text{\rm rank}\{\prod_{q=1}^m\mathbf A^{(q)}-\prod_{q=1}^m\mathbf B^{(q)}\}
 \le \sum_{q=1}^m\text{\rm rank}\{\mathbf A^{(q)}-\mathbf B^{(q)}\}.
\end{equation}
Let $\Z^{(1)},\hdots,\Z^{(q)}$ be defined as in \eqref{reprsimple}.
Furthermore, introduce the matrices
\begin{align}
 \mathbf H^{(q)}&=\begin{bmatrix}&\Z^{(q)}&\mathbf O\\&
 \mathbf O&(\Z^{(m-q+1)})^*\end{bmatrix}, \quad
\V_{a,b}=\prod_{q=a}^b\mathbf H^{(q)},\quad 
\notag\\& 
\J(\alpha)=\begin{bmatrix}
&\mathbf O&-\alpha\I_{n_0}\\&-\overline\alpha\I_{n_m}&\mathbf O\end{bmatrix},\quad
\J=\J(-1). \notag
\end{align}
Using these notations we have
\begin{equation}\notag
 \V=\V(\varphi)=\V_{1,m}\J,
\end{equation}
and
\begin{equation}\notag
 \frac{\partial \V_{1m}}{\partial Z_{jk}^{(q)}}
 =\frac{1}{1+\delta_{q,m-q+1}}
\bigg( \V_{1,q-1} \frac{\partial \mathbf H^{(q)}}{\partial Z_{jk}^{(q)}}\V_{q+1,m}
 +\V_{1,m-q}\frac{\partial\mathbf H^{(m-q+1)}}{\partial Z_{jk}^{(q)}}\V_{m-q+2,m} \bigg) \,.
\end{equation}

For $q=1,\hdots,m+1$,
let $\mathbf e_j^{(q)}$, $j=1,\ldots, n_{q-1}+n_{m-q+1}$, denote 
the standard ortho\-gonal basis in $\mathbb R^{n_{q-1}+n_{m-q+1}}$.
Then
\begin{equation}\label{rhs1}
 \frac{\partial \V_{1m}}{\partial Z_{jk}^{(q)}}
 =\V_{1,q-1}\mathbf e_j^{(q)}(\mathbf e_k^{(q+1)})^T\V_{q+1,m}
 +\V_{1,m-q}\mathbf e_{k+n_{m-q}}^{(m-q+1)}(\mathbf e_{j+n_{m-q+1}}^{(m-q+2)})^T \V_{m-q+2,m},
\end{equation}
for $j=1,\ldots,n_{q-1}; k=1,\ldots,n_{q}$.
From here it follows that
\begin{equation}\label{rhs}
 g_{jk}^{(q)}=\Tr\left(\left(\V_{1,q-1}\mathbf e_j^{(q)}(\mathbf e_k^{(q+1)})^T\V_{q+1,m}
+\V_{1,m-q}\mathbf e_{k+n_{m-q}}^{(m-q+1)}(\mathbf e_{j+n_{m-q+1}}^{(m-q+2)})^T \V_{m-q+2,m}\right)\J\R^2\right).
\end{equation}
Consider for instance the first term in the right hand side of \eqref{rhs}. We have
\begin{equation}\notag
     \big|\Tr\big(\V_{1,q-1}\mathbf e_j^{(q)}(\mathbf e_k^{(q+1)})^T\V_{q+1,m}\J\R^2\big)\big|
 \le v^{-2}\|\V_{1,q-1}\mathbf e_j^{(q)}\|_2\|(\mathbf e_k^{(q+1)})^T\V_{q+1,m}\|_2.
\end{equation}
%
%
%
Note that for each $q=1,\hdots,m$, 
the vector $\V_{1,q-1} \mathbf e_j^{(q)}$ is independent of the matrix $\Z^{(q)}$
due to the block structure of the matrices $\mathbf H^{(q)}$.
From here it follows that
\begin{equation}\notag
 \E\{\|\V_{1,q-1}\mathbf e_j^{(q)}\|_2^2\big|X_{jk}^{(q)},Y_{jk}^{(q)}\}
 =\E\|\V_{1,q-1}\mathbf e_j^{(q)}\|_2^2,
\end{equation}
and by Lemma 5.1 in the Appendix of \cite{AGT:10b}, we have
\begin{equation}\label{rhs2}
 \E\{\|\V_{1,q-1}\mathbf e_j^{(q)}\|_2^2\big|X_{jk}^{(q)},Y_{jk}^{(q)}\}\le C.
\end{equation}
Similarly we get
\begin{equation}\notag
 \E\{(\|\mathbf e_k^{(q+1)})^T\V_{q+1,m}\|_2^2\big|X_{jk}^{(q)},Y_{jk}^{(q)}\}\le C,
\end{equation}
$q=1,\hdots,m$.
%
%
Combining these estimates, it follows that
\begin{equation}\label{deriv001}
 \E\{|g_{jk}^{(q)}|\big|X_{jk}^{(q)},Y_{jk}^{(q)}\}\le Cv^{-2}.
\end{equation}

Furthermore, it is straightforward to check that
\begin{equation}\notag
 \frac{\partial^2\V}{\partial(Z_{jk}^{(q)})^2}=\mathbf O.
\end{equation}
This implies that
\begin{equation}\label{deriv10}
 \frac{\partial g_{jk}^{(q)}}{\partial Z_{jk}^{(q)}}=
-\Tr \frac{\partial\V}{\partial Z_{jk}^{(q)}}\R^2\frac{\partial\V}
{\partial Z_{jk}^{(q)}}\R
-\Tr \frac{\partial\V}{\partial Z_{jk}^{(q)}}\R\frac{\partial\V}
{\partial Z_{jk}^{(q)}}\R^2.
\end{equation}
Using equalities \eqref{deriv10}, \eqref{rhs1}, \eqref{rhs2} 
and Lemma 5.1 in the Appendix of \cite{AGT:10b}, we get
\begin{equation}\notag
 \E\bigg\{\bigg|\frac{\partial g_{jk}^{(q)}}{\partial Z_{jk}^{(q)}}\bigg|\bigg|X_{jk}^{(q)},Y_{jk}^{(q)}\bigg\}
 \le Cv^{-3}.
\end{equation}
Furthermore,
\begin{align}
 \frac{\partial^2 g_{jk}^{(q)}}{\partial (Z_{jk}^{(q)})^2}&=
2\Tr\frac{\partial\V}{\partial Z_{jk}^{(q)}}\R^2
\frac{\partial\V}{\partial Z_{jk}^{(q)}}\R\frac{\partial\V}
{\partial Z_{jk}^{(q)}}\R\notag\\&\qquad+
2\Tr\frac{\partial\V}{\partial Z_{jk}^{(q)}}\R
\frac{\partial\V}{\partial Z_{jk}^{(q)}}\R^2\frac{\partial\V}
{\partial Z_{jk}^{(q)}}\R+
2\Tr\frac{\partial\V}{\partial Z_{jk}^{(q)}}\R
\frac{\partial\V}{\partial Z_{jk}^{(q)}}\R\frac{\partial\V}
{\partial Z_{jk}^{(q)}}\R^2.
\label{deriv11}
\end{align}
Using equalities \eqref{deriv11}, \eqref{rhs1}, \eqref{rhs2} 
and Lemma 5.2 in the Appendix of \cite{AGT:10b}, it is straightforward 
to prove that
\begin{equation}\label{deriv100+}
 \E\bigg\{\bigg|\frac{\partial^2 g_{jk}^{(q)}}{\partial (Z_{jk}^{(q)})^2}\bigg|
 \bigg|X_{jk}^{(q)},Y_{jk}^{(q)}\bigg\}\le Cv^{-4}.
\end{equation}
 Inequalities \eqref{deriv001}, \eqref{deriv10}, \eqref{deriv100+} imply that conditions 
 \eqref{mat3}, \eqref{mat31}, \eqref{mat32} hold.
Thus, from Theorem \ref{singularvalueuniversality}, it follows that the limit distribution of the singular values 
of the matrices $\F_{\X}$ is the same as the limit distribution 
of the singular values of the matrices $\F_{\Y}$.
For the Gaussian case we may prove that the random matrices 
$(\prod_{q=1}^{l-1}\Y^{(q)})^*(\prod_{q=1}^{l-1}\Y^{(q)})$ and 
$\Y^{(l)}{\Y^{(l)}}^*$ are asymptotically free for any $l=1,\ldots,m$. 
For details see \cite{AGT:10b}, Lemma~4.1. From here and Lemma \ref{rectang} 
it~follows that the $S$-transform of the  distribution function $G^{(m)}(x)$ 
is given by
\begin{equation}\label{strans1}
S(z)=\prod_{q=1}^m\frac1{1+y_qz}.
\end{equation}
The last relation implies  that
\begin{equation}\notag
1+zs(z)-s(z)\prod_{q=1}^m(1-y_q-y_qzs(z))=0.
\end{equation}
For details, see Equations (4.9) and (4.13) in \cite{AGT:10b}.
Thus Theorem \ref{productsthm} is proved.
\end{proof}
\begin{cor}
Assume that the conditions of Theorem \ref{productsthm} hold and $y_1=\cdots=y_m=1$. Then
\begin{equation}\notag
 \lim _{n\to\infty}\mathcal G_n(x)=G(x) \text{ in probability},
\end{equation}
where the Stieltjes transform of $G(x)$ is determined by the equation
\begin{equation}\notag
1+zs(z)+(-1)^{m+1}z^ms^{m+1}(z)=0.
\end{equation}
\end{cor}

\subsubsection{Powers of Random Matrices}\label{powers} 
Consider an $n\times n$ random matrix $\X=\frac1{\sqrt n}(X_{jk})$ 
with independent entries $X_{jk}$ 
as in \eqref{eq:secondmomentstructure-8a} or \eqref{eq:secondmomentstructure-8b}.
We shall assume that the Lindeberg condition \eqref{lind} holds.
For fixed $m \geq 1$, consider the function $\mathbb F(\X)=\X^m$.
Let $\F = \F_\X$ and $\W = \F \F^*$, and denote by $\mathcal G_n(x)$ 
the distribution function of the eigenvalues of the matrix $\W$.
Then we~have the following result, cf.\@ Alexeev, G\"otze and Tikhomirov 
\cite{AGT:10}, \cite{AGT:10a}.

\begin{thm}\label{productsingular} 
Assume that the random variables $X_{jk}$ satisfy the Lindeberg condition \eqref{lind}. 
Then
 \begin{equation}\notag
  \lim_{n\to\infty}\mathcal G_n(x)=G^{(m)}(x) \quad\text{in probability},
 \end{equation}
 where $G^{(m)}(x)$ is defined by its Stieltjes transform $s(z)$,
 which satisfies the equation
 \begin{equation}\label{fuss1}
  1+zs(z)+(-1)^{m+1}z^ms^{m+1}(z)=0,
 \end{equation}
 or its moments
 \begin{equation}\label{fuss}
  M_k=\int_0^{\infty} x^kdG^{(m)}(x)=\frac1{mk+1}\binom {(m+1)k}{k}.
 \end{equation}
\end{thm}
\begin{rem}
 The numbers $M_k$ appearing in \eqref{fuss} are called Fuss--Catalan numbers.
\end{rem}
\begin{proof}
Again, for simplicity we shall consider only the case that the r.v.'s $X_{jk}$ are real.
Let $\Y$ be a Gaussian matrix as in Section \ref{sec:notation}.

We start by noting that the rank condition \eqref{rank} holds with constant $C(\mathbb F)=m$. 
In~fact,
\begin{equation}\notag
 \text{\rm rank}\{\mathbf A^m-\mathbf B^m\}\le m \,\text{\rm rank} \{\mathbf A-\mathbf B\}.
\end{equation}
In this case we have
\begin{equation}\notag
 g_{jk}=\sum_{q=1}^m\Tr\Big(\mathbf H^{q-1}\frac{\partial \mathbf H}{\partial X_{jk}}
 \mathbf H^{m-q}\J\R^2\Big),
\end{equation}
where $\mathbf H=\begin{bmatrix}&\Z&\mathbf O&\\&\mathbf O&\Z^*&\end{bmatrix}$
and $\J = \J(-1)$. Clearly,
\begin{equation}\notag
 \frac{\partial \mathbf H}{\partial X_{jk}}=\mathbf e_{j}\mathbf e_k^T+\mathbf e_{k+n}\mathbf e_{j+n}^T=:\mathbf \Delta_{jk}.
\end{equation}
Using this notation we may write
\begin{align*}
 g_{jk}&=\sum_{q=1}^m\Tr\Big(\mathbf \Delta_{jk}\mathbf H^{m-q}\J\R^2\mathbf H^{q-1}\Big) \\
       &=\sum_{q=1}^m\mathbf e_k^T\mathbf H^{m-q}\J\R^2\mathbf H^{q-1}\mathbf e_{j}
        +\sum_{q=1}^m\mathbf e_{j+n}^T\mathbf H^{m-q}\J\R^2\mathbf H^{q-1}\mathbf e_{k+n}.
\end{align*}
Consider for instance the first sum on the right-hand side.
Applying H\"older's inequality, we get
\begin{equation}\notag
 |g_{jk}|\le \sum_{q=1}^m\|\mathbf e_k^T\mathbf H^{m-q}\|_2\|\J\R^2
 \mathbf H^{q-1}\mathbf e_j\|_2
\le v^{-2}\|\mathbf e_k^T\mathbf H^{m-q}\|_2\|\mathbf H^{q-1}\mathbf e_j\|_2.
\end{equation}
Furthermore, we note that 
$$
\mathbf H=\mathbf H^{(j,k)}+\frac{Z_{jk}}{\sqrt{n}}\mathbf \Delta_{jk},
$$
where $\mathbf H^{(j,k)}$ is obtained from $\mathbf H$ 
by replacing the entry $Z_{jk}$ with zero.
Note that, for $q\ge2$, 
$$
\mathbf \Delta_{jk}^q
= 
\begin{cases}
	\mathbf \Delta_{jk},&\text{ for }j=k,\\ 
	\mathbf O,&\text{ for }j\ne k.
\end{cases}
$$
We shall use the representation 
\begin{equation}\label{power1}
 \mathbf H^q=
(\mathbf H^{(j,k)})^q+\sum_{s=1}^q
 \left(\frac{Z_{jk}}{\sqrt n}\right)^s\sum\limits_{m_1,\ldots,m_{q-s}\ge0:\atop
              m_1+\cdots+m_{q-s}\le q-s}(\mathbf H^{(j,k)})^{m_1}
              \mathbf \Delta_{jk}\cdots
(\mathbf H^{(j,k)})^{m_{q-s}}. 
\end{equation}
By the independence of the matrices $\mathbf H^{(j,k)}$ and the random variables $Z_{jk}$, we have
\begin{align}
 |\E\{\|\mathbf e_k^T\mathbf H^{m-q}\|_2^2|X_{jk}, Y_{jk}\}|&\le
C_1\E\|\mathbf e_k^T(\mathbf H^{(j,k)})^{m-q}\|_2^2\notag\\
+C_2\sum_{s=1}^{m-q}\sum\limits_{m_1,\ldots,m_{m-q-s}\ge0:\atop
              m_1+\cdots+m_{q-s}\le m-q-s}
              &\E\|\mathbf e_k^T(\mathbf H^{(j,k)})^{m_1}\mathbf \Delta_{jk}\cdots
(\mathbf H^{(j,k)})^{m_{m-q-s}}\|_2^2,\notag
\end{align}
for some absolute positive constants $C_1,C_2$. Similarly we have
\begin{align}
 |\E\{\|\mathbf H^{q-1}\mathbf e_j\|_2^2|X_{jk}, Y_{jk}\}|
 &\le
C_1\E\|(\mathbf H^{(j,k)})^{q-1}\mathbf e_j\|_2^2\notag\\
+C_2\sum_{s=1}^{q-1}\sum\limits_{m_1,\ldots,m_{q-1-s}\ge0:\atop
              m_1+\cdots+m_{q-1-s}\le q-1-s}&\E\|(\mathbf H^{(j,k)})^{m_1}
              \mathbf \Delta_{jk}\cdots
(\mathbf H^{(j,k)})^{m_{q-1-s}}\mathbf e_j\|_2^2.\notag
\end{align}
By Lemma 3 in \cite{TT:13}, we get
\begin{equation}\notag
 \E\{|g_{jk}||X_{jk},Y_{jk}\}\le C,
\end{equation}
for some positive constant $C>0$.
(Let us note here that the moment conditions here are a bit different from those in \cite{TT:13},
but using \eqref{eq:momentsAfterTruncation-11} -- \eqref{eq:momentsAfterTruncation-14},
it is easy to see that the conclusion still holds.)
Furthermore,
\begin{align}\label{deriv1}
 \frac{\partial g_{jk}}{\partial Z_{jk}}&=
-\Tr\frac{\partial \mathbf H^m}{\partial Z_{jk}}\J\R^2\frac{\partial \mathbf H^m}{\partial Z_{jk}}\J\R
-\Tr\frac{\partial \mathbf H^m}{\partial Z_{jk}}\J\R\frac{\partial \mathbf H^m}{\partial Z_{jk}}\J\R^2
+\Tr\frac{\partial^2 \mathbf H^m}{\partial Z_{jk}^2}\J\R^2.
\end{align}
We have
\begin{align}
\frac{\partial^2 \mathbf H^m}{\partial Z_{jk}^2} 
&=\sum_{q=1}^m\sum_{s=1}^{q-1}\mathbf H^{s-1}\frac{\partial \mathbf H}{\partial Z_{jk}}
\mathbf H^{q-1-s}\frac{\partial \mathbf H}{\partial Z_{jk}} \mathbf H^{m-q}+
\sum_{q=1}^m\sum_{s=1}^{m-q}\mathbf H^{q-1}\frac{\partial \mathbf H}{\partial Z_{jk}}
\mathbf H^{s-1}\frac{\partial \mathbf H}{\partial Z_{jk}} \mathbf H^{m-q-s}\notag\\
&=\sum_{q=1}^m\sum_{s=1}^{q-1}\mathbf H^{s-1}\mathbf \Delta_{jk}
\mathbf H^{q-1-s} \mathbf \Delta_{jk}\mathbf H^{m-q}+
\sum_{q=1}^m\sum_{s=1}^{m-q}\mathbf H^{q-1}\mathbf \Delta_{jk}
\mathbf H^{s-1}\mathbf \Delta_{jk} \mathbf H^{m-q-s}.\notag
\end{align}
Thus, we may rewrite the equality \eqref{deriv1} in the form
\begin{align}\label{deriv1*}
 \frac{\partial g_{jk}}{\partial Z_{jk}}=&
-\sum_{q=1}^m\sum_{r=1}^m\Tr\mathbf H^{q-1}\mathbf \Delta_{jk}\mathbf H^{m-q}\J\R^2
\mathbf H^{r-1}\mathbf \Delta_{jk}\mathbf H^{m-r}\J\R\notag\\&
-\sum_{q=1}^m\sum_{r=1}^m\Tr\mathbf H^{q-1}\mathbf \Delta_{jk}\mathbf H^{m-q}\J\R
\mathbf H^{r-1}\mathbf \Delta_{jk}\mathbf H^{m-r}\J\R^2\notag\\&+
\sum_{q=1}^m\sum_{s=1}^{q-1}\Tr\mathbf H^{s-1}\mathbf \Delta_{jk}
\mathbf H^{q-1-s} \mathbf \Delta_{jk}\mathbf H^{m-q}\J\R^2\notag\\&+
\sum_{q=1}^m\sum_{s=1}^{m-q}\Tr\mathbf H^{q-1}\mathbf \Delta_{jk}
\mathbf H^{s-1}\mathbf \Delta_{jk} \mathbf H^{m-q-s}\J\R^2.
\end{align}
All summands on the r.h.s of \eqref{deriv1*} may be bounded similarly.
For instance,
\begin{align*}
 &\E\big\{|\Tr \mathbf H^{q-1}\mathbf \Delta_{jk}\mathbf H^{m-q}\J
\R^2\mathbf H^{r-1}\mathbf \Delta_{jk}\mathbf H^{m-r}\J\R|\big|X_{jk},Y_{jk}\big\}
\\&\qquad \qquad \qquad \le 
v^{-3}\sum_{t,u,v,w}\E\big\{\|\mathbf e_t^T\mathbf H^{m-q}\|_2\|\mathbf H^{r-1} \mathbf e_u\|_2 
	\|\mathbf e_v^T\mathbf H_{m-r}\|_2\|\mathbf H^{q-1}\mathbf e_{w}\|_2\big|X_{jk},Y_{jk}\big\},
\end{align*}
where the sum is taken over all $t,u,v,w$ from the set $\{j,k,j+n,k+n\}$.
Applying H\"older's inequality, the representation \eqref{power1} 
and Lemmas 5.2 and 5.4 in \cite{AGT:12}, we get
\begin{equation}\notag
 \E\big\{|\Tr \mathbf H^{(q-1)} \mathbf \Delta_{jk}\mathbf H^{m-q}\J
\R^2\mathbf H^{r-1}\mathbf \Delta_{jk}\mathbf H^{m-r}\J\R|\big|X_{jk},Y_{jk}\big\}\le C.
\end{equation}
Thus we can prove that
\begin{equation}\label{deriv2+}
 \E\Big\{\left|\frac{\partial g_{jk}} {\partial Z_{jk}}\right|\Big|X_{jk},Y_{jk}\Big\}\le C v^{-3},
\end{equation}
which means that condition \eqref{mat31} holds.

Consider now
\begin{align}
 \frac{\partial^2 g_{jk}}{\partial Z_{jk}^2}&=\Tr \frac{\partial^3 \mathbf H^q}
 {\partial Z_{jk}^3}\R^2
-3\Tr\frac{\partial^2 \mathbf H^q}{\partial Z_{jk}^2}\R^2\frac{\partial \mathbf H^q}
{\partial Z_{jk}}\R
-3\Tr\frac{\partial^2 \mathbf H^q}{\partial Z_{jk}^2}\R\frac{\partial \mathbf H^q}
{\partial Z_{jk}}\R^2\notag\\
&+2\Tr\frac{\partial \mathbf H^q}{\partial Z_{jk}}\R^2\frac{\partial \mathbf H^q}
{\partial Z_{jk}}\R
\frac{\partial \mathbf H^q}{\partial Z_{jk}}\R+2\Tr\frac{\partial \mathbf H^q}
{\partial Z_{jk}}\R
\frac{\partial \mathbf H^q}{\partial Z_{jk}}
\R^2
\frac{\partial \mathbf H^q}{\partial Z_{jk}}\notag\\&+2\Tr\frac{\partial \mathbf H^q}
{\partial Z_{jk}}\R\frac{\partial \mathbf H^q}{\partial Z_{jk}}
\R
\frac{\partial \mathbf H^q}{\partial Z_{jk}}\R^2.\notag
\end{align}
Similarly to inequality \eqref{deriv2+} we get
\begin{equation}\notag
 \E\Big\{\left|\frac{\partial^2 g_{jk}}
 {\partial Z_{jk}^2}(\theta)\right|\Big|X_{jk},Y_{jk}\Big\}\le C v^{-4}.
\end{equation}
Thus condition \eqref{mat32} is proved.

As follows from Theorem \ref{singularvalueuniversality} the limiting singular value distributions 
of the matrices $\F_{\X}$ and $\F_{\Y}$ are the same. 
In the Gaussian case the limit distribution is computed 
in Section~4 of \cite{AGT:12}.
\end{proof}
\begin{rem}
 It follows from equation \eqref{fuss1} that the $S$-transform 
 of the distribution $G^{(m)}(x)$ is given by the formula
 \begin{equation}\label{stransform}
 S(z)=\frac1{(1+z)^m}.
 \end{equation}
 See equality \eqref{strans1} for $y_1=\cdots=y_m=1$ as well.
\end{rem}

\subsubsection{Product of Powers of Independent Matrices}\label{productpowers} 
Consider independent random $n \times n$ matrices $\X^{(1)},\ldots,\X^{(m)}$ 
with independent entries $\frac{1}{\sqrt{n}} X_{jk}^{(q)}$, 
$q=1, \ldots,m$, $j,k=1,\ldots,n$. 
We shall assume that \eqref{eq:secondmomentstructure-8a} or \eqref{eq:secondmomentstructure-8b} holds.
Let $m_1,\ldots,m_m$ be fixed positive integers. 
Let $\mathbb{F}(\X^{(1)},\hdots,\X^{(q)})=\prod_{q=1}^m(\X^{(q)})^{m_q}$,
and let $\mathcal G_n(x)$ denote the empirical distribution function 
of the eigenvalues of matrix $\W_\X=\F_\X^{}\F_\X^*$.
Then we~have the following result, cf.\@ Timushev and Tikhomirov \cite{TT:13}.

\begin{thm}\label{prodpowers}
Assume that the random variables $X_{jk}^{(q)}$, 
for $q=1,\ldots,m$ and $j,k=1,\ldots,n$, satisfy the 
Lindeberg condition
\eqref{lind}. Then 
\begin{equation}\notag
 \lim_{n\to\infty}\mathcal G_n(x)=G^{(m)}(x) \quad\text{in probability},
\end{equation}
where the Stieltjes transform $s(z)=\int_{-\infty}^{\infty}\frac1{x-z}dG^{(m)}(x)$ 
of the distribution function $G^{(m)}(x)$ is determined by the equation
\begin{equation}\notag
 1+zs(z)+(-1)^{k+1}z^ks^{k+1}(z)=0,
\end{equation}
where $k=m_1+\cdots+m_m$.
\end{thm}
\begin{proof}
For simplicity we shall assume that the $X_{jk}^{(q)}$ are real. 
Let $\Y^{(1)},\ldots,\Y^{(m)}$ denote the corresponding Gaussian matrices. 
We shall apply Theorem \ref{singularvalueuniversality}.
First we note that
\begin{equation}\notag
 \text{\rm rank}\{\F_{\X}-\F_{\widehat{\X}}\}\le 
 \sum_{q=1}^m m_{q}\,
\text{\rm rank}\{{\X^{(q)}}-{\widehat{\X}}^{(q)}\}.
\end{equation}
This implies condition \eqref{rank}. 
Conditions \eqref{mat3}, \eqref{mat31}, \eqref{mat32} 
may be checked similarly as in Subsections \ref{products} and \ref{powers}. 
For more details see \cite{TT:13}.

Theorem \ref{singularvalueuniversality} now implies that the limit distributions of the singular values 
of the matrices $\F_{\X}$ and $\F_{\Y}$ are the same. 
For the Gaussian case we use the asymptotic freeness of the matrices
$(\prod_{q=l}^{m}(\Y^{(l)})^{m_l})(\prod_{q=l}^{m}(\X^{(l)})^{m_l})^*$ 
and $((\Y^{(l-1)})^{m_{l-1}})^*(\Y^{(l-1)})^{m_{l-1}})$ 
for $l=2,\ldots,m$. 
The proof of this claim repeats the proof of  Lemma  4.2 in \cite{AGT:12}. 
From here it follows that the $S$-transform of the distribution of $\mathcal G_n(x)$
is given by
\begin{equation}\notag
 S(z)=\prod_{q=1}^m\frac1{(1+z)^{m_q}}=\frac1{(1+z)^k}.
\end{equation}
This completes the proof of Theorem \ref{prodpowers}.
\end{proof}

\subsubsection{Polynomials of Random Matrices}
Let $\X^{(1)},\ldots, \X^{(m)}$ be independent $n\times n$ random matrices
with the entries $\frac{1}{\sqrt{n}} X_{jk}^{(q)}$,
where the r.v.'s $X_{jk}^{(q)}$ are independent random variables
as in \eqref{eq:secondmomentstructure-8a} or \eqref{eq:secondmomentstructure-8b}.
Consider the matrix-valued function 
$\mathbb F(\X^{(1)},\hdots,\X^{(m)})=\sum_{q=1}^m\sum_{1\le i_1,\ldots, i_q\le m}a_{i_1\cdots i_q}
\prod_{s=1}^q\X^{(i_s)}$ and the matrix $\W_{\X}=\F_{\X}^{}\F_{\X}^*$.

Let $\Y^{(1)},\hdots,\Y^{(q)}$ be Gaussian random matrices as in Section \ref{sec:notation},
and let $\W_\Y$ be defined analogously to $\W_\X$.
Let $\mathcal G_{\X}(x)$ and $\mathcal G_{\Y}(x)$ denote the empirical spectral
distribution functions of the matrices $\W_{\X}$ and $\W_{\Y}$, respectively.

\begin{thm}\label{polynom}
 Let $\E X_{jk}^{(q)}=0$ and $\E|X_{jk}^{(q)}|^2=1$.
 Assume that the random variables $X_{jk}^{(q)}$ for $q=1,\ldots,m$; $j,k=1,\ldots,m$
 satisfy the Lindeberg condition \eqref{lind}. 
 Then
 \begin{equation}\notag
  \lim_{n\to\infty}(\mathcal  G_{\X}(x)-\mathcal G_{\Y}(x))=0 \quad\text{in probability.}
 \end{equation}
\end{thm}
\begin{proof}[Sketch of Proof]
Similarly as in Sections \ref{products} and \ref{productpowers} 
we may check the conditions \eqref{rank} and \eqref{mat3} -- \eqref{mat32a} 
of Theorem \ref{singularvalueuniversality} for each monomial functional $\mathbb F$.
Using linearity and the boundedness of the resolvents $\R_\Z$,
we may conclude that all these condition hold for the polynomial functional $\mathbb F$. 
Thus, by Theorem \ref{singularvalueuniversality}, the matrices $\W_{\X}$ and $\W_{\Y}$ 
have the~same limiting empirical spectral distribution,
and Theorem \ref{polynom} is proved.
\end{proof}

\begin{rem}
There has recently been considerable progress in computing
the limiting spectral distributions for polynomials of random matrices;
see the approach by Belinschi, Mai and Speicher \cite{BMS:13}
for self-adjoint polynomials of self-adjoint random matrices.
Possibly this approach can also be used to compute 
the limiting distributions in Theorem \ref{polynom}.
\end{rem}

\subsubsection{Spherical Ensemble}
In this section we consider the so-called spherical ensemble. 
Assume that the $X^{(q)}_{jk}$, $q=1,2$, $j,k=1,\hdots,n$,
are independent random variables 
as in \eqref{eq:secondmomentstructure-8a} or \eqref{eq:secondmomentstructure-8b}.
Moreover, assume that the r.v.'s $X_{jk}^{(2)}$ satisfy the condition
\begin{equation}\label{eq:spherical-UI2}
\max_{j,k}\E|X_{jk}^{(2)}|^2\mathbb I\{|X_{jk}^{(2)}|>M\}\to 0,\quad\text{as}\quad M\to\infty.
\end{equation}
Let $\F=\X^{(1)}(\X^{(2)})^{-1}$, where $\X^{(1)}$ and $\X^{(2)}$ denote the $n \times n$ matrices
with the entries $\frac{1}{\sqrt{n}} X_{jk}^{(1)}$ and $\frac{1}{\sqrt{n}} X_{jk}^{(2)}$,
respectively.

\emph{Remark.} It is well-known that under Condition \eqref{eq:spherical-UI2}, 
the matrix $\X^{(2)}$ is invertible with probability $1 + o(1)$ as $n \to \infty$, 
see e.g.\@ Lemma \ref{C1} in Appendix \ref{sec:SmallSingularValues}.
Thus, since we are interested in convergence in probability, 
we may restrict ourselves to the event where $\X^{(2)}$ is invertible.
This will tacitly be assumed in the subsequent proofs.

Let $\W=\F\F^*$, and let $\mathcal G_n(x)$ denote the empirical spectral distribution function 
of the matrix $\W$. Then we~have the following result, cf. Tikhomirov \cite{Tikh:13}.

\begin{thm}\label{spherical}
Assume that the random variables $X_{jk}^{(q)}$, for $q=1,2$ and $j,k=1,\ldots,n$ 
satisfy the Lindeberg condition \eqref{lind}.
Also, assume that Condition \eqref{eq:spherical-UI2} holds.
Then
\begin{equation}\notag
 \lim_{n\to\infty}\mathcal G_n(x)=G(x) \quad\text{in probability},
\end{equation}
where $g(x)=G'(x)=\frac1{\pi}\frac1{\sqrt x(1+x)}\mathbb I\{x\ge 0\}.$
\end{thm}

\noindent\emph{Remark.} Note that if $\xi$ has Cauchy density then $\eta=\xi^2$ has density $p(x)$.

\begin{proof}
In order to apply Theorem \ref{singularvalueuniversality} we need to regularize the inverse matrix $(\X^{(2)})^{-1}$.
To begin with, note that 
$
(\X^{(2)})^{-1}=((\X^{(2)})^*\X^{(2)})^{-1} (\X^{(2)})^*=(\X^{(2)})^* (\X^{(2)}(\X^{(2)})^*)^{-1}.
$
We now introduce the following matrices.
For any $t > 0$, let 
$$
\mathbf A_t=((\X^{(2)})^*\X^{(2)}+t\I)^{-1}, \quad
\widetilde{\mathbf A}_t=(\X^{(2)}(\X^{(2)})^*+t\I)^{-1},
$$
$$
(\X^{(2)})^{-1}_t=\mathbf A_t(\X^{(2)})^*=(\X^{(2)})^*\widetilde{\mathbf A}_t, \quad
\F_t=\X^{(1)}(\X^{(2)})^{-1}_t, \quad 
\W_t=\F_t\F_t^*.
$$
Also, let $s_t(z)=\frac1n\Tr \R_t$ and $s(z)=\frac1n\Tr \R$,
where $\R_t=(\W_t-z\I)^{-1}$, $\R=(\W-z\I)^{-1}$,
and $z=u+iv$, $v>0$. We prove the following.
\begin{lem}\label{inverse10} 
Under condition \eqref{eq:spherical-UI2}, we have
\begin{equation}\notag
\lim_{t\to 0}\limsup_{n\to\infty}|s_t(z)-s(z)|=0 \quad\text{in probability}.
\end{equation}
\end{lem}
\begin{proof}
Write
\begin{align}
  \R_t - \R 
= \int_0^t \frac{d\R_u}{du} \, du 
= - \int_0^t \R_u \frac{d\W_u}{du} \R_u \, du
= 2 \int_0^t \R_u \F_u \widetilde{\mathbf A}_u \F_u^* \R_u \, du \,.
\label{eq:reg-1}
\end{align}
Because the matrix $\widetilde{\mathbf A}_t$ is positive definite, 
we have
\begin{align}      
|\Tr(\R_u \F_u \widetilde{\mathbf A}_u \F_u^* \R_u)|  
\leq 
\Tr \widetilde{\mathbf A}_u \| \R_u \F_u \| \| \F_u^* \R_u \| \,.
\label{eq:reg-2}
\end{align}
Also, for any $u > 0$, we have
\begin{align}
\label{eq:reg-3}
\|\R_u \F_u\|^2 \leq v^{-1} ( 1 + |z| v^{-1} ) 
\quad\text{and}\quad
\|\F_u^* \R_u\|^2 \leq v^{-1} ( 1 + |z| v^{-1} ) \,.
\end{align}
We therefore obtain
\begin{align}      
|s_t(z) - s_0(z)| &\leq 2v^{-1} ( 1 + |z| v^{-1} ) \, \int_0^t \tfrac1n \Tr \widetilde{\mathbf A}_u \, du \,.
\label{eq:reg-5}
\end{align}
Let $s_1 \ge \cdots \ge s_n$ denote the singular values of the matrix $\X^{(2)}$.
Then the integral in~\eqref{eq:reg-5} may be represented as
\begin{align}      
  \int_0^t \tfrac1n \Tr \widetilde{\mathbf A}_u \, du
= \tfrac1n \sum_{k=1}^{n} \int_0^t (s_k^2 + u)^{-1} \, du
= \tfrac1n \sum_{k=1}^{n} \big( \log (s_k^2 + t) - \log (s_k^2) \big) \,.
\label{eq:reg-6}
\end{align}
Now, by the Marchenko--Pastur theorem (Theorem \ref{MarPas}),
we have, for any fixed $t > 0$,
\begin{align}      
\lim_{n \to \infty} \tfrac1n \sum_{k=1}^{n} \log(s_k^2 + t) = \int_0^4 \log(x + t) \, \tfrac{1}{2\pi} \sqrt{(4-x)/x} \, dx \quad\text{in probability} \,.
\label{eq:reg-7}
\end{align}
By Assumption \eqref{eq:spherical-UI2} and Lemmas \ref{C0} -- \ref{C2} 
in Appendix \ref{sec:SmallSingularValues},
the matrix $\X^{(2)}$ satisfies Conditions $(C0)$, $(C1)$ and $(C2)$.
Thus, by the Marchenko--Pastur theorem and Lemma \ref{lem:logintegrability}, 
we~also~have
\begin{align}      
\lim_{n \to \infty} \tfrac1n \sum_{k=1}^{n} \log(s_k^2) = \int_0^4 \log(x) \, \tfrac{1}{2\pi} \sqrt{(4-x)/x} \, dx \quad\text{in probability} \,.
\label{eq:reg-8}
\end{align}
Now fix $\varepsilon > 0$, and take $t > 0$ sufficiently small so that
$$
\int_0^4 \left( \log(x + t) - \log(x) \right) \, \tfrac{1}{2\pi} \sqrt{(4-x)/x} \, dx \leq \frac{\varepsilon}{3} \,.
$$
It then follows from \eqref{eq:reg-6} -- \eqref{eq:reg-8} that
\begin{multline*}
    \lim_{n \to \infty} \Pr \left\{ \int_0^{t} \tfrac1n \Tr \widetilde{\mathbf A}_u \, du \geq \varepsilon \right\} \\
\le \lim_{n \to \infty} \Pr \left\{ \left| \tfrac1n \sum_{k=1}^{n} \log(s_k^2 + t) - \int_0^4 \log(x+t) \, \tfrac{1}{2\pi} \sqrt{(4-x)/x} \, dx \right| \geq \frac{\varepsilon}{3} \right\} \\
  + \lim_{n \to \infty} \Pr \left\{ \left| \tfrac1n \sum_{k=1}^{n} \log(s_k^2) - \int_0^4 \log(x) \, \tfrac{1}{2\pi} \sqrt{(4-x)/x} \, dx \right| \geq \frac{\varepsilon}{3} \right\} 
 = 0 \,.
\end{multline*}
In view of \eqref{eq:reg-5}, this implies the statement of the lemma.
\end{proof}

Now it is enough to determine the limit distribution of the singular values  
of the matrix $\F_t$ for fixed $t > 0$ and then to take the limit as $t \to 0$ 
to find the limit distribution of the singular values of the matrix $\F$. 
We check that the conditions of Theorem \ref{singularvalueuniversality} hold
for the matrix-valued function $\mathbb F_t(\X^{(1)},\X^{(2)})=\X^{(1)}(\X^{(2)})^{-1}_t$.
Let $\F_t(\X)=\X^{(1)}(\X^{(2)})^{-1}_t$ and
$\F_t(\Y)=\Y^{(1)}(\Y^{(2)})^{-1}_t$, 
where the $\Y^{(q)}$ denote random matrices with independent Gaussian entries
as in Section \ref{sec:notation}.
Also,
let $\mathbf A_t(\X^{(2)})=((\X^{(2)})^*\X^{(2)}+t\I)^{-1}$
and $\mathbf A_t(\Y^{(2)})=((\Y^{(2)})^*\Y^{(2)}+t\I)^{-1}$.
We first check the rank condition \eqref{rank}. Clearly,
\begin{align*}
     \text{\rm rank}&\{\F_t(\X)-\F_t(\Y)\} \\
&\le \text{\rm rank} (\X^{(1)}-\Y^{(1)}) 
  + \text{\rm rank} (\mathbf A_t(\X^{(2)})-\mathbf A_t(\Y^{(2)}))
  + \text{\rm rank} (\X^{(2)}-\Y^{(2)})^* \\
&\le \text{\rm rank} (\X^{(1)}-\Y^{(1)}) + 3 \text{\rm rank} (\X^{(2)}-\Y^{(2)}) \,.  
\end{align*}
Thus, the rank condition \eqref{rank} holds with $C(\mathbb F_t) = 3$.
We now check conditions \eqref{mat3}, \eqref{mat31} and \eqref{mat32}.
As usual, we~restrict ourselves to the real case for simplicity.
By \eqref{def+}, we have 
\begin{equation}\notag
g_{jk}^{(q)}=\Tr\frac{\partial\V}{\partial Z_{jk}^{(q)}}\R^2,
\end{equation}
where 
$$
\V=\begin{bmatrix}&\mathbf O&\F_t(\mathbf Z)&\\&\F_t^*(\mathbf Z)&\mathbf O&\end{bmatrix}
\quad\text{and}\quad \R := (\V - z\I)^{-1}.
$$
Introduce the matrices
\begin{equation}\notag
 \mathbf H^{(1)}=\begin{bmatrix}&\Z^{(1)}&\mathbf O&\\&\mathbf O
 &\Z^{(2)}\mathbf A_t&\end{bmatrix},\quad
 \mathbf H^{(2)}=\begin{bmatrix}&\mathbf A_t({\Z^{(2)}})^*&\mathbf O&\\
 &\mathbf O&({\Z^{(1)}})^*&\end{bmatrix},
\end{equation}
where now $\mathbf A_t=((\Z^{(2)})^*\Z^{(2)}+t\I)^{-1}$.
We have the representation
\begin{equation}\notag
 \V=\mathbf H^{(1)}\mathbf H^{(2)}\J,
\end{equation}
where $\J=\J(-1)$. 
(Recall that $\J(\alpha)$ was defined in \eqref{jz}.)
Denote by~$\mathbf e_j$, $j=1,\dots,2n$, the vectors of 
the standard orthonormal basis of $\mathbb R^{2n}$.
First we note, for $q=1$ and $j,k=1,\ldots,n$,
\begin{align}\label{spher1}
\frac{\partial\V}{\partial Z_{jk}^{(1)}}=\mathbf e_j\mathbf e_{k}^T
\mathbf H^{(2)}\J
+\mathbf H^{(1)}
\mathbf e_{k+n}\mathbf e_{j+n}^T\J.
\end{align}
Applying H\"older's inequality, we get
\begin{align}
|g_{jk}^{(1)}(\theta)|\le \|\mathbf e_{k}^T\mathbf H^{(2)}\J\|_2\|\R^2\mathbf e_j\|_2+
\|\mathbf e_{j+n}^T\J\R^2\|_2\|\mathbf H^{(1)}\mathbf e_{k+n}\|_2.\notag
\end{align}
This implies that
\begin{multline*}
\E\left\{|g_{jk}^{(1)}(\theta)|\Big|X_{jk}^{(1)},Y_{jk}^{(1)}\right\} \\
\le v^{-2}\E\left\{\|\mathbf e_{k}^T\mathbf H^{(2)}\J\|_2\Big|X_{jk}^{(1)},Y_{jk}^{(1)}\right\}
+v^{-2}\E\left\{\|\mathbf H^{(1)}\mathbf e_{k+n}\|_2\Big|X_{jk}^{(1)},Y_{jk}^{(1)}\right\}.\notag
\end{multline*}
Note that
\begin{equation}\notag
 \|\mathbf e_{k}^T\mathbf H^{(2)}\J\|_2
=\|\overline{\mathbf e}_k^T\mathbf A_t(\Z^{(2)})^*\|_2,\quad
 \|\mathbf H^{(1)}\mathbf e_{k+n}\|_2
=\|\Z^{(2)}\mathbf A_t{\overline{\mathbf e}_k}\|_2,
\end{equation}
where $\overline{\mathbf e}_k$ denotes the corresponding standard basis vector of $\mathbb R^{n}$. 
Because the matrices $\Z^{(2)}$ and the r.v.'s $Z_{jk}^{(1)}$ are independent and 
$\|\Z^{(2)}\mathbf A_t\|\le t^{-\frac12}$, we get
\begin{align}\notag
\E\Big\{|g_{jk}^{(1)}(\theta)|\Big|X_{jk}^{(1)},Y_{jk}^{(1)}\Big\}&\le Cv^{-2}t^{-\frac12}.
\end{align}
Consider the function $g_{jk}^{(2)}$ now. Introduce some auxiliary matrices.
Let 
\begin{align}
\mathbf L_t=\begin{bmatrix}&\mathbf A_t&\mathbf O&\\&\mathbf O&\mathbf A_t&\end{bmatrix},
\quad 
\mathbf M_t=\begin{bmatrix}&\Z^{(2)}&\mathbf O&\\
&\mathbf O&\Z^{(2)}&\end{bmatrix}.\notag
\end{align}
With this notation, we have
\begin{align}\label{fu1}
  \frac{\partial \mathbf H^{(1)}}{\partial Z_{jk}^{(2)}}
 =\mathbf e_{j+n}\mathbf e_{k+n}^T\mathbf L_t
 -\mathbf M_t\mathbf L_t\mathbf e_{k+n}\mathbf e_{j+n}^T\mathbf M_t\mathbf L_t
 -\mathbf M_t\mathbf L_t\mathbf M_t^*\mathbf e_{j+n}\mathbf e_{k+n}^T\mathbf L_t
\end{align}
and
\begin{align}\label{fu2}
  \frac{\partial \mathbf H^{(2)}}{\partial Z_{jk}^{(2)}}
 =\mathbf L_t\mathbf e_{k}\mathbf e_{j}^T
 - \mathbf L_t\mathbf e_{k}\mathbf e_{j}^T\mathbf M_t\mathbf L_t\mathbf M_t^*
 - \mathbf L_t\mathbf M_t^*\mathbf e_{j}\mathbf e_{k}^T\mathbf L_t\mathbf M_t^*.
\end{align}
Furthermore,
\begin{align}\label{deriv100++}
\frac{\partial \V}{\partial Z_{jk}^{(2)}}=
\frac{\partial \mathbf H^{(1)}}{\partial Z_{jk}^{(2)}}\mathbf H^{(2)}\J+
\mathbf H^{(1)}\frac{\partial \mathbf H^{(2)}}{\partial Z_{jk}^{(2)}}\J.
\end{align}
By H\"older's inequality, we have
\begin{align}
|g_{jk}^{(2)}|&\le 
\|\mathbf e_{k+n}^T\mathbf L_t\mathbf H^{(2)}\J\|_2
\|\R^2\mathbf e_{j+n}\|_2\notag\\&+
\|\mathbf e_{j+n}^T\mathbf M_t\mathbf L_t\mathbf H^{(2)}\J\|_2
\|\R^2\mathbf M_t\mathbf L_t\mathbf e_{k+n}\|_2+
\|\mathbf e_{k+n}^T\mathbf L_t\mathbf H^{(2)}\J\|_2
\|\R^2\mathbf M_t\mathbf L_t\mathbf M_t^*\mathbf e_{j+n}\|_2\notag\\&+
\|\mathbf e_j^T\J\|_2
\|\R^2\mathbf H^{(1)}\mathbf L_t\mathbf e_k\|_2\notag\\&+
\|\mathbf e_j^T\mathbf M_t\mathbf L_t\mathbf M_t^*\J\|_2
\|\R^2\mathbf H^{(1)}\mathbf L_t\mathbf e_k\|_2+
\|\mathbf e_{k}^T\mathbf L_t\mathbf M_t^*\J\|_2
\|\R^2\mathbf H^{(1)}\mathbf L_t\mathbf M_t^*\mathbf e_j\|_2.\notag
\end{align}
Simple calculations show that
\begin{align}
|g_{jk}^{(2)}|&\le 
v^{-2} \|\overline{\mathbf e}_{k}^T\mathbf A_t(\Z^{(1)})^*\|_2
\|\overline{\mathbf e}_{j}\|_2\notag\\&+
v^{-2} \|\overline{\mathbf e}_{j}^T\Z^{(2)}\mathbf A_t(\Z^{(1)})^*\|_2
\|\Z^{(2)}\mathbf A_t\overline{\mathbf e}_{k}\|_2+
v^{-2} \|\overline{\mathbf e}_{k}^T\mathbf A_t(\Z^{(1)})^*\|_2
\|\Z^{(2)}\mathbf A_t(\Z^{(2)})^*\overline{\mathbf e}_{j}\|_2\notag\\&+
v^{-2} \|\overline{\mathbf e}_j^T\|_2
\|\Z^{(1)}\mathbf A_t\overline{\mathbf e}_k\|_2\notag\\&+
v^{-2} \|\overline{\mathbf e}_j^T\Z^{(2)}\mathbf A_t(\Z^{(2)})^*\|_2
\|\Z^{(1)}\mathbf A_t\overline{\mathbf e}_k\|_2+
v^{-2} \|\overline{\mathbf e}_{k}^T\mathbf A_t(\Z^{(2)})^*\|_2
\|\Z^{(1)}\mathbf A_t(\Z^{(2)})^*\overline{\mathbf e}_j\|_2.\notag
\end{align}
By definition of $\mathbf A_t$, we have
$$
\|\mathbf A_t\| \leq t^{-1},\quad
\|\Z^{(2)}\mathbf A_t\| \leq t^{-1/2},\quad
\|\Z^{(2)}\mathbf A_t(\Z^{(2)})^{*}\| \leq 1.
$$
Combining the last two relations, we get
\begin{align}\label{norma0}
\E\Big\{|g_{jk}^{(2)}(\theta)|\Big|X_{jk}^{(2)},Y_{jk}^{(2)}\Big\}&\le 
v^{-2} \E\Big\{\|\overline{\mathbf e}_{k}^T\mathbf A_t(\Z^{(1)})^*\|_2\Big|X_{jk}^{(2)},Y_{jk}^{(2)}\Big\}\notag\\&+
v^{-2} t^{-1/2} \E\Big\{\|\overline{\mathbf e}_{j}^T\Z^{(2)}\mathbf A_t(\Z^{(1)})^*\|_2\Big|X_{jk}^{(2)},Y_{jk}^{(2)}\Big\}\notag\\&+
v^{-2} \E\Big\{\|\overline{\mathbf e}_{k}^T\mathbf A_t(\Z^{(1)})^*\|_2\Big|X_{jk}^{(2)},Y_{jk}^{(2)}\Big\}\notag\\&+
v^{-2} \E\Big\{\|\Z^{(1)}\mathbf A_t\overline{\mathbf e}_k\|_2\Big|X_{jk}^{(2)},Y_{jk}^{(2)}\Big\}\notag\\&+
v^{-2} \E\Big\{\|\Z^{(1)}\mathbf A_t\overline{\mathbf e}_k\|_2\Big|X_{jk}^{(2)},Y_{jk}^{(2)}\Big\}\notag\\&+
v^{-2} t^{-1/2} \E\Big\{\|\Z^{(1)}\mathbf A_t(\Z^{(2)})^*\overline{\mathbf e}_j\|_2\Big|X_{jk}^{(2)},Y_{jk}^{(2)}\Big\}.
\end{align}
Now, writing $\Z^{(1)} = (\Z^{(1)} - \E\Z^{(1)}) + \E\Z^{(1)}$
and using Eqs.\@ \eqref{asumplind} and \eqref{eq:momentsAfterTruncation-11},
it is straightforward to check that, for any constant vector $\mathbf v$, we have
$\E \|\Z^{(1)}\mathbf v\|_2^2 \leq C \|\mathbf v \|_2^2$.
Thus, because the matrix $\Z^{(1)}$ and the r.v.'s $X_{jk}^{(2)},Y_{jk}^{(2)}$ 
are independent, we obtain
\begin{align}
\label{norma2}
\E\Big\{\|\overline{\mathbf e}_j^T\Z^{(2)}\mathbf A_t{\Z^{(1)}}^*\|_2^{2}\Big|X_{jk}^{(2)},Y_{jk}^{(2)}\Big\} \leq Ct^{-1} \,,
\\
\label{norma4}
\E\Big\{\|\overline{\mathbf e}_j^T\mathbf A_t
{\Z^{(1)}}^*\|_2^2\Big|X_{jk}^{(2)},Y_{jk}^{(2)}\Big\} \le Ct^{-2}.
\end{align}
Inserting the bounds \eqref{norma2} and \eqref{norma4} into \eqref{norma0}, we get
\begin{equation}\notag
\E\Big\{|g_{jk}^{(2)}(\theta)|\Big|X_{jk}^{(2)},Y_{jk}^{(2)}\Big\}\le Cv^{-2}(t^{-2}+t^{-1}). 
\end{equation}
This inequality implies condition \eqref{mat3} for $q=2$.
Now we consider the condition \eqref{mat31}.
We have
\begin{equation}\label{derivspher}
 \frac{\partial g_{jk}^{(q)}}{\partial Z_{jk}^{(q)}}
 =\Tr\frac{\partial^2 \V}{\partial (Z_{jk}^{(q)})^2}\R^2
 -\Tr \frac{\partial \V}{\partial {Z_{jk}^{(q)}}}
 \R\frac{\partial \V}{\partial {Z_{jk}^{(q)}}}\R^2-
 \Tr \frac{\partial \V}{\partial {Z_{jk}^{(q)}}}\R^2
 \frac{\partial \V}{\partial {Z_{jk}^{(q)}}}\R.
\end{equation}
Using representation \eqref{spher1}, it is straightforward to check that
\begin{equation}\notag
 \frac{\partial^2 \V}{\partial (Z_{jk}^{(1)})^2}=0.
\end{equation}
This equality and \eqref{derivspher} together imply
\begin{align}
\frac{\partial g_{jk}^{(1)}}{\partial Z_{jk}^{(1)}}
&=-2\Tr (\frac{\partial \V}{\partial {Z_{jk}^{(1)}}}\R)^2\R\notag\\&=
-2\Tr (\mathbf e_j\mathbf e_{k}^T\mathbf H^{(2)}\J\R
+\mathbf H^{(1)}
\mathbf e_{k+n}\mathbf e_{j+n}^T\J\R)^2\R=-2(T_1+T_2+T_3+T_4),\notag
\end{align}
where
\begin{align}
 T_1&=\Tr \mathbf e_j\mathbf e_k^T\mathbf H^{(2)}\J\R
 \mathbf e_j\mathbf e_k^T\mathbf H^{(2)}\J\R^2,\notag\\
 T_2&=\Tr \mathbf e_j\mathbf e_k^T\mathbf H^{(2)}\J\R
 \mathbf H^{(1)}\mathbf e_{k+n}\mathbf e_{j+n}^T\J\R^2,\notag\\
 T_3&=\Tr\mathbf H^{(1)}\mathbf e_{k+n}\mathbf e_{j+n}^T\J
 \R\mathbf e_j\mathbf e_k^T\mathbf H^{(2)}\J\R^2,\notag\\
 T_4&=\Tr\mathbf H^{(1)}\mathbf e_{k+n}\mathbf e_{j+n}^T\J
 \R\mathbf H^{(1)}\mathbf e_{k+n}\mathbf e_{j+n}^T\J\R^2.\notag
\end{align}
It is easy to see that
\begin{align}
 |T_1|&\le |\mathbf e_k^T\mathbf H^{(2)}\J\R\mathbf e_j|
 |\mathbf e_k^T\mathbf H^{(2)}\J\R^2\mathbf e_j|\le 
 v^{-3}\|\mathbf e_k^T\mathbf H^{(2)}\|_2^2\le v^{-3}t^{-1},\notag\\
 |T_2|&\le |\mathbf e_k^T\mathbf H^{(2)}\J\R\mathbf H^{(1)}
 \mathbf e_{k+n}||\mathbf e_{j+n}^T\J\R^2\mathbf e_j|\le 
 v^{-3}\|\mathbf e_k^T\mathbf H^{(2)}\|_2\|\mathbf H^{(1)}\mathbf e_{k+n}\|_2
 \le v^{-3}t^{-1},\notag\\
 |T_3|&\le |\mathbf e_{j+n}^T\J\R\mathbf e_j||\mathbf e_k^T\mathbf H^{(2)}
 \J\R^2\mathbf H^{(1)} \mathbf e_{k+n}|
 \le v^{-3}\|\mathbf H^{(1)}\mathbf e_{k+n}\|_2\|\mathbf e_k^T
 \mathbf H^{(2)}\|_2\le v^{-3}t^{-1},\notag\\
 |T_4|&\le |\mathbf e_{j+n}^T\J
 \R\mathbf H^{(1)}\mathbf e_{k+n}||\mathbf e_{j+n}^T\J\R^2 
 \mathbf H^{(1)}\mathbf e_{k+n}|
 \le v^{-3}\|\mathbf H^{(1)}\mathbf e_{k+n}\|_2^2\le v^{-3}t^{-1}.\notag
\end{align}
This implies that
\begin{equation}\label{e01}
 \E\Big\{\Big|\frac{\partial g_{jk}^{(1)}}{\partial Z_{jk}^{(1)}}
 \Big|\,\Big|X_{jk}^{(1)},Y_{jk}^{(1)}\Big\}\le C v^{-3}t^{-1}.
\end{equation}
Thus condition \eqref{mat31} holds for $q=1$.
Consider $q=2$ now.
We have
\begin{align}\label{e02}
\frac{\partial g_{jk}^{(2)}}{\partial Z_{jk}^{(2)}}
&=-2\Tr \Big(\frac{\partial \V}{\partial {Z_{jk}^{(2)}}}\R\Big)^2\R+
\Tr \frac{\partial^2 \V}{\partial \big(Z_{jk}^{(2)}\big)^2}\R^2.
\end{align}
Using formula \eqref{deriv100++}, we get
\begin{align}\label{e03}
 \frac{\partial^2 \V}{\partial \big(Z_{jk}^{(2)}\big)^2}
 =\frac{\partial^2 \mathbf H^{(1)}}{\partial \big(Z_{jk}^{(2)}\big)^2}\mathbf H^{(2)}\J+
 2\frac{\partial \mathbf H^{(1)}}{\partial {Z_{jk}^{(2)}}}
 \frac{\partial \mathbf H^{(2)}}{\partial {Z_{jk}^{(2)}}}\J
 +\mathbf H^{(1)}\frac{\partial^2 \mathbf H^{(2)}}{\partial \big(Z_{jk}^{(2)}\big)^2}\J.
\end{align}
Introduce the matrices
\begin{align}
 \mathbf P_t^{(1)}&=\mathbf e_{k+n}\mathbf e_{j+n}^T\mathbf M_t
 +\mathbf M_t^*\mathbf e_{j+n}\mathbf e_{k+n}^T,\notag\\
 \mathbf P_t^{(2)}&=\mathbf e_{k}\mathbf e_{j}^T\mathbf M_t
 +\mathbf M_t^*\mathbf e_{j}\mathbf e_{k}^T.\notag
\end{align}
Simple calculations show that
\begin{align}\notag
 \frac{\partial^2 \mathbf H^{(1)}}{\partial (Z_{jk}^{(2)})^2}={\pm}A_1{\pm}\cdots{\pm}A_{9},
\end{align}
where
\begin{align}
  A_1&=\mathbf e_{j+n}\mathbf e_{k+n}^T\mathbf L_t\mathbf P_t^{(1)}
  \mathbf L_t,\quad\quad\quad\quad\quad
 A_2=\mathbf e_{j+n}\mathbf e_{k+n}^T\mathbf L_t\mathbf e_{k+n}\mathbf e_{j+n}^T
 \mathbf M_t\mathbf L_t,\notag\\
  A_3&=\mathbf M_t\mathbf L_t\mathbf P_t^{(1)}\mathbf L_t\mathbf e_{k+n}\mathbf e_{j+n}^T
  \mathbf M_t\mathbf L_t,\quad
 A_4=\mathbf M_t\mathbf L_t\mathbf e_{k+n}\mathbf e_{j+n}^T\mathbf e_{j+n}
 \mathbf e_{k+n}^T\mathbf L_t,\notag\\
A_5&=\mathbf M_t\mathbf L_t\mathbf e_{k+n}\mathbf e_{j+n}^T\mathbf M_t\mathbf L_t
\mathbf P_t^{(1)}\mathbf L_t,\quad
 A_6=\mathbf e_{j+n}\mathbf e_{k+n}^T\mathbf L_t\mathbf M_t^*\mathbf e_{j+n}
 \mathbf e_{k+n}^T\mathbf L_t,\notag\\
 A_7&=\mathbf M_t\mathbf L_t\mathbf P_t^{(1)}\mathbf L_t\mathbf M_t^*\mathbf e_{j+n}
 \mathbf e_{k+n}^T\mathbf L_t,\quad
 A_8=\mathbf M_t\mathbf L_t\mathbf e_{k+n}\mathbf e_{j+n}^T\mathbf e_{j+n}
 \mathbf e_{k+n}^T\mathbf L_t,\notag\\
 A_9&=\mathbf M_t\mathbf L_t\mathbf M_t^*\mathbf e_{j+n}
 \mathbf e_{k+n}^T\mathbf L_t\mathbf P_t^{(1)}\mathbf L_t.\notag
 \end{align}
Using H\"older's inequality, we may prove that
\begin{multline}\label{e1}
 |\Tr A_i \mathbf H^{(2)}\J\R^2|\le Cv^{-2}(t^{-1}+t^{-\frac12})(\|\overline{\mathbf e}_j^T\Z^{(2)}
 \mathbf A_t{\Z^{(1)}}^*\|_2+
 \|\overline{\mathbf e}_j^T\mathbf A_t{\Z^{(1)}}^*\|_2), 
 \\\quad\text{for}\quad i=1,\ldots,9.
\end{multline}
Using inequality \eqref{norma2} and \eqref{norma4}, we therefore obtain
\begin{equation}\label{e2}
 \E\Big\{\Big|\Tr\frac{\partial^2\mathbf H^{(1)} }{(\partial Z_{jk}^{(2)})^2}
 \mathbf H^{(2)}\J\R^2\Big|\Bigg|X_{jk}^{(2)},Y_{jk}^{(2)}\Big\}
 \le Cv^{-2}(t^{-2}+t^{-1}). 
\end{equation}
Analogously we get
\begin{align}\notag
 \frac{\partial^2 \mathbf H^{(2)}}{\partial (Z_{jk}^{(2)})^2}={\pm}B_1{\pm}\cdots{\pm}B_9,
\end{align}
where
\begin{align}
 B_1=&\mathbf L_t\mathbf P_t^{(2)}\mathbf L_t\mathbf e_k\mathbf e_j^T,
 \quad\quad\quad
 B_2=\mathbf L_t\mathbf P_t^{(2)}\mathbf L_t\mathbf e_k\mathbf e_j^T\mathbf M_t
 \mathbf L_t\mathbf M_t^*,\notag\\
 B_3=&\mathbf L_t\mathbf e_k\mathbf e_j^T\mathbf e_j\mathbf e_k^T
 \mathbf L_t\mathbf M_t^*,\quad
 B_4=\mathbf L_t\mathbf e_k\mathbf e_j^T\mathbf M_t\mathbf L_t\mathbf P_t^{(2)}
 \mathbf L_t\mathbf M_t^*,\notag\\
 B_5=&\mathbf L_t\mathbf e_k\mathbf e_j^T\mathbf M_t\mathbf L_t\mathbf e_k
 \mathbf e_j^T,\quad
 B_6=\mathbf L_t\mathbf P_t^{(2)}\mathbf L_t\mathbf M_t^*\mathbf e_j
 \mathbf e_k^T\mathbf L_t\mathbf M_t^*,\notag\\
 B_7=&\mathbf L_t\mathbf e_k\mathbf e_j^T\mathbf e_j\mathbf e_k^T\mathbf L_t
 \mathbf M_t^*,\quad
 B_8=\mathbf L_t\mathbf M_t^*\mathbf e_j\mathbf e_k^T\mathbf L_t\mathbf P_t^{(2)}
 \mathbf L_t\mathbf M_t^*,\notag\\
 B_9=&\mathbf L_t\mathbf M_t^*\mathbf e_j\mathbf e_k^T\mathbf L_t\mathbf e_k\mathbf e_j^T.\notag
 \end{align}
Using H\"older's  inequality we get
\begin{multline}\label{e4}
 |\Tr \mathbf H^{(1)}B_i\J\R^2|\le Cv^{-2}(t^{-1}+t^{-\frac12})(\|\overline{\mathbf e}_j
 \Z^{(2)}\mathbf A_t{\Z^{(1)}}^*\|_2+
 \|\overline{\mathbf e}_j\mathbf A_t{\Z^{(1)}}^*\|_2), \\
 \quad\text{for}\quad i=1,\ldots,9.
\end{multline}
Similarly to \eqref{e2} we obtain
\begin{equation}\label{e3}
 \E\Big\{\Big|\Tr\mathbf H^{(1)}\frac{\partial^2\mathbf H^{(2)} }
 {(\partial Z_{jk}^{(2)})^2}\J\R^2\Big|
 \Bigg|X_{jk}^{(2)},Y_{jk}^{(2)}\Big\}\le Cv^{-2}(t^{-2}+t^{-1}). 
\end{equation}
Finally, using the block structure of the matrices $\mathbf H^{(1)}$ and $\mathbf H^{(2)}$,
it is easy to see that
\begin{align}\label{e3a}
\Tr\frac{\partial \mathbf H^{(1)}}{\partial Z_{jk}^{(2)}}
\frac{\partial \mathbf H^{(2)}}{\partial Z_{jk}^{(2)}}
\J\R^2=0.
\end{align}
Relations \eqref{e03}, \eqref{e2}, \eqref{e3} and \eqref{e3a} together imply
\begin{equation}\label{e5}
 \E\Big\{\Big|\Tr \frac{\partial^2 \V}
 {\partial {Z_{jk}^{(2)}}^2}\R^2\Big|\Bigg|X_{jk}^{(2)},Y_{jk}^{(2)}\Big\}
 \le Cv^{-2}(t^{-2}+t^{-1}).
\end{equation}
Furthermore, relations \eqref{fu1}, \eqref{fu2}, \eqref{deriv100++},
together imply
\begin{align}\label{e7}
\Big|\Tr (\frac{\partial \V}{\partial {Z_{jk}^{(2)}}}\R)^2\R\Big|
&\le Cv^{-3}(t^{-1}+1)
(\|\overline{\mathbf e}_j^T\Z^{(2)}\mathbf A_t{\Z^{(1)}}^*\|_2^2+
\|\overline{\mathbf e}_j^T\mathbf A_t{\Z^{(1)}}^*\|_2^2)
\end{align}
and therefore, by \eqref{norma2} and \eqref{norma4},
\begin{equation}\label{e5a}
\E\Big\{\Big|\Tr (\frac{\partial \V}{\partial {Z_{jk}^{(2)}}}\R)^2\R\Big|
\Bigg|X_{jk}^{(2)},Y_{jk}^{(2)}\Big\}\le Cv^{-3}(t^{-3}+t^{-1}).
\end{equation}
This concludes the proof of \eqref{mat31} for $q=2$.
The proof of condition \eqref{mat32} is similar and hence omitted.

\pagebreak[2]

We may now apply Theorem \ref{singularvalueuniversality}.
It follows from this theorem that for each $z \in \mathbb C$,
the~Stieltjes transforms $s_t(z;\X)$ and $s_t(z;\Y)$ 
associated with the matrices $\W_t(\X)$ and $\W_t(\Y)$
have the same limit in probability (if~existent).
It then follows by Lemma \ref{inverse10} that for each $z \in \mathbb C$,
the Stieltjes transforms $s(z;\X)$ and $s(z;\Y)$ 
associated with the matrices $\W(\X)$ and $\W(\Y)$
also have the same limit in probability (if~existent).
Thus, it~remains to identify the limit in the Gaussian case.

From now on, let $\mathbf A_t$ be defined by $\mathbf A_t = ((\Y^{(2)})^*\Y^{(2)}+t\I)^{-1}$.
Then the matrices $(\Y^{(1)})^*\Y^{(1)}$ and $\mathbf A_t (\Y^{(2)})^*\Y^{(2)} \mathbf A_t$ 
are asymptotically free. By the Marchenko--Pastur theorem (Theorem \ref{MarPas}), 
their limiting (mean) empirical spectral distributions
are given by $\varrho$, the Marchenko--Pastur distribution,
and $\sigma_t$, the induced measure of $\varrho$
under the mapping $x \mapsto (x+t)^{-1} x (x+t)^{-1}$,
respectively.
Thus, by Lemma \ref{rectang},
the limiting (mean) empirical spectral distribution of $\W_t(\Y)$
is given by $\varrho \boxtimes \sigma_t$,
with $S$-transform $S_\varrho \cdot S_{\sigma_t}$.
Clearly, as $t \to 0$, we have $\sigma_t \to \sigma$ in Kolmogorov distance,
where $\sigma$ denotes the induced measure of $\varrho$ under the mapping $x \mapsto x^{-1}$.
Using \eqref{eq:boxtimes-2} and \eqref{ST-continuity}, we obtain
$\varrho \boxtimes \sigma_t \to \varrho \boxtimes \sigma$ in Kolmogorov distance
and $S_\varrho \cdot S_{\sigma_t} \to S_\varrho \cdot S_\sigma$.
It therefore follows by Lemma \ref{inverse10} that 
the limiting (mean) empirical spectral distribution 
of $\W(\Y)$ has the $S$-transform $S_\varrho \cdot S_\sigma$.
Now, by Remarks \ref{MP} and \ref{inverse-MP}, we have
\begin{equation}\notag
S_\varrho(z)=\frac1{1+z}
\qquad\text{and}\qquad
S_\sigma(z)=-z
\end{equation}
and therefore
\begin{equation}\notag
(S_\varrho \cdot S_\sigma)(z)=-\frac{z}{1+z}.
\end{equation}
After a simple calculation we get that the density of the limit distribution is given by
\begin{equation}\label{dens}
p(x)=\frac1{\pi}\frac1{\sqrt x(1+x)}\mathbb I\{x\ge 0\}.
\end{equation}
This completes the proof of Theorem \ref{spherical}. 
\end{proof}

\subsubsection{Product of Independent Matrices from Spherical Ensemble}
In this section we consider products of independent matrices 
of type $\X^{(2q-1)}(\X^{(2q)})^{-1}$,
for $q=1,\ldots, m$, assuming that all matrices and all entries of matrices are independent. 
Let $\F=\prod_{q=1}^m\X^{(2q-1)}(\X^{(2q)})^{-1}$ and $\W=\F\F^*$. 
Let $s_1^2\ge \ldots \ge s_n^2$ denote the eigenvalues of the matrix $\W$ 
and let $\mathcal G_n(x)$ denote  the empirical distribution function
$$
\mathcal G_n(x)=\frac1n\sum_{j=1}^n\mathbb I\{s_j^2\le x\}.
$$
We shall assume as usual that $\X^{(q)} = \frac{1}{\sqrt{n}} (X_{jk}^{(q)})$,
and that \eqref{eq:secondmomentstructure-8a} or \eqref{eq:secondmomentstructure-8b} holds.
Then we~have the following result, which was already announced 
by the first and third author of this paper in \cite{Tikh:talk12}. 
See also Forrester \cite{Forrester} and Forrester and Liu \cite{ForresterLiu} for the Gaussian case.
The density in \eqref{eq:psm-density} also occurs in Biane \cite{Biane}
in the context of free multiplicative L\'evy processes.

\begin{thm}\label{productspherical} 
Assume that the random variables $X_{jk}^{(q)}$, for $q=1,\ldots,2m$ and
$j,k=1,\ldots,n$ satisfy condition \eqref{eq:spherical-UI2}. 
Then 
\begin{equation}\notag
 \lim_{n\to\infty}\mathcal G_n(x)=G_m(x) \quad\text{in probability},
\end{equation}
where $G_m(x)$ denotes the distribution function with density $p_m(x)=G_m'(x)$ given by
\begin{equation}\notag
p_m(x)=\frac{1}{\pi}\frac{\sin\frac{\pi m}{m+1}}{x^{\frac m{m+1}}
\big(x^{\frac {2}{m+1}}-2x^{\frac1{m+1}}\cos\frac{\pi m}{m+1}+1\big)}.
\end{equation} 
\end{thm}

Similarly as the proof of Theorem \ref{spherical},
the proof of Theorem \ref{productspherical} is rather technical.
Before we can apply Theorem \ref{singularvalueuniversality}, we must regularize all the inverse matrices,
and this requires a slightly more complicated construction than 
in the previous subsection. 

Let us formulate a general result.
Let $\mathbf A$ and $\mathbf B$ be random matrices of size $n \times n$,
and let $\X$ be a Girko-Ginibre matrix of size $n \times n$
satisfying \eqref{eq:UI2}.
Introduce the inverse $\X_0^{-1} := \X^{-1}$
and the regularized inverse
$$
\X^{-1}_t:=(\X^*\X+t\I)^{-1}\X^*=\X^*(\X\X^*+t\I)^{-1} \,.
$$
For any $t \geq 0$,
let $\F_t := \mathbf A \X_t^{-1} \mathbf B$,
$\W_t := \F_t \F_t^*$,
$\R_t(z) := (\W_t - z)^{-1}$ ($z \not\in \Real$)
and $g_t(z) := \tfrac1n \Tr \R_t(z)$.

\begin{lem}[Regularization Lemma]
\label{AXB-lemma}
Let $\mathbf A$ and $\mathbf B$ be random matrices of size $n \times n$,
and let $\X$ be Girko--Ginibre random matrices of size $n \times n$
satisfying \eqref{eq:UI2}.
Suppose that the matrices $\mathbf B$ satisfy Conditions (C0), (C1), (C2)
and that their squared singular value distributions converge weakly in probability
to some probability measure $\nu$.
Then, for any $z \not\in \mathbb R$, 
$$
\lim_{t \to 0} \limsup_{n \to \infty} |g_{t}(z) - g_{0}(z)| = 0 \quad \text{in probability,}
$$
Even more, the convergence is uniform in $\mathbf A$.
\end{lem}

\noindent\emph{Remark.}
Note that our assumptions on $\mathbf X$ and $\mathbf B$ imply that
$$
\Pr(\{ \text{$\X$ invertible and $\mathbf B$ invertible} \}) = 1 + o(1) \quad \text{as $n \to \infty$ \,.}
$$
In the following considerations we always work on this event.
(This is possible because we~are interested in convergence in probability.)
In particular, the inverse $\X_0^{-1}$ exists on this event.


\begin{proof}
The main idea of the proof is as follows.
Firstly, we replace the matrix $\mathbf B$ by some regularized version $\mathbf B_s$
whose singular value distribution is bounded away from zero and infinity.
Secondly, we regularize the matrix $\X^{-1}$ as described above.
Thirdly, we undo the~regularization of the matrix $\mathbf B$.

\medskip

For $x > 0$ and $s \in (0,1)$, let
$$
f(s,x) := \frac{\sqrt{x^2 + s}}{1 + s \sqrt{x^2 + s}} \,.
$$
Note that as $s \to 0$, we have $f(s,x) \to f(0,x) := x$ for any $x > 0$.
Furthermore, note that for each $s > 0$, the function $x \mapsto f(s,x)$ is increasing in $x$,
with values in the bounded interval $(\frac{s^{1/2}}{1+s^{3/2}},s^{-1})$.
Also, setting $h(s,x) := \log f(s,x)$, we may write
$$
f(s,x) = \exp(h(s,x)) \quad \text{and} \quad \partial_1 f(s,x) = f(s,x) \, \partial_1 h(s,x) \,. 
$$
Here, $\partial_1$ denotes the partial derivative w.r.t.\@ the first argument.

Write $\mathbf B = \U \mathbf \Lambda \V^*$ (singular value distribution), 
and set $\mathbf B_s := \U f(s,\mathbf \Lambda) \V^*$, $s \geq 0$.
Here, $f(s,\mathbf \Lambda)$ is obtained by applying $f(s,\,\cdot\,)$
to the diagonal elements of $\mathbf \Lambda$.
Note that
\begin{align}
\label{eq:condition}
\varkappa(\mathbf B_s) := \| \mathbf B_s \| \| \mathbf B_s^{-1} \| \leq 1 + s^{-3/2}
\end{align}
and
\begin{align}
\label{eq:hs-def}
  \frac{\partial \mathbf B_s}{\partial s} 
= \U f(s,\mathbf \Lambda) \partial_1 h(s,\mathbf \Lambda) \V^*
= \U f(s,\mathbf \Lambda) \V^* \V \partial_1 h(s,\mathbf \Lambda) \V^*
=: \mathbf B_s \mathbf H_s \,.
\end{align}
Set $\F_{t,s} := \mathbf A \X_t^{-1} \mathbf B_s$,
$\W_{t,s} := \F_{t,s} \F_{t,s}^*$ and
$\R_{t,s}(z) := (\W_{t,s} - z)^{-1}$ ($z \not\in \Real$).
Then we may write
\begin{align}
  \R_t(z)- \R_0(z)
&= \R_{t,0}(z) - \R_{t,s}(z)
+ \R_{t,s}(z) - \R_{0,s}(z)
+ \R_{0,s}(z) - \R_{0,0}(z)
\nonumber\\ &= 
- \int_0^s \frac{\partial}{\partial s} \R_{t,u}(z) \, du
+ \int_0^t \frac{\partial}{\partial t} \R_{u,s}(z) \, du
+ \int_0^s \frac{\partial}{\partial s} \R_{0,u}(z) \, du \,.
\label{eq:RtR0}
\end{align}
It is easy to check that
\begin{align}
\label{eq:ddtRts}
  \frac{\partial}{\partial t} \R_{t,s} 
= \R_{ts} \F_{ts} \Big( \mathbf B_s^{-1} (\X \X^* + t)^{-1} \mathbf B_s + \mathbf B_s^* (\X \X^* + t)^{-1} (\mathbf B_s^*)^{-1} \Big) \F_{ts}^* \R_{ts} 
\end{align}
and, by \eqref{eq:hs-def},
\begin{align}  
\label{eq:ddsRts}
  \frac{\partial}{\partial s} \R_{t,s} 
= - 2 \R_{ts} \F_{ts} \mathbf H_{s} \F_{ts}^* \R_{ts} \,.
\end{align}
Now, for any $n \times n$ matrices $\M_1$, $\M_2$, $\M_3$, with $\M_2$ self-adjoint, we have
$$
|\Tr(\M_1 \M_2 \M_3)| \leq \| \M_1 \| \| \M_3 \| \Tr |\M_2| ,
$$
where $|\M_2|$ is defined by spectral calculus.
Furthermore, for any $t,s \geq 0$, we have
$$
\| \R_{ts} \F_{ts} \| \leq \left( v^{-1}(1+|z|v^{-1}) \right)^{1/2} \,,
\quad
\| \F_{ts}^* \R_{ts} \| \leq \left( v^{-1}(1+|z|v^{-1}) \right)^{1/2} \,.
$$
It therefore follows from \eqref{eq:condition}, \eqref{eq:ddtRts} and \eqref{eq:ddsRts} that
\begin{align}
\label{eq:ddtRts-2}
     \left| \frac{\partial}{\partial t} \, \tfrac1n \Tr \R_{t,s} \right|
\leq 2(1 + s^{-3/2}) \left( v^{-1}(1+|z|v^{-1}) \right) \tfrac1n \Tr (\X \X^* + t)^{-1} \,.
\end{align}
and
\begin{align}  
\label{eq:ddsRts-2}
     \left| \frac{\partial}{\partial s} \, \tfrac1n \Tr \R_{t,s} \right|
\leq 2 \left( v^{-1}(1+|z|v^{-1}) \right) \tfrac1n \Tr |\mathbf H_s| \,.
\end{align}
Taking the normalized trace in \eqref{eq:RtR0} and using the previous estimates, 
it follows that for~any $s \in (0,1)$,
\begin{multline*}
|\tfrac1n \Tr \R_t(z) - \tfrac1n \Tr \R_0(z)|
\leq 
2(1 + s^{-3/2}) \left( v^{-1}(1+|z|v^{-1}) \right) \int_0^t \tfrac1n \Tr (\X \X^* + u)^{-1} \, du \\
+
4 \left( v^{-1}(1+|z|v^{-1}) \right) \int_0^s \tfrac1n \Tr |\mathbf H_u| \, du \,.
\end{multline*}
We will now show the following:
\begin{align}
\label{eq:reg-conv-t}
(\forall \varepsilon > 0) \ \lim_{t \to 0} \limsup_{n \to \infty} \Pr\left( \int_0^t \tfrac1n \Tr (\X \X^* + u)^{-1} \, du \geq \varepsilon\right) = 0 \,.
\end{align}
\begin{align}
\label{eq:reg-conv-s}
(\forall \varepsilon > 0) \ \lim_{s \to 0} \limsup_{n \to \infty} \Pr\left( \int_0^s \tfrac1n \Tr |\mathbf H_u| \, du \geq \varepsilon\right) = 0 \,.
\end{align}
Once we have \eqref{eq:reg-conv-t} and \eqref{eq:reg-conv-s}, it is easy to complete the proof.
Indeed, fix $\varepsilon > 0$.
Then, by \eqref{eq:reg-conv-t}, there exists a function $s(t)$
such that $s(t) \to 0$ as $t \to 0$ and still
\begin{align}
\label{eq:reg-diagonalization}
\lim_{t \to 0} \limsup_{n \to \infty} \Pr\left( (1+s(t)^{-3/2}) \int_0^t \tfrac1n \Tr (\X \X^* + u)^{-1} \, du \geq \varepsilon\right) = 0 \,.
\end{align}
Using \eqref{eq:reg-conv-s}, we therefore obtain
\begin{multline*}
\lim_{t \to 0} \limsup_{n \to \infty} \Pr(|\tfrac1n \Tr \R_t(z) - \tfrac1n \Tr \R_0(z)| \geq 2\varepsilon)
\\ \leq
\lim_{t \to 0} \limsup_{n \to \infty} \Pr\left( 2(1 + s(t)^{-3/2})\big(v^{-1} (1+|z|v^{-1})\big) \int_0^t \tfrac1n \Tr (\X \X^* + u)^{-1} \, du \geq \varepsilon\right)
\\ +
\lim_{t \to 0} \limsup_{n \to \infty} \Pr\left( 4\big(v^{-1} (1+|z|v^{-1})\big) \int_0^{s(t)} \tfrac1n \Tr |\mathbf H_u| \, du \geq \varepsilon\right)
=
0 \,,
\end{multline*}
i.e.\@ the desired conclusion.

Thus, it remains to show \eqref{eq:reg-conv-t} and \eqref{eq:reg-conv-s}.
We have already checked that \eqref{eq:reg-conv-t} follows from our assumption
that $\X$ satisfies Conditions (C0), (C1) and (C2); 
see the proof of Lemma \ref{inverse10}.
It is straightforward to check that
$$
\partial_1 h(s,x) = \frac{1-2(x^2+s)^{3/2}}{2(x^2+s)(1+s(x^2+s)^{1/2})} \,.
$$
Since $s < 1$, it follows that
$$
\left| \partial_1 h(s,x) \right| \leq \frac{C}{s+x^2} \quad (x \le 1)
\qquad\text{and}\qquad
\left| \partial_1 h(s,x) \right| \leq \frac{Cx}{1+sx} \quad (x \ge 1) \,.
$$
We therefore obtain
\begin{align*}
     \int_0^{s} & \tfrac1n \Tr |\mathbf H_u| \, du
\leq \int_0^{s} \tfrac1n \sum_{k=1}^{n} \left| \partial_1 h(u,s_k(\mathbf B)) \right| \, du \\
&\leq \frac{C}{n} \sum_{\substack{k: s_k(\mathbf B) \le 1}} \int_0^s \frac{1}{u+s_k^2(\mathbf B)} \, du
    + \frac{C}{n} \sum_{\substack{k: s_k(\mathbf B) \ge 1}} \int_0^s \frac{s_k(\mathbf B)}{1+u s_k(\mathbf B)} \, du \\
&\leq \frac{C}{n} \sum_{\substack{k: s_k(\mathbf B) \le 1}} \Big( \log(s+s_k^2(\mathbf B)) - \log(s_k^2(\mathbf B)) \Big)
    + \frac{C}{n} \sum_{\substack{k: s_k(\mathbf B) \ge 1}} \log(1+s s_k(\mathbf B)) \,.
\end{align*}
Since the matrix $\mathbf B$ satisfies (C0), (C1), (C2)
and the squared singular value distribution of $\mathbf B$ is weakly convergent
to a limit $\nu$ with $\int |\log| \, d\nu < \infty$ 
(by Lemma \ref{lem:logintegrability} in the appendix),
this implies \eqref{eq:reg-conv-s}.
\end{proof}

We now return to Theorem \ref{productspherical}.

\begin{proof}[Proof of Theorem \ref{productspherical}]
The proof is by induction on $m$.
Suppose that $m > 1$ and that the result is true of all smaller values of $m$.
(The case of a single spherical matrix is Theorem~\ref{spherical}.)
We start with a regularisation of the inverse matrices.
We introduce the following matrices
\begin{align}
(\X^{(2q)})_0^{-1}=(\X^{(2q)})^{-1}
\quad\text{and}\quad
(\X^{(2q)})_t^{-1}=({\X^{(2q)}}^*{\X^{(2q)}}+t\I)^{-1}
{\X^{(2q)}}^*.\notag
\end{align}
Let $\F_t=\prod_{q=1}^m(\X^{(2q-1)}(\X^{(2q)})_t^{-1})$ 
and $\W_t=\F_t\F_t^*$, $t \ge 0$.
Also, let $\R_t(z) := (\W_t - z\I)^{-1}$
and $s_t(z) := \frac1n \Tr \R_t(z)$, $t \geq 0$, $z \in \mathbb{C}^+$.

\begin{lem}\label{smooth10}
\begin{equation}\notag
\lim_{t\to0}\limsup_{n\to\infty}|s_t(z)-s_0(z)|=0 \quad \text{in probability.}
\end{equation}
\end{lem}

\begin{proof}
We apply Lemma \ref{AXB-lemma}.
Let $\Z_t^{(k)} := \X^{(2k-1)} (\X^{(2k)})_t^{-1}$, $k = 1,\hdots,m$ and
$$
s(t_1,\hdots,t_m;z) := \tfrac{1}{n} \Tr\big((\Z^{(1)}_{t_1} \cdots \Z^{(m)}_{t_m})(\Z^{(1)}_{t_1} \cdots \Z^{(m)}_{t_m})^* - z\I\big)^{-1} \,.
$$
Clearly, writing $\mathbf v_k := (t,\hdots,t,0,\hdots,0)$
for the vector consisting of $k$ $t$'s and $m-k$ $0$'s,
we have
\begin{align}
\label{eq:s-telescope}
|s_t(z) - s_0(z)| \leq \sum_{k=1}^{n} | s_k(\mathbf v_k;z) - s_k(\mathbf v_{k-1};z) | \,.
\end{align}
By Lemma \ref{AXB-lemma}, each of the summands converges to zero in probability.
Here we use
(i) the inductive hypothesis
and
(ii) the fact that an arbitrary product of independent spherical matrices 
satisfies Conditions (C0) -- (C2). 
The latter will be checked in the proof of Theorem \ref{spherensprod};
note that the verification does not rely on the results in this section.
This com\-pletes the proof of Lemma \ref{smooth10}.
\end{proof}

We continue with the proof of Theorem \ref{productspherical}.
We may now use Theorem \ref{singularvalueuniversality} for the matrix $\W_t$. 
The Lindeberg condition \eqref{lind} follows from condition \eqref{eq:spherical-UI2}.
The check of the remaining conditions of Theorem \ref{singularvalueuniversality} 
is similar to that in the previous subsection; we~omit the details.
Thus, by a similar argument as in the proof of Theorem \ref{spherical}
(but with Lemma \ref{smooth10} instead of Lemma \ref{inverse10}),
it remains to identify the limiting empirical spectral distribution
of the matrix $\W(\Y)$ in the Gaussian case.

Here we can use the same approach and notation as in the proof of Theorem \ref{spherical}.
Firstly, by asymptotic freeness, Lemma \ref{rectang} and Theorem \ref{MarPas},
the limiting (mean) empirical spectral distribution of the matrices $\W_t(\Y)$ 
is given by $\varrho^{\boxtimes m} \boxtimes \sigma_t^{\boxtimes m}$,
with corresponding $S$-transform $S_{\W_t}(z)=S_{\varrho}^m(z) \cdot S_{\sigma_t}^m(z)$.
Secondly, using Lemma \ref{smooth10}, it follows that
the limiting (mean) empirical spectral distribution of the matrices $\W(\Y)$
is given by $\varrho^{\boxtimes m} \boxtimes \sigma^{\boxtimes m}$,
with corresponding $S$-transform $S_{\W}(z)=S_{\varrho}^m(z) \cdot S_{\sigma}^m(z)$.
Thirdly, by Remarks \ref{MP}~and~\ref{inverse-MP}, we get
\begin{equation}\label{str9}
S_{\W}(z)=(-1)^m\frac{z^{m}}{(z+1)^m}.
\end{equation}
Using the representation \eqref{str9} we may now determine the Stieltjes transform of the asymptotic 
distribution of the eigenvalues of the matrix $\W$ and from here determine the density of its distribution.
Let $g(z)$ denote the Stieltjes transform of the asymptotic distribution of the eigenvalues of
matrix $\W$. By definition of the $S$-transform, we have
\begin{equation}\notag
S_{\W}\big({-}(1+zg(z))\big)=\frac{g(z)}{1+zg(z)}.
\end{equation}
Combining the last two equalities,  we get
\begin{equation}\notag
(-1)^m\frac{(1+zg(z))^m}{(zg(z))^m}=\frac{g(z)}{1+zg(z)}.
\end{equation}
Solving the last equation, we obtain
\begin{equation}\notag
g(z)=-\frac1{z-(-1)^{\frac m{m+1}}z^{\frac m{m+1}}}.
\end{equation}
Since $\im g(z)\ge 0$ we may take here the root 
$(-1)^{\frac m{m+1}}=\cos\frac{\pi m} {m+1}-i\sin\frac{\pi m}{m+1}$.
The density of the distribution of $F_m(x)$ satisfies the equality
\begin{align}\label{eq:psm-density}
p_m(x)=F_m'(x)=\frac{1}{\pi}\lim_{v\to0}\im g(x+iv))=\frac{\sin\frac{\pi m}{m+1}}
{\pi x^{\frac m{m+1}}(x^{\frac2{m+1}}-2x^{\frac1{m+1}}\cos\frac{\pi m}{m+1}+1)}.
\end{align}
For $m=2$, we have
\begin{equation}\notag
p_2(x)=\frac{\sqrt 3}{2\pi x^{\frac23}(x^{\frac23}+x^{\frac13}+1)}.
\end{equation}
Thus Theorem \ref{spherical} is proved completely.
\end{proof}
\subsection{Applications of Theorem \ref{eigenvalueuniversality}: Distribution of eigenvalues} 
In this section we consider applications of Theorem \ref{eigenvalueuniversality}.
These applications rely on the previous applications of Theorem \ref{singularvalueuniversality}.
Actually, we need the universality of the limiting singular value distribution
not only for the matrices $\F_\X$ and $\F_\Y$, but also for all the shifted matrices 
$\F_\X-\alpha\I$ and $\F_\Y-\alpha\I$, with $\alpha \in \mathbb{C}$. 
However, since these extensions are more~and~less straightforward,
we omit the details.

\subsubsection{Circular law}
Let $\X$ denote the $n\times n$ random matrix with independent entries $\frac1{\sqrt n}X_{jk}$ 
such that \eqref{eq:secondmomentstructure-8a} or \eqref{eq:secondmomentstructure-8b} holds.
Assume that the r.v.'s $X_{jk}$ satisfy the condition
\begin{equation}\label{uniformintegr}
 \max_{j,k}\E|X_{jk}|^2\mathbb I\{|X_{jk}|>M\}\to 0,\quad\text{as}\quad M\to\infty.
\end{equation}
Let $\lambda_1,\ldots,\lambda_n$ denote the eigenvalues of the matrix $\X$. 
Denote by $\mu_n$ the empirical spectral distribution of the matrix $\X$.
Then we~have the following result, cf.\@ Girko \cite{Girko:84}, Bai \cite{Bai},
Pan and Zhou \cite{PanZhou}, G\"otze and Tikhomirov \cite{GT:10aop} as well as 
Tao and Vu \cite{TaoVu:10}.

\begin{thm}\label{circlelaw} 
Assume that condition \eqref{uniformintegr} holds. 
Then the measures $\mu_n$ converge weakly in probability 
to the uniform distribution $\mu$ on the unit disc.
\end{thm}
\begin{proof}
To prove Theorem \ref{circlelaw}, we first apply Theorem \ref{eigenvalueuniversality}. 
Then we compute the limit distribution for the Gaussian case using the result 
of Theorem \ref{denseigval}. 

Note that condition \eqref{lind} is implied by condition \eqref{uniformintegr}
and that conditions \eqref{rank} and \eqref{mat3} -- \eqref{mat32a} of Theorem~\ref{singularvalueuniversality} 
have been checked in Subsection \ref{MPLaw}.
It remains to check conditions (C0), (C1) and (C2) of Theorem \ref{eigenvalueuniversality}.
To this end we can use existing bounds for singular values;
see Lemmas \ref{C0}, \ref{C1} and \ref{C2} in Appendix \ref{sec:SmallSingularValues}.
We may therefore apply Theorem \ref{eigenvalueuniversality}.

Let $\V = \V_\Y$ be defined as in Section \ref{sec:stieltjes}, with $\F = \F_\Y = \Y$.
By Proposition \ref{prop:asympfree}, the matrices $\V$ and $\J(\alpha)$ 
are asymptotically free. 
Furthermore, it follows from Theorem \ref{MarPas} 
that the limiting eigenvalue distribution 
of the matrices $\V_\Y$ is given by the semi-circular law. 
We now compute the limiting eigenvalue distribution 
of the matrices $\F_\Y$ using Theorem \ref{denseigval}.
It is well-known, see e.g.\@ Rao and Speicher \cite{SpeicherRao:2007}, Section~3, that
$$
S_{\V}(z)=-\frac1{\sqrt z}.
$$
Thus, the equations \eqref{mainequation10} read
\begin{align*}
 \psi(\alpha)(1-\psi(\alpha))&=|\alpha|^2\varkappa(\alpha)^2,\\
 \varkappa(\alpha)&=\sqrt{1-\psi(\alpha)}.
\end{align*} 
Solving these equations, we get 
\begin{align*}
\label{eq:twosolutions}
\psi(\alpha)=1, \varkappa(\alpha)=0
\qquad
\text{or} 
\qquad
\psi(\alpha)=|\alpha|^2, \varkappa(\alpha)=\sqrt{1-|\alpha|^2}.
\end{align*}
Now recall that 
(i) $\psi(\alpha) \in [0,1]$ by Theorem \ref{eigenvalueuniversality},
(ii) $\varkappa$ is continuous by Remark \ref{rem:limit},
and
(iii) $\varkappa(\alpha) \ne 0$ for $\alpha \approx 0$ by Lemma \ref{lem:nonzero}.
Thus, we obtain a unique solution, namely
\begin{align*}
 \psi=\begin{cases}|\alpha|^2,\quad& u^2+v^2\le1,\\1,\quad& u^2+v^2>1,\end{cases}\\
\varkappa=\begin{cases}\sqrt{1-|\alpha|^2},
\quad& u^2+v^2\le1,\\0,\quad& u^2+v^2>1.\end{cases}
\end{align*}
It follows from here that
\begin{equation}\notag
 u\frac{\partial \psi}{\partial u}+ v\frac{\partial \psi}{\partial v}=
 \begin{cases}2|\alpha|^2,\quad&u^2+v^2\le1,\\0,\quad&u^2+v^2>1.\\\end{cases}
\end{equation}
Using \eqref{density1}, it therefore follows that the density $f(u,v)$ 
of the limiting empirical spectral distribution of the matrix $\Y$ 
is given by the equality
$$
f(u,v)=\begin{cases}\frac1{\pi},\quad &u^2+v^2\le1,\\0,\quad &u^2+v^2>1.\end{cases}
$$
\end{proof}

\subsubsection{Product of Independent Square Matrices} 
Let $m \ge 1$. 
Consider independent random matrices $\X^{(q)}$, $q=1,\ldots,m$ 
with independent entries $\tfrac{1}{\sqrt{n}} X_{jk}^{(q)}$, $1\le j,k\le n$, $q=1,\ldots,m$,
and suppose that \eqref{eq:secondmomentstructure-8a} or \eqref{eq:secondmomentstructure-8b} holds.
Let $\F=\prod_{q=1}^m\X^{(q)}$.
Then we~have the following result, 
cf. Burda, Janik and Waclaw \cite{Burda1} for the Gaussian case 
and G\"otze and Tikhomirov \cite{GT:12} as well as O'Rourke and Soshnikov \cite{SR:10} 
for the general case.

\begin{thm}\label{producteigenvalue}
Let the r.v.'s $X_{jk}^{(q)}$ satisfy the condition
\begin{equation}\label{uniformintegr1}
 \max_{1\le q\le m}\max_{j,k\ge 1}\E|X_{jk}^{(q)}|^2\mathbb I\{|X_{jk}^{(q)}|>M\}\to 0,
 \quad\text{as}\quad M\to\infty.
\end{equation}
Then the empirical spectral distributions $\mu_n$
of the matrices $\F$ converge weakly in probability 
to the measure $\mu$ with Lebesgue density
\begin{equation}\notag
 f(u,v)=\frac1{m\pi (u^2+v^2)^{\frac{m-1}{m}}}\mathbb I\{u^2+v^2\le1\}.
\end{equation}
\end{thm}

\begin{rem}
The limiting measure $\mu$ is the induced measure of the uniform distribution on the unit disc
under the mapping $z \mapsto z^m$.
Consequently, the product of $m$ independent square matrices 
has the same limiting empirical spectral distribution 
as the $m$th power of a single matrix.
\end{rem}

\begin{proof}
The conditions of Theorem \ref{singularvalueuniversality} were checked in the proof of Theorem \ref{products}. 
The condition $(C0)$ of Theorem  \ref{eigenvalueuniversality} for $p=2$ follows from
Lemma 7.2 in \cite{GT:12}, where it is shown that
\begin{equation}\notag
 \frac1n\E\|\F\|_2^2\le C.
\end{equation}
Furthermore, in Lemma 5.1 in \cite{GT:12} it is proved that there exist
positive constants $Q$ and $A$ such that
\begin{equation}\notag
 \Pr\{s_n(\F-\alpha\I)\le n^{-Q}\}\le Cn^{-A}.
\end{equation}
This implies condition $(C1)$ of Theorem \ref{eigenvalueuniversality}. 
Moreover, inequality (5.16) and Lemma 5.2 in \cite{GT:12} together imply that for 
some $0<\gamma<1$, 
for any sequence $\delta_n\to0$
\begin{equation}\notag
 \lim_{n\to\infty}\frac1n\sum_{k=n_1}^{n_2}|\log s_k(\F-\alpha\I)|=0,
\end{equation}
with $n_1=[n-n\delta_n]+1$ and $n_2=[n-n^{\gamma}]$. \pagebreak[1]
This implies condition $(C2)$ of Theorem~\ref{eigenvalueuniversality}. 
According to Theorem \ref{eigenvalueuniversality} we may now consider 
the Gaussian matrices $\Y^{(q)}$, $q=1,\ldots,m$.
Let $$
\V=\begin{bmatrix}\mathbf O&\F_\Y\\\F_\Y^*&\mathbf O\end{bmatrix}.
$$
By Proposition \ref{prop:asympfree}, the matrices $\V$ and $\J(\alpha)$ are asymptotically free.
Moreover, it follows from Remark \ref{MP} and Lemma \ref{rectang}
that the $S$-transform corresponding to the matrices $\W := \F^{}_\Y \F_\Y^*$ is given by
$$
 S_{\W}(z)=\left(\frac{1}{z+1}\right)^m \,.
$$
Thus, the $S$-transform of the~matrix $\V$ is given by the formula
\begin{equation}\notag
 S_{\V}(z)=-\frac1{\sqrt z(1+z)^{\frac{m-1}2}}.
\end{equation}
We rewrite equations (\ref{mainequation10}) for this case:
\begin{align}\notag
 \psi(\alpha)(1-\psi(\alpha))=|\alpha|^2\varkappa^2(\alpha),\notag\\
 \varkappa(\alpha)=\sqrt{1-\psi(\alpha)}(\psi(\alpha))^{-\frac{m-1}{2}}.
\end{align}
Solving this system we find by similar arguments as in the previous subsection that
\begin{align}
 \psi(\alpha)&=\begin{cases}|\alpha|^{\frac2m},&u^2+v^2 \le1,\\1,&u^2+v^2>1,\\\end{cases}\\
 \varkappa(\alpha)&=\begin{cases}|\alpha|^{-\frac{m-1}m}\sqrt{1-|\alpha|^{\frac2m}},&u^2+v^2\le1,\\0,&u^2+v^2>1,\\\end{cases}
\end{align}
 and, for $u^2+v^2\le 1$,
\begin{equation}\notag
 u\frac{\partial \psi}{\partial u}+ v\frac{\partial \psi}{\partial v}
 =\frac{2|\alpha|^{\frac 2m}}m.
\end{equation}
By \eqref{density1}, these relations immediately imply that
\begin{equation}\notag
 f(x,y)=\begin{cases}\frac1{\pi m(x^2+y^2)^{\frac{m-1}{m}}},\quad &x^2+y^2\le1,\\0,
 \quad &x^2+y^2>1,\end{cases}
\end{equation}
and Theorem \ref{producteigenvalue} is proved.
\end{proof}
\subsubsection{Product of Independent Rectangular Matrices}
Let $m\ge1$ be fixed. Let for any $n\ge1$ be given  
integers $n_0=n, n_1\ge n,\ldots, n_{m-1}\ge n$ and $n_m=n$.
Assume that $y_{q}=\lim_{n\to\infty}\frac{n}{n_{q}}\in (0,1]$, $q=1,\ldots,m$.
Note that $y_m=1$.
Consider independent random matrices $\X^{(q)}$ of order $n_{q-1}\times n_{q}$, 
$q=1,\ldots,m$, with independent entries $\frac{1}{\sqrt{n_q}} X_{jk}^{(q)}$
as in \eqref{eq:secondmomentstructure-8a} or \eqref{eq:secondmomentstructure-8b}.
Put $\F=\prod_{q=1}^m\X^{(q)}$.
Then we have the following~result,
see also Burda, Jarosz, Livan, Nowak, Swiech \cite{Burda2} and Tikhomirov \cite{Tikh:12} 
for related results in the Gaussian case and in the general case, respectively.

\begin{thm}\label{productrecteig}
 Assume that the r.v.'s $X_{jk}^{(q)}$ for $q=1,\ldots,m$, and $j=1,\ldots,p_{q-1}$, 
 $k=1,\ldots,p_{q}$ satisfy the condition \eqref{uniformintegr1}. 
 Then the empirical spectral distributions of the matrices $\F$
 weakly converge in probability to the measure $\mu$ with Lebesgue density 
 \begin{align}
 \label{eq:productrecteig-f}
 f(u,v)=\frac{1}{2\pi |\alpha|^2}\left(u\frac{\partial \psi}{\partial u}+v\frac{\partial \psi}{\partial v}\right)\mathbb{I}\{u^2+v^2 \le 1\},
 \end{align}
 where $\alpha=u+iv$ and $\psi(\alpha)$ is given by the unique solution in the interval $[0,1]$
 to the equation
 \begin{align*}
 \psi(\alpha)\prod_{\nu=1}^{m-1}(1-y_{\nu}+y_{\nu}\psi(\alpha))=|\alpha|^2, \qquad |\alpha| \le 1 \,.
 \end{align*}
 In the case $m=2$ this is
 \begin{equation}\notag
  f(u,v)=\frac{1}{\pi\sqrt{(1-y_1)^2+4(u^2+v^2)y_1}}\mathbb I\{u^2+v^2\le1\}.
 \end{equation}
\end{thm}
\begin{proof}
The conditions of Theorem \ref{singularvalueuniversality} were checked in Subsubsection \ref{products}. 
The conditions $(C0)$, $(C1)$ and $(C2)$ of Theorem \ref{eigenvalueuniversality}
may be checked similarly as in the proof of the pre\-vious Theorem;
we omit the details.
To compute the limit measure $\mu$ in the Gaussian~case,
we may use Theorem \ref{denseigval} now.
Using Remark \ref{MP} and Lemma \ref{rectang}, we may show that 
\begin{align}
\label{eq:productrecteig-V}
S_{\V}(z)=-\frac1{\sqrt z}\prod_{\nu=1}^{m-1}\frac1{\sqrt{1+y_{\nu}z}}.
\end{align}
We have used here that $y_m=1$. Inserting \eqref{eq:productrecteig-V} into \eqref{mainequation10}, we get
\begin{align}
 \psi(\alpha)(1-\psi(\alpha))&=|\alpha|^2\varkappa(\alpha)^2,\notag\\
 \varkappa(\alpha)&=\sqrt{1-\psi(\alpha)}\prod_{\nu=1}^{m-1}\frac1{\sqrt{1-y_{\nu}+y_\nu\psi(\alpha)}}\,.\notag
\end{align} 
Solving this system we find that for $|\alpha| \leq 1$,
\begin{equation}\notag
 \psi(\alpha)\prod_{\nu=1}^{m-1}(1-y_{\nu}+y_{\nu}\psi(\alpha))=|\alpha|^2
 \qquad\text{and}\qquad
 \varkappa(\alpha)=0,
\end{equation}
while for $|\alpha| > 1$, 
\begin{equation}\notag
 \psi(\alpha) = 1
 \qquad\text{and}\qquad
 \varkappa(\alpha) = 0.
\end{equation}
We have used here that the function 
$
h(x) := x \prod_{\nu=1}^{m-1} (1-y_\nu+y_\nu x)
$
is strictly increasing on $[0,\infty)$ with $h(0) = 0$ and $h(1) = 1$.
Now \eqref{eq:productrecteig-f} follows immediately from \eqref{density1}.
In~order to check that $f$ is in fact a probability density,
regard $\psi$ and $f$ as functions of $r := \sqrt{u^2+v^2}$.
Then
$$
f(r) = \frac{\psi'(r)}{2 \pi r}, \qquad 0 < r < 1,
$$
and it follows using polar coordinates that
$$
\int_{\{ u^2+v^2<1 \}} f(u,v) \, du \, dv = \int_0^1 2\pi r f(r) \, dr = \int_0^1 \psi'(r) \, dr = \psi(1) - \psi(0) = 1 \,.  
$$

Finally, for $m=2$ and $u^2+v^2\le1$, we get
\begin{equation}\notag
 \psi(\alpha) (1-y_1+y_1\psi(\alpha))=|\alpha|^2
\end{equation}
and therefore
$$
\psi(\alpha)=\frac{-(1-y_1)+\sqrt{(1-y_1)^2+4|\alpha|^2y_1}}{2y_1}.
$$
Hence, on the set $\{ u^2+v^2\le1 \}$, we obtain
$$
f(u,v)=\frac{1}{2\pi|\alpha|^2}\left( u\frac{\partial\psi}{\partial u}+v\frac{\partial\psi}{\partial v} \right) =  \frac{1}{\pi\sqrt{(1-y_1)^2+4|\alpha|^2y_1}} \,.
$$
Theorem \ref{productrecteig} is proved.
\end{proof}

\subsubsection{Spherical Ensemble} 
Let $\X^{(1)}$ and $\X^{(2)}$ be independent $n\times n$ random matrices 
with independent entries $\frac{1}{\sqrt{n}} X_{jk}^{(q)}$ such that
\eqref{eq:secondmomentstructure-8a} or \eqref{eq:secondmomentstructure-8b} holds.
Consider the matrix $\F_\X=\X^{(1)}({\X^{(2)}})^{-1}$. 
Let $\mu_{n}$ denote the empirical spectral distribution 
of the matrix $\F_\X$.
Then we have the following result, cf.\@ Bordenave \cite{Bordenave}.
\begin{thm}\label{spherens}
 Let the r.v.'s $X_{jk}^{(q)}$ for $q=1,2$ and $j,k=1,\ldots,n$ satisfy the condition 
 \eqref{uniformintegr1}. Then the measures $\mu_{n}$ weakly converge in probability 
 to the measure $\mu$ with Lebesgue density
 \begin{equation}\notag
  f(x,y)=\frac1{\pi(1+(x^2+y^2))^2}.
 \end{equation}
\end{thm}
\begin{rem}
 This density corresponds after stereographic projection of the complex plane to the uniform 
 distribution on the sphere. 
\end{rem}
\begin{proof}
We use Remark \ref{spher0}. By Theorem \ref{spherical}, for each $\alpha \in \mathbb{C}$, 
the empirical singular value distributions of the matrices $\F_\X-\alpha\I$ have the same weak limit in probability
as those of the matrices $\F_\Y-\alpha\I$, where $\F_\Y=\Y^{(1)}({\Y^{(2)}})^{-1}$
and $\Y^{(1)}$ and $\Y^{(2)}$ are independent random matrices with independent Gaussian entries.
We check the conditions $(C0)$, $(C1)$, $(C2)$. Fix $\alpha \in \mathbb C$,
and write $\F$ instead of $\F_\X$.
We start with the condition $(C0)$ for some $p<\frac19$. 
According to Theorem~3.3.14, c) in \cite{HornJ}, we have
\begin{equation}\notag
 \frac1n\sum_{k=1}^ns_k^p(\F-\alpha\I)
 \le\frac1n\sum_{k=1}^n s_k^p(\X^{(1)}-\alpha\X^{(2)}) s_k^p((\X^{(2)})^{-1}). 
\end{equation}
Applying H\"older inequality, we get
\begin{equation}\label{g1}
 \frac1n\sum_{k=1}^ns_k^p(\F-\alpha\I)
 \le\Big(\frac1n\sum_{k=1}^ns_k^2(\X^{(1)}-\alpha\X^{(2)})\Big)^{\frac p2}
 \Big(\frac1n\sum_{k=1}^ns_k^{-\frac{2p}{2-p}}(\X^{(2)})\Big)^{\frac{2-p}2}.
\end{equation}
 Note that
 \begin{equation}\notag
  \E\Big(\frac1n\sum_{k=1}^ns_k^2(\X^{(1)}-\alpha\X^{(2)})\Big)\le 1+|\alpha|.
 \end{equation}
This implies that the first factor on the r.h.s.\@ of \eqref{g1} is bounded in probability.
For $p<\frac29$ we have $\beta=\frac{2p}{2-p}<\frac14$.  
Denote by $\mathcal G_n(x)$ the empirical distribution function 
of the squared singular values of the matrix $\X^{(2)}$. 
By the Marchenko--Pastur theorem
\begin{equation}\notag
 \lim_{n\to\infty}\varkappa_n := \lim_{n \to \infty}\sup_x|\mathcal G_n(x)-G(x)|=0\quad\text{in probability},
\end{equation}
where $G(x)$ has Lebesgue density $g(x)=G'(x)=\frac{\sqrt{4-x}}{2\pi\sqrt x}
\mathbb I\{0<x\le4\}$.
Furthermore, let $0 < \gamma < 1$ be as in Lemma \ref{C12}, and let 
$n_1=[n-n^{\gamma}]$ and $n_2=\min\{n_1,[n(1-\varkappa_n)]\}$.
We~have the following decomposition
\begin{equation}\label{g3}
\frac1n\sum_{k=1}^ns_k^{-\beta}(\X^{(2)})
=\frac1n\sum_{k=n_1+1}^ns_k^{-\beta}(\X^{(2)})
+\frac1n\sum_{k=n_2+1}^{n_1}s_k^{-\beta}(\X^{(2)})+
\frac1n\sum_{k=1}^{n_2}s_k^{-\beta}(\X^{(2)}).
\end{equation}
By Lemma \ref{C1}, we have
\begin{equation}\notag
 \lim_{n\to\infty}\Pr\{s_k(\X^{(2)})\le n^{-Q}\}=0.
\end{equation}
This implies that for $p<\frac{2(1-\gamma)}{2Q+1-\gamma}$
\begin{equation}\notag
 \frac1n\sum_{k=n_1+1}^ns_k^{-\beta}(\X^{(2)})\to 0\quad\text{as}
 \quad n\to\infty\quad \text{in probability}.
\end{equation}
Furthermore, by Lemma \ref{C12}, we have
\begin{equation}\label{g5}
 \Pr \Big\{s_k^{-\beta}(\X^{(2)})>c\left(\frac{n-k}{n}\right)^{-\beta} \text{ for some $k=1,\ldots,n_1$} \Big\} \le \exp\{-cn\}.
\end{equation}
Since $\beta < 1$, this implies that
\begin{equation}\notag
 \Pr\Big\{\frac1n\sum_{k=n_2+1}^{n_1}s_k^{-\beta}(\X^{(2)})
> c' \left(\frac{n-n_2}{n}\right)^{1-\beta}\Big\}\leq \exp\{-c'n\}.
\end{equation}
From $\frac {n-n_2}n\le\varkappa_n$ it follows that 
\begin{equation}\notag
 \lim_{n\to\infty}\frac1n\sum_{k=n_2+1}^{n_1}s_k^{-\beta}(\X^{(2)})=0
 \quad\text{in probability}.
\end{equation}
It remains to prove that the last summand on the r.h.s.\@ of \eqref{g3}
is bounded in probability.
Again by Lemma \ref{C12}, with probability $1 - o(1)$, 
we have $s_{n_2}(\X^{(2)}) \geq c\varkappa_n$ and therefore
\begin{equation}\label{g4}
 \frac1n\sum_{k=1}^{n_2}s_k^{-\beta}(\X^{(2)})
 \le \Big|\int_{c\varkappa_n}^{\infty}x^{-\beta/2} \, d(\mathcal G_n(x)-G(x))\Big|
 +\int_0^4x^{-\beta/2}dG(x).
\end{equation}
The last integral on the r.h.s. of \eqref{g4} is bounded for $\beta<1$.
Integrating by parts in the first integral on the r.h.s.\@ of \eqref{g4}, we get
\begin{align}
\Big|\int_{c\varkappa_n}^{\infty}x^{-\beta}\, d(\mathcal G_n(x)-G(x))\Big|
&\le C\varkappa_n^{-\beta} |\mathcal G_n(\varkappa_n)-G(\varkappa_n)| + \int_{c\varkappa_n}^{\infty} \beta y^{-\beta-1} |\mathcal G_n(y)-G(y)| \, dy \notag\\&\le
C\varkappa_n^{1-\beta}.\notag
\end{align}
The last inequalities imply that the last summand on the r.h.s.\@ 
of \eqref{g3} is bounded in probability. 
This concludes the proof of the condition $(C0)$.
The condition $(C1)$ follows from the~bound
\begin{equation}\notag
 s_n(\F-\alpha\I) \ge s_n(\X^{(1)}-\alpha \X^{(2)}) s_1^{-1}(\X^{(2)})
\end{equation}
and Lemmas \ref{C0} and \ref{C11}.
To prove the condition $(C2)$, fix a sequence $(\delta_n)$ 
with $\delta_n \geq n^{-\gamma}$ for all $n$ and $\delta_n \to 0$,
and let $n_2 := n[1-\delta_n]$.
Then, by the arguments for condition $(C1)$ as well as Theorem 3.3.4 in \cite{HornJ}, 
we have, with probability $1 - o(1)$,
\begin{multline*}
    \lim_{n \to \infty} \frac1n\sum_{k=n_2+1}^{n_1} |\log s_k(\F-\alpha\I)| \\
\le \lim_{n \to \infty} \frac1n\sum_{k=n_2+1}^{n_1} |\log s_k(\X^{(1)}-\alpha\X^{(2)})|
 +  \lim_{n \to \infty} \frac1n\sum_{k=n_2+1}^{n_1} |\log s_{n-k+1}(\X^{(2)})| \,.
\end{multline*}
Note that both sums on the r.h.s.\@ converge to zero in probability.
For the first sum, this follows from Lemma \ref{C12} applied to the matrix $\X^{(1)} - \alpha\X^{(2)}$,
\pagebreak[2] while for the second sum, this follows from the observation that we have,
with probability $1 - o(1)$,
 \begin{align}
  \frac1n\sum_{k=n_2+1}^{n_1} |\log s_{n-k+1}(\X^{(2)})|
  &\le \frac1n\sum_{k=1}^{n-n_2} |\log s_k(\X^{(2)})| \notag\\
  &\le \frac1n\sum_{1\le k\le n-n_2\atop s_k\le \delta_n^{-1}}|\log s_k(\X^{(2)})|
  +\int_{\delta_n^{-1}}^{\infty} \log x \, d\mathcal G(x)\notag\\
  &\le \frac{n-n_2}n|\log\delta_n|+\delta_n^2|\log\delta_n| \, 
  \frac1n\sum_{k=1}^n s_{k}^2(\X^{(2)})\notag\\&
  \le \delta_n|\log\delta_n|\left( 1+\frac1n\|\X^{(2)}\|_2^2 \right) \,.\notag
 \end{align}
Combining these estimates, we come to the conclusion that $\frac1n\sum_{k=n_2+1}^{n} |\log s_k(\F-\alpha\I)|$ converges to zero in probability, 
i.e.\@ condition (C2) is proved.

\pagebreak[2]

Thus, the conditions (C0), (C1), (C2) have been checked for the matrices $\F_\X$.
For the Gaussian matrices $\F_\Y$, the proof is the same.
We may now apply Theorem \ref{eigenvalueuniversality} to~conclude
that the limiting eigenvalue distributions of the matrices
$\F_\X$ and $\F_\Y$ are the~same (if~existent). Thus, it remains
to compute the limit of the empirical distribution of the eigen\-values 
of the matrix $\F_\Y$.

From now on, let the matrices $\F_t := \F_t(\Y)$ be defined as in the proof of Theorem \ref{spherical}.
We shall use the asymptotic freeness of matrices 
$$
\V_t=\begin{bmatrix}&\mathbf 0&\F_t&\\&\F_t^*&\mathbf 0&\end{bmatrix}
\quad\text{and}\quad
\J(\alpha)=\begin{bmatrix}&\mathbf O&-\alpha\I&\\&-\overline \alpha\I&\mathbf O &\end{bmatrix}.
$$
As we have seen in the proof of Theorem \ref{spherical}, 
the limiting (mean) empirical spectral distribution of the~matrices $\F_t \F_t^*$
is given by $\sigma \boxtimes \varrho_t$, with $S$-transform $S_\sigma \cdot S_{\varrho_t}$.
Here, $\sigma$~and~$\varrho_t$ denote the Marchenko-Pastur distribution
and its induced measure under the mapping $x \mapsto (x+t)^{-1} x (x+t)^{-1}$, respectively.
Thus, the limiting (mean) empirical spectral distribution of the matrices $\V_t$
is given by $\mathcal Q^{-1}(\sigma \boxtimes \varrho_t)$,
where $\mathcal Q$ is as in \eqref{eq:Q-definition},
and the limiting spectral distribution of the matrices $\V_t(\alpha) := \V_t + \J(\alpha)$
is given by $\mathcal Q^{-1}(\sigma \boxtimes \varrho_t) \boxplus T(\alpha)$,
where $T(\alpha)$ is as in Section~\ref{sec:stieltjes}.
By a variant of Lemma \ref{inverse10} for shifted matrices
as well as relations \eqref{eq:boxplus-2} and \eqref{eq:boxtimes-2},
it then follows that the limiting (mean) empirical spectral distributions of the matrices
$\V$ and $\V(\alpha) := \V + \J(\alpha)$ are given by 
$\mathcal Q^{-1}(\sigma \boxtimes \varrho)$
and
$\mathcal Q^{-1}(\sigma \boxtimes \varrho) \boxplus T(\alpha)$,
respectively.

Moreover, the $S$-transform of the limiting eigenvalue distribution of $\F \F^*$ is given by 
\begin{align}
\label{eq:spherical-SW}
S_{\F \F^*}(z)=-\frac z{z+1},
\end{align}
as we have seen in the proof of Theorem \ref{spherical}.
Thus, by \eqref{eq:ST-symmetric}, the $S$-transform of the limiting eigenvalue distribution of $\V$ 
is given by
\begin{align}
\label{eq:spherical-SV}
S_{\V}(z) = i \,.
\end{align}
Finally, by Theorem \ref{freer}, the Stieltjes transform $g_t(z,\alpha)$
associated with the matrices $\V_t(\alpha)$ 
satisfies the system of equations \eqref{mainequation-01}
with $S_\V$ replaced by $S_{\V_t}$.
It therefore follows by continuity that the Stieltjes transform $g(z,\alpha)$
associated with the matrices $\V(\alpha)$ 
satisfies the system of equations \eqref{mainequation-01}.

Thus, the assumptions stated above Assumption \ref{lipschitz} are satisfied,
and we may apply Theorem \ref{denseigval}. Solving now the system
\begin{align}
\psi(\alpha)(1-\psi(\alpha))=|\alpha|^2\varkappa^2(\alpha),\notag\\
\varkappa(\alpha)=1-\psi(\alpha),
\end{align}
we obtain
\begin{equation}\notag
 \psi(\alpha)=\frac{|\alpha|^2}{1+|\alpha|^2},
 \qquad
 \varkappa(\alpha)=\frac{1}{1+|\alpha|^2} 
\end{equation}
and
\begin{equation}\notag
 u\frac{\partial \psi}{\partial u}+ v\frac{\partial \psi}{\partial v}=\frac{2|\alpha|^2}
 {(1+|\alpha|^2)^2}.
\end{equation}
The last equality and equality \eqref{density1} together imply
\begin{equation}\notag
 f(u,v)=\frac1{\pi(1+(u^2+v^2))^2}.
\end{equation}
Thus Theorem \ref{spherens} is proved.
\end{proof}
\subsubsection{Product of Independent Matrices from Spherical Ensemble}

For fixed $m \ge 1$, let $\X^{(q)}$, $q=1,\ldots,2m$, be independent $n \times n$ random matrices
with independent entries $\frac{1}{\sqrt{n}} X_{jk}^{(q)}$.
Suppose that \eqref{eq:secondmomentstructure-8a} or \eqref{eq:secondmomentstructure-8b} holds.
Consider the matrix $\F_m=\F_m(\X)=\prod_{q=1}^m\X^{(2q-1)}(\X^{(2q)})^{-1}$,
and denote by $\mu_n$ its empirical spectral distribution.
\begin{thm}\label{spherensprod}Let the r.v.'s $X_{jk}^{(q)}$ for $q=1,\ldots,2m$ and
$j,k=1,\ldots,n$ satisfy the condition \eqref{uniformintegr1}. Then 
 the measures $\mu_n$ weakly converge in probability to the measure 
 $\mu$ with Lebesgue density
 \begin{equation}\notag
  p(x,y)=\frac1{\pi m(u^2+v^2)^{\frac {m-1}m}(1+(u^2+v^2)^{\frac1m})^2}.
 \end{equation}
 
\end{thm}
\begin{proof} The condition of Remark \ref{spher0} follows from Theorem \ref{productspherical}.
Conditions $(C0)$--$(C2)$ follow from induction principle and the proof of Theorem \ref{spherens}.
Indeed, for $\alpha \in \mathbb{C}$, we have the representation
\begin{align}
  \F_m-\alpha\I 
= \F_{m-1} (\X^{(2m-1)} - \alpha \F_{m-1}^{-1} \X^{(2m)}) (\X^{(2m)})^{-1}
\end{align}
Write $s_1(\mathbf A) \ge \cdots \ge s_n(\mathbf A)$ for the singular values of the matrix $\mathbf A$.
Then, similarly as in \cite{HornJ}, Theorem 3.3.14\,(c),
we~have, for any $k=1,\ldots,n$ and for any function $f$ such that 
$\varphi(t)=f({\rm e}^t)$ is increasing and convex, 
\begin{equation}\notag
 \sum_{j=1}^kf(s_j(\F_m-\alpha\I))\le \sum_{j=1}^k
 f\Big( s_j({\F_{m-1}}) s_j(\X^{(2m-1)}-\alpha\F_{m-1}^{-1}\X^{(2m)}) s_j((\X^{(2m)})^{-1}) \Big).
\end{equation}
Now use similar arguments as in the proof of Theorem \ref{spherens}.
Thus, Theorem \ref{eigenvalueuniversality} and Remark \ref{spher0} are applicable,
and it remains to determine the limiting empirical spectral distribution
in the Gaussian case.

Write $\F = \F(\Y)$ for the products of independent ``Gaussian'' spherical matrices.
To find their limiting empirical spectral distribution,
we use the results from Section \ref{density}.
For~brevity, we give only a formal proof;
it is straightforward (although a bit cumbersome) to make this proof rigorous
by using similar arguments as in the proof of Theorem \ref{spherens}
(using a variant of the regularization lemma \ref{smooth10} this time).
First, we find the $S$-transform $S_{\F\F^*}(z)$ associated with the matrices $\W := \F\F^*$.
Remember that, for any $q=1,\ldots,m$, 
the $S$-transform associated with the matrices 
$\Y^{(q)}{\Y^{(q+1)}}^{-1}({\Y^{(q+1)}}^{-1})^*(\Y^{(q)})^*$ 
is given by
$$
-\frac z{z+1}
$$
by \eqref{eq:spherical-SW}.
Thus, by the multiplicative property of $S$-transform, we formally have
$$
S_{\W}(z)=\Big({-}\frac z{z+1}\Big)^m.
$$
Using \eqref{eq:ST-symmetric}, it~follows that
$$
S_{\V}(z)=i\Big({-}\frac z{z+1}\Big)^{\frac{m-1}{2}}.
$$
Solving now the system
\begin{align}
 \psi(1-\psi)=|\alpha|^2\varkappa^2,\notag\\
 \varkappa=(1-\psi)(\frac {1-\psi}{\psi})^{\frac {m-1}{2}},
\end{align}
we find that
\begin{equation}\notag
 \psi=\frac{|\alpha|^{\frac 2m}}{1+|\alpha|^{\frac 2m}}
\end{equation}
and
\begin{equation}\notag
 u\frac{\partial \psi}{\partial u}+ v\frac{\partial \psi}{\partial v}
 =\frac{2|\alpha|^{\frac 2m}}{m(1+|\alpha|^{\frac 2m})^2}.
\end{equation}
The last equality and equality (\ref{density1}) together imply
\begin{equation}
 f(u,v)=\frac1{\pi m(u^2+v^2)^{\frac {m-1}m}(1+(u^2+v^2)^{\frac1m})^2}.
\end{equation}
Thus Theorem \ref{spherensprod} is proved.
\end{proof}

\appendix

\section{Appendix}
\subsection{Variance of Stieltjes Transforms}
\label{sec:variance}

In this subsection, $\F = \F_{\X}$ is defined as in Section \ref{sec:notation}, 
and $\V$ and $\R$ are the Hermitian matrices defined by
 $$
 \V=\begin{bmatrix}&\mathbf O&\F+\mathbf B&\\
 &{(\F+\mathbf B)}^*&\mathbf O&\end{bmatrix},\quad
 \R=(\V-z\I)^{-1},
 $$ 
where $\mathbf B$ is a non-random matrix and $z = u + iv$ with $v > 0$.

\begin{lem}\label{variance}
Suppose that the {\it rank condition} \eqref{rank} holds. 
Then
\begin{equation}\notag
 \E\Big|\frac1n\Tr\R-\E\frac1n\Tr\R\Big|^2\le \frac{C}{nv^2}.
\end{equation}
\end{lem}
\begin{proof}
We introduce the $\sigma$-algebras $\textfrak M_{q,j}=\sigma\{X^{(q)}_{lk},\,
j< l\le n_{q-1}, k=1,\ldots,n_{q}; X^{(r)}_{pk}$,
$r=q+1,\ldots m, \,p=1,\ldots,n_{r-1},\, k=1,\ldots,n_{r}\}$ and use the representation
\begin{equation}\notag
\Tr\R-\E\Tr\R=\sum_{q=1}^m\sum_{j=1}^{n_{q-1}}(\E_{q,j-1}\Tr\R
-\E_{q,j}\Tr\R),
\end{equation}
where $\E_{q,j}$ denotes  conditional expectation given the  
$\sigma$-algebra $\textfrak M_{q,j}$.
 Note that $\textfrak M_{q,n_{q-1}}=\textfrak M_{q+1,0}$. 
 Furthermore, we introduce the matrices $\X^{(q,j)}$ 
 obtained from $\X^{(q)}$ by replacing the entries 
 $X_{jk}^{(q)}$ ($k=1,\hdots,n_q$) by zero's. Define the matrices
 $$
 \F^{(q,j)}=\mathbb F(\X^{(1)},\ldots,\X^{(q-1)},\X^{(q,j)},
 \X^{(q+1)},\ldots,\X^{(m)})
 $$
 and 
 $$
 \V^{(q,j)}=\begin{bmatrix}&\mathbf O&\F^{(q,j)}+\mathbf B&\\
 &(\F^{(q,j)}+\mathbf B)^*&\mathbf O&\end{bmatrix},\quad
 \R^{(q,j)}=(\V^{(q,j)}-z\I)^{-1}.
 $$ 
 Note that
 $\E_{q,j}\Tr \R^{(q,j)}=\E_{q,j-1}\Tr \R^{(q,j)}$, and we may write
 $$
 \Tr\R-\E\Tr\R=\sum_{q=1}^m\sum_{j=1}^{n_{q-1}}
 \Big(\E_{q,j-1}(\Tr\R-\Tr \R^{(q,j)})-\E_{q,j}(\Tr\R-\Tr\R^{(q,j)})\Big),
 $$
 and
 $$
 \E|\Tr\R-\E\Tr\R|^2=\sum_{q=1}^m\sum_{j=1}^{n_{q-1}}
 \E\Big|\E_{q,j-1}(\Tr\R-\Tr \R^{(q,j)})-\E_{q,j}(\Tr\R-\Tr \R^{(q,j)})\Big|^2.
 $$
 By the rank inequality of Bai, we have
 \begin{equation}\notag
  \Big|\Tr\R-\Tr\R^{(q,j)}\Big|\le 
  \frac{\text{\rm rank}\{\V-\V^{(q,j)}\}}{v}.
 \end{equation}
 By the {\it rank condition} \eqref{rank}, we have
 \begin{equation}\notag
  \text{\rm rank}\{\V-\V^{(q,j)}\}\le 2\text{\rm rank}\{\F-\F^{(q,j)}\} 
  \le 2 C(\mathbb F) \text{\rm rank}\{\X^{(q)}-\X^{(q,j)}\}\le 2 C(\mathbb F) .
 \end{equation}
This concludes the proof of Lemma \ref{variance}.
\end{proof}

\subsection{\textit{S}-Transform for Rectangular Matrices}

The $S$-transform of a compactly supported probability distribution 
on $\Real$ was introduced in Section~\ref{freeprob}.
It is well-known that for free random variables $\xi,\eta\ge0$,
with $\xi,\eta \ne 0$,
$$
S_{\xi\eta}(z)=S_{\eta}(z)S_{\xi}(z).
$$
We may interpret this equality for random matrices as follows.
For each $n \in \mathbb N$, let $\X_n$ and $\Y_n$ 
be two random square matrices of size $n\times n$.
Assume that the matrices $\X_n^*\X_n$ and $\Y_n\Y_n^*$
are asymptotically free and that the mean empirical spectral distributions
of the matrices $\X_n \X_n^*$ and $\Y_n \Y_n^*$ converge in moments 
to compactly supported probability measures $\mu$ and $\nu$, respectively,
with $\mu,\nu \ne \delta_0$.
Then the mean empirical spectral distributions of the matrices
$\X_n\Y_n\Y_n^*\X_n^*$ converge in moments
to the probability measure $\mu\boxtimes\nu$ and 
\begin{align}
\label{eq:square}
S_{\mu\boxtimes\nu}(z)=S_{\mu}(z)S_{\nu}(z).
\end{align}
In the case of rectangular matrices this relation is not true anymore.
The next lemma gives the correct relation for products of rectangular matrices.

\begin{lem}\label{rectang}
For each $n \in \mathbb N$, let $\X_n$ and $\Y_n$ be two rectangular random matrices
of the~sizes $n \times p_n$ and $p_n \times \widetilde{p}_n$, respectively,
and assume that $y=\lim_{n \to \infty} \frac {n}{p_n} \in (0,\infty)$.
Assume additionally that the matrices $\X_n^*\X_n$ and $\Y_n\Y_n^*$
are asymptotically free and that the mean empirical spectral distributions
of the matrices $\X_n\X_n^*$ and $\Y_n\Y_n^*$ converge in moments 
to compactly supported probability measures $\mu$ and $\nu$, respectively,
with $\mu,\nu \ne \delta_0$.
Denote by $S_{\mu}$ and $S_{\nu}$ the corresponding $S$-transforms. 
Then the mean empirical spectral distributions of the matrices $\X_n\Y_n\Y_n^*\X_n^*$ 
converge in moments to a probability measure $\xi$,
and the corresponding $S$-transform $S_\xi$ is equal to
$$
S_{\xi}(z)=S_{\mu}(z)S_{\nu}(zy).
$$
\end{lem}

\begin{proof}
By slight abuse of notation, given a sequence $(\mathbf A_n)_{n \in \mathbb N}$ of self-adjoint matrices, 
we denote by $\mu_{\mathbf A}$ the limiting eigenvalue distribution 
and by $S_{\mathbf A}(z)$ the corresponding $S$-transform (if~existent).
For definiteness, assume that $n\le p_n$.
It is easy to see that
$$
\mu_{\Y\Y^*\X^*\X}=y\mu_{\X\Y
\Y^*\X^*}
+(1-y)\delta_0
$$
where $\delta_0$ denotes the unit atom at zero.
From here it follows that 
\begin{equation}\notag
S_{\Y\Y^*\X^*\X}(z)=\frac{z+1}{z+y}S_{\X\Y\Y^*\X^*}(\frac zy).
\end{equation}
We may rewrite this equality as follows
\begin{equation}\label{key}
S_{\X\Y\Y^*\X^*}(z)=\frac{y(z+1)}{yz+1}S_{\Y\Y^*
\X^*\X}(zy).
\end{equation}
By asymptotic freeness and the multiplicative property of the $S$-transform, we have
$$
S_{\Y\Y^*\X^*\X}(zy)=S_{\Y\Y^*}(zy)S_{\X^*\X}(zy).
$$
(In particular, the limiting eigenvalue distribution $\mu_{\Y\Y^*\X^*\X}$ exists.)
By the same argument as for (\ref{key}), we get
$$
S_{\X\X^*}(z)=\frac{y(z+1)}{yz+1}S_{\X^*\X}(zy).
$$
The three last equalities together imply the result of Lemma.
\end{proof}

\begin{rem}
Using the preceding results, it is easy to see why
the $m$th power of a random square matrix $\Y_n$ 
and the product $\Y^{(1)}_n \cdots \Y^{(m)}_n$
of $m$ independent copies of this matrix
should have the same limiting singular value and eigenvalue distributions.
Indeed, let $\Y_n$ be bi-unitary invariant random square matrices such that 
the empirical spectral distribution of the matrices $\Y_n\Y^*_n$ 
converges weakly in probability as well as in moments 
to a~compactly supported probability measure $\mu_{\Y\Y^*}$.

\pagebreak[2]

Then, similarly as in Hiai and Petz \cite{Petz-1},
using the singular value decomposition of the matrix $\Y_n$,
one can show that $\Y_n \Y^*_n$ and $\Y^*_n \Y_n$ are asymptotically free,
and it follows from Equation \eqref{eq:square} (and induction)
that $\Y^m_n (\Y^m_n)^*$ converges in moments to $\mu_{\Y\Y^*}^{\boxtimes m}$.
A similar argument, also based on Equation \eqref{eq:square},
shows that the same is true for 
$(\Y^{(1)}_n \cdots \Y^{(m)}_n)(\Y^{(1)}_n \cdots \Y^{(m)}_n)^*$,
where $\Y^{(1)}_n,\hdots,\Y^{(m)}_n$ are $m$ independent copies of $m$.
Thus, the matrices $\Y^m_n$ and $\Y^{(1)}_n \cdots \Y^{(m)}_n$
will have the same limiting singular value distributions.

Now the {\it $S$-transform} of the limiting singular value distribution
of the shifted matrices $\F-\alpha\I$ is well defined by
$\alpha$ and the limiting singular value distribution of the matrix $\F$. 
To prove this we must use the additive property of the {\it $R$-transform}
and the correspondence between {\it $R$-} and {\it $S$-}{\it transforms}.
Furthermore, note that the limit measure for the eigenvalue distribution is well defined 
by its logarithmic potential and that we may reconstruct the logarithmic potential from
the family of the singular value distribution of the shifted matrices. 
It therefore follows that the limiting eigenvalue distribution 
of the matrices $\Y^m_n$ and $\Y^{(1)}_n \cdots \Y^{(m)}_n$ will also be the same.


But the eigenvalues of the $m$th power of a matrix 
are the $m$th powers of the eigenvalues of that matrix.
For example, if the limiting eigenvalue distribution of the random matrix $\Y_n$ is the circular law, 
then the limiting eigenvalue distribution of the product $\Y^{(1)}_n\cdots\Y^{(m)}_n$ 
is the $m$th power of the uniform distribution in the unit disc.
\end{rem}

\subsection{Bounds on Singular Values}
\label{sec:SmallSingularValues}

Throughout this subsection, let $\X$ denote an $n\times n$ random matrix with independent entries 
$\frac{1}{\sqrt n}X_{jk}$ such that $\E X_{jk}=0$ and $\E|X_{jk}|^2=1$. 
Let $s_1(\X) \ge\cdots\ge s_n(\X)$ denote the singular values of the matrix $\X$.
Then we have the following result:

\begin{lem}\label{C0}
We have $\lim_{t \to \infty} \limsup_{n \to \infty} \Pr\{ \tfrac1n \sum_{k=1}^{n} s_k^2(\X) \geq t \} = 0$.
\end{lem}

\begin{proof}
This follows from the observation that
\begin{align*}
 \E\left(\frac1n \sum_{k=1}^ns_k^2(\X)\right)=\frac1n\E\|\X\|_2^2=\frac{1}{n^2}\sum_{j,k}\E|X_{jk}|^2=1
\end{align*}
and Markov's inequality.
\end{proof}

Henceforward, assume additionally that the r.v.'s $X_{jk}$ satisfy the condition
\begin{equation}\label{eq:UI2}
\max_{j,k}\E|X_{jk}|^2\mathbb I\{|X_{jk}|>M\}\to 0,\quad\text{as}\quad M\to\infty.
\end{equation}
Under these assumptions, we have the following bounds on the small singular values,
see G\"otze and Tikhomirov, \cite{GT:10aop}, Theorem 4.1 and \cite{GT:12}, Lemma 5.2 and Proposition 5.1. 
(For~the i.i.d.\@ case similar results were obtained by Tao and Vu, \cite{TaoVu:10}, Lemma 4.1 and 4.2.)
Let $s_1(\X-\alpha\I) \ge\cdots\ge s_n(\X-\alpha\I)$ denote the singular values of the matrix $\X-\alpha\I$.

\begin{lem}\label{C1}
 Suppose that condition \eqref{eq:UI2} holds. Then, for~any fixed $\alpha \in \mathbb{C}$,
 there exist positive constants $Q$ and $B$ such that
 \begin{equation}\notag
  \Pr\{s_n(\X-\alpha\I)\le n^{-Q}\}\le n^{-B}.
 \end{equation}
\end{lem}
For a proof of this lemma see the proof of Theorem 4.1 in \cite{GT:10aop}.

\begin{lem}\label{C2}
 Suppose that condition \eqref{eq:UI2} holds. Then, for any fixed $\alpha \in \mathbb{C}$,
 there exists a constant $0<\gamma<1$ such that for any sequence $\delta_n \to0 $, 
 \begin{equation}\notag
  \lim_{n\to\infty}\frac1n\sum_{{n_1}\le j\le{ n_2} }\ln s_j(\X-\alpha\I)=0 \quad \text{almost surely},
 \end{equation}
 with $n_1=[n-n\delta_n]+1$ and $n_2=[n-n^{\gamma}]$.
\end{lem}
For a proof of this lemma see the proof of inequality (5.17) and Lemma 5.2 in \cite{GT:12}. \pagebreak[2]

For the investigation of the spherical ensembles, we need the following extensions 
of these results; see Equations (5.9) and (5.17) in \cite{GT:12}.

\begin{lem}\label{C11}
 Suppose that condition \eqref{eq:UI2} holds. Then, for~any $K>0$ and $L>0$,
 there exist positive constants $Q$ and $B$ such that
 for any non-random matrix $\M$ with $\|\M\|_2 \leq Kn^L$, we have
 \begin{equation}\notag
  \Pr\{s_n(\X-\M)\le n^{-Q}\}\le n^{-B}.
 \end{equation}
\end{lem}

\begin{lem}\label{C12}
 Suppose that condition \eqref{eq:UI2} holds. Then, for any fixed $K>0$ and $L>0$,
 there exist constants $0<\gamma<1$ and $c>0$ such that for any non-random matrix $\M$
 with $\|\M\|_2 \leq Kn^{L}$, we have
 \begin{equation}\notag
  \Pr \Big\{ s_{j}(\X-\M) \geq c \frac{n-j}{n} \quad \text{for all $j=1,\hdots,n-n^\gamma$} \Big\}
  \geq 1 - \exp(-n^\gamma) \,.
 \end{equation}
\end{lem}

\subsection{Technical Details for Section \ref{density}} \label{sec:TechnicalDetails}

In this subsection, we state some technical lemmas which have been used in Section \ref{density}.

\medskip

Using similar arguments as in the proof of Theorem \ref{eigenvalueuniversality} / Remark \ref{spher0},
we may prove the following.

\begin{lem}\label{lem:logintegrability}
Assume that the matrices $\F_{\Y}$ satisfy the conditions $(C0)$, $(C1)$ and $(C2)$.
Moreover, assume that the singular value distributions of the matrices $\F_{\Y}$
converge weakly in probability to a non-random probability measure $\nu$.
Then the logarithm is integrable w.r.t.\@ $\nu$, and we have
$$
\lim_{n \to \infty} \left( \tfrac1n \sum_{k=1}^{n} \log s_k(\F_\Y) \right) \to \int_0^\infty \log(x) \, d\nu(x) 
\quad\text{in probability}
$$
as well as
$$
\lim_{n \to \infty} \left( \tfrac1n \sum_{k=1}^{n} |\log s_k(\F_\Y)| \right) \to \int_0^\infty |\log(x)| \, d\nu(x) 
\quad\text{in probability.}
$$
\end{lem}

\begin{proof}
Clearly, the main problem is to show that the logarithm is integrable w.r.t.\@ $\nu$.
Once this is shown, it is straightforward to adapt the proof of Theorem \ref{eigenvalueuniversality},
replacing the singular value distributions of the matrices $\F_\Y$
with the fixed distribution $\nu$.
We will show separately that $\log^+ := \max \{ +\log,0 \}$ and $\log^- := \max \{ -\log,0 \}$ 
are integrable w.r.t.\@ $\nu$.
In doing so, we write $\nu_n$ for the singular value distribution of $\F_\Y$.

Let us begin with the positive part.
First of all, passing to a suitable subsequence,
we~may assume w.l.o.g. that $\nu_n \Rightarrow \nu$ almost surely.
Now, by assumption (C0), there exists a constant $K > 0$ such that
$$
\Pr \left\{ \int x^p \, d\nu_n \geq K \right\} \leq \tfrac12
$$
for all sufficiently large $n \in \mathbb N$. This implies that the event
$$
A := \left\{ \omega : \int x^p \, \nu_n(\omega,dx) \leq K \ \text{infinitely often} \right\}
$$
has positive probability. Fix $\omega \in \Omega$ such that $\nu_n(\omega) \Rightarrow \nu$
and a subsequence $(\nu_{n_k}(\omega))$ such that $\int x^p \, \nu_{n_k}(\omega,dx) \leq K$ for all $k \in \mathbb N$.
Then $x^{p/2}$, and hence $\log^+(x)$, is uniformly integrable w.r.t. $(\nu_{n_k}(\omega))$,
and it follows that $\log^+ \in L^1(\nu)$.

Let us now consider the negative part.
Again, we may select a subsequence $(\nu_{n_k})$
such that $\nu_{n_k} \Rightarrow \nu$ almost surely.
Moreover, using monotone convergence and weak~convergence, we have
\begin{multline*}
\int \log^- \, d\nu 
= \lim_{a \to \infty} \int (\log^- \wedge\, a) \, d\nu \\
= \lim_{a \to \infty} \lim_{k \to \infty} \int (\log^- \wedge\, a) \, d\nu_{n_k}
\leq \liminf_{k \to \infty} \int \log^- \, d\nu_{n_k}
\qquad\text{almost surely.}
\end{multline*}
Suppose by way of contradiction that $\int \log^- \, d\nu = \infty$.
Then, for any $l \in \mathbb N$, we~have
$$
\lim_{k \to \infty} \int_0^{1/l} \log^- \, d\nu_{n_k} = \infty \qquad\text{almost surely.}
$$
Thus, for any $l \in \mathbb N$, we~may~find an index $k(l)$ such that
$$
\int_{0}^{1/l} \log^- \, d\nu_{n_{k(l)}} \geq l \quad \text{with probability} \geq 1 - 2^{-l} \,.
$$ 
We may assume w.l.o.g. that the sequence $(k(l))$ is increasing.
Thus, we obtain a~sub\-sequence (which we again denote by $\nu_{n_k}$, by abuse of notation)
such that
\begin{align}
\label{eq:nullsequence}
\int_{0}^{1/k} \log^- \, d\nu_{n_k} \geq k
\end{align}
for almost all $k \in \mathbb N$.
We now proceed similarly as in the proof of Theorem \ref{eigenvalueuniversality}.
Put $\delta_k := 1/\log k$,
$n_1(k) := [n_k - \delta_k n_k] + 1$ and $n_2(k) := [n_k - n_k^\gamma]$.
(We~write $n_1(k)$ and $n_2(k)$ here to emphasize the dependence on~$k$.)
Then, by the same arguments as in the proof of Theorem \ref{eigenvalueuniversality},
the probability of the following events tends to~zero: 
$$
\{ s_{n_k}(\F_\Y) < n_k^{-Q} \} \,,\
\Big\{ \tfrac{1}{n_k} \sum_{n_1(k) \leq j \leq n_2(k)} |\log s_j(\F_\Y)| > 1 \Big\} \,,\
\{ s_{n_1(k)}(\F_\Y) \leq \tfrac{1}{k} \} \,.
$$
Thus, we may select a sequence $(k(l))$
such that with probability $1$, we have
$$
s_{n_{k(l)}}(\F_\Y) \geq n_{k(l)}^{-Q} \,,\
\tfrac{1}{n_{k(l)}} \sum_{n_1(k(l)) \leq j \leq n_2(k(l))} |\log s_j(\F_\Y)| \leq 1 \,,\
s_{n_1(k(l))}(\F_\Y) > \tfrac{1}{k(l)}
$$
for almost all $l \in \mathbb N$.
It then follows using \eqref{eq:nullsequence} that with probability $1$, we have
\begin{multline*}
k(l) \leq \int_0^{1/k(l)} \log^- \, d\nu_{n_{k(l)}}
\leq \frac{1}{n_{k(l)}} \sum_{j=n_1(k(l))}^{n_2(k(l))-1} |\log s_j(\F_\Y)| \\ + \frac{1}{n_{k(l)}} \sum_{j=n_2(k(l))}^{n_{k(l)}} |\log s_j(\F_\Y)| 
\leq 1 + \frac{1}{n_{k(l)}} (n_{k(l)}^\gamma + 1) Q \log n_{k(l)} 
\end{multline*}
for almost all $l \in \mathbb N$. But this is a contradiction. 
We therefore come to the conclusion that~$\log^- \in L^1(\nu)$.
\end{proof}

We now prove Lemma \ref{lem:support}.
For convenience, we repeat the~statement of the lemma.

\begin{lem}[= Lemma \ref{lem:support}]
Assumptions \ref{continuity} and \ref{lipschitz} hold for probability measures $\mu_\V$ 
such that $\mu_\V([-x,+x]^c) = \mathcal O(x^{-\eta})$ $(x \to \infty)$ for some $\eta > 0$.
\end{lem}

\def\lest{\le_{\qopname\relax{no}{st}}}

\begin{proof} The proof consists of several parts.
We will use the fact that the free additive convolution is monotone 
with respect to stochastic order $\lest$ 
(see e.g. Proposition 4.16 in Bercovici and Voiculescu \cite{BercVoic}),
i.e. we~have
\begin{align}
\label{eq:monotonicity}
\mu_1 \lest \mu_2 \ \wedge \ \nu_1 \lest \nu_2 \quad \Rightarrow \quad \mu_1 \boxplus \mu_2 \lest \nu_1 \boxplus \nu_2 \,.
\end{align}

\paragraph{Preliminary Estimates.}
It follows from \eqref{eq:monotonicity} that $\mu_{\V}([-x,+x]^c) = \mathcal O(x^{-\eta})$ ($x \to \infty$)
implies $\mu_{\V(\alpha)}([-x,+x]^c) = \mathcal O(x^{-\eta})$ ($x \to \infty$),
where the $\mathcal O$-bound is locally uniform in~$\alpha$.
Thus, using integration by parts,
we find that for any continuously differentiable (possibly complex-valued) function $f$
such that $f'(x) = \mathcal O(|x|^{-1})$ as $|x| \to \infty$, we have
$$
  \int_{\Real} f(x) \, d\mu_{\V(\alpha)}(x)
= f(0) + \int_0^\infty f'(y) (1 - \mathcal F_{\V(\alpha)}(y)) \, dy - \int_{-\infty}^0 f'(y) \mathcal F_{\V(\alpha)}(y) \, dy \,,
$$
where $\mathcal F_{\V(\alpha)}$ denotes the distribution function of $\mu_{\V(\alpha)}$.
Therefore, for any $\alpha,\beta \in \Real^2$,
\begin{align*}
   \int_{\Real} f(x) \, d\mu_{\V(\alpha)}(x) - \int_\Real f(x) \, d\mu_{\V(\beta)}(x)
&= \int_{\Real} f'(y) (\mathcal F_{\V(\beta)}(y) - \mathcal F_{\V(\alpha)}(y)) \, dy \,.
\end{align*}
Suppose w.l.o.g. that $|\alpha| \le |\beta|$, and set $m := \frac{|\alpha|+|\beta|}{2}$, $\varepsilon := |\beta|-|\alpha|$
and $\xi := \mu_\V \boxplus T(m)$. Then, by \eqref{eq:monotonicity}, we have 
$$
\xi \boxplus \delta_{-\varepsilon/2} = \mu_\V \boxplus \tfrac{1}{2} (\delta_{-|\beta|} + \delta_{+|\alpha|})
\lest \mu_{\V(\alpha)}, \mu_{\V(\beta)}
\lest \mu_\V \boxplus \tfrac{1}{2} (\delta_{-|\alpha|} + \delta_{+|\beta|}) = \xi \boxplus \delta_{+\varepsilon/2}
$$
and therefore
$$
|\mathcal{F}_{\V(\beta)}(x) - \mathcal{F}_{\V(\alpha)}(x)| \leq \mathcal{F}_\xi(x+\tfrac12\varepsilon) - \mathcal{F}_\xi(x-\tfrac12\varepsilon) \,,
$$
It follows that
\begin{align*}
     \left| \int_{\Real} f'(y) (\mathcal F_{\V(\beta)}(y) - \mathcal F_{\V(\alpha)}(y)) \, dy \right|
\leq \|f'\|_\infty \int_{\Real} \Big(\mathcal{F}_\xi(y+\tfrac12\varepsilon) - \mathcal{F}_\xi(y-\tfrac12\varepsilon)\Big) \, dy
     = \varepsilon \,.
\end{align*}
Combining these inequalities, we find that for any $\alpha,\beta \in \Real^2$
and for any function $f$ with the above-mentioned properties, we have
\begin{align}
\label{eq:newdifference}
      \left| \int_\Real f(x) \, d\mu_{\V(\alpha)}(x) - \int_\Real f(x) \, d\mu_{\V(\beta)}(x) \right|
&\leq |\alpha - \beta| \|f'\|_\infty \,.
\end{align}
In particular, the integral $\int f(x) \, d\mu_{\V(\alpha)}(x)$ is continuous in $\alpha$.

\paragraph{Proof of Continuity.}
By general properties of the Stieltjes transform, the function $g(iy,\alpha)$
is locally uniformly continuous in $y$, uniformly in $\alpha$.
Thus, it remains to show that the function $g(iy,\alpha)$ is continuous in $\alpha$.
This follows by taking $f(x) := \frac{1}{x-iy}$ in \eqref{eq:newdifference}, 
with $y > 0$ fixed. 

\paragraph{Proof of Differentiability.}
By general properties of the Stieltjes transform, the function $g(iy,\alpha)$
is differentiable with respect to $y$, with derivative
$$
\frac{\partial g}{\partial y}(iy,\alpha) = \int \frac{i}{(x-iy)^2} \, d\mu_{\V(\alpha)}(x) \,.
$$
It therefore follows by the same arguments as in the preceding paragraph
that $\frac{\partial g}{\partial y}(iy,\alpha)$ is continuous.

The argument for $\frac{\partial g}{\partial u}(iy,\alpha)$ is a bit longer,
and we confine ourselves to a rough sketch.
It is straightforward to see \pagebreak[2]
(e.g.\@ by using the additivity of the \emph{Voiculescu transform},
\linebreak[2] see e.g.\@ Corollary 5.7 in \cite{BercVoic})
that the Stieltjes transform $g(z,\alpha)$
is locally analytic in $(z,\alpha)$ around the point $(z_0,\alpha_0)$, 
for $\alpha_0 \ne 0$ fixed and $z_0 \in \mathbb C^+$ with $\im z_0$ sufficiently large.
Thus, we \emph{locally} have the power series expansions
\begin{align*}
   g(z,\alpha) 
&= \sum_{k=0}^{\infty} \sum_{j=0}^{\infty} c_{jk} (\alpha-\alpha_0)^j (z-z_0)^k
=: \sum_{k=0}^{\infty} c_k(\alpha) (z-z_0)^k \,,
\\
  \frac{\partial g}{\partial u}(z,\alpha) 
&= \sum_{k=0}^{\infty} \sum_{j=0}^{\infty} c_{jk} \, j (\alpha-\alpha_0)^{j-1} (z-z_0)^{k}
=: \sum_{k=0}^{\infty} \widetilde{c}_k(\alpha) (z-z_0)^k \,.
\end{align*}
Suppose that the bivariate power series converge on the set of all $(z,\alpha)$
with $|z-z_0| < \varepsilon$ and $|\alpha-\alpha_0| < \varepsilon$,
where $\varepsilon = \varepsilon(z_0,\alpha_0) > 0$. \pagebreak[1]

Let us investigate the growth of the coefficients,
and hence the radius of convergence,
of the \emph{univariate} power series in $z$.
Since 
$$
  |c_k(\alpha)| 
= \left| \frac{1}{k!} \frac{\partial^k}{\partial z^k} g(z_0,\alpha) \right|
= \left| \frac{1}{k!} \int \frac{k!}{(t-z)^{k+1}} \, d\mu_{\V(\alpha)}(dt) \right|
\leq \frac{1}{(\im z_0)^{k+1}} \,,
$$
$$
  |\widetilde{c}_k(\alpha)| 
= \left| \frac{1}{k!} \frac{\partial^k}{\partial z^k} \frac{\partial}{\partial u} g(z_0,\alpha) \right|
= \left| \frac{1}{k!} \frac{\partial}{\partial u} \frac{\partial^k}{\partial z^k} g(z_0,\alpha) \right|
\leq \frac{k+1}{(\im z_0)^{k+1}} \,,
$$
where the last estimate follows from Equation \eqref{eq:newdifference},
the two power series have radius of convergence $\geq \im z_0$.
Furthermore, since the functions $c_k(\alpha)$ and $\widetilde{c}_k(\alpha)$ 
are continuous in $\alpha$
and, for any fixed $\delta > 0$, the two power series converge uniformly
for $|z-z_0| < (\im z_0) - \delta$ and $|\alpha - \alpha_0| < \varepsilon - \delta$,
they~represent continuous functions $f$ and $\widetilde{f}$
defined on the set $B(z_0,\im z_0) \times B(\alpha_0,\varepsilon)$.
(Here, $B(z,r)$ denotes the open ball of radius $r$ around the point $z$.)
Thus, again by uniform convergence, we may conclude that 
the~function $f$ is continuously differentiable with respect to $u$.
\pagebreak[2] Since the function $f(z,\alpha)$ coincides with $g(z,\alpha)$
on the set $B(z_0,\varepsilon) \times B(\alpha_0,\varepsilon)$ (by~construction)
and therefore on the set $B(z_0,\im z_0) \times B(\alpha_0,\varepsilon)$
(by analytic continuation in $z$), this proves our claim about
the~existence and the~continuity of $\frac{\partial g}{\partial u}(iy,\alpha)$.

\paragraph{Proof of \eqref{eq:derivatives}.}
Since for fixed $\alpha$, $g(z,\alpha)$ is a non-constant analytic function 
in a certain open set containing the upper imaginary half-axis,
there exists an at most countable set $Y_\alpha = \{ y_1,y_2,y_3,\hdots \}$
such that for all $y \not\in Y_\alpha$,
\begin{align}
\label{deriv41}
g(iy,\alpha) \ne \frac{i}{2|\alpha|}
\quad\text{and}\quad
\frac{\partial g}{\partial y}(iy,\alpha) \ne 0 \,.
\end{align}
For $y \not\in Y_\alpha$,
differentiating the second equation in \eqref{mainequation} with respect to $y$,
we get
\begin{multline}
\label{deriv42}
-\frac{\partial g}{\partial y}(iy,\alpha) 
= 
\widetilde S_{\V}'\Big({-}(1+iyg(iy,\alpha)-\tfrac12+\tfrac12\sqrt{1+4|\alpha|^2g(iy,\alpha)^2}\,)\Big) \\
\times \left[ {-}ig(iy,\alpha)-iy\tfrac{\partial g}{\partial y}(iy,\alpha)-\frac{2|\alpha|^2g(iy,\alpha)\tfrac{\partial g}{\partial y}(iy,\alpha)}{\sqrt{1+4|\alpha|^2g(iy,\alpha)^2}} \right] \,,
\end{multline}
where $\widetilde S_\V(z) := z S_\V(z)$.

\pagebreak[2]

Now fix $(iy_0,\alpha_0)$ with $\alpha_0 \ne 0$ and $y_0 \not\in Y_{\alpha_0}$.
Then there exists a small neighborhood $N$
such that for $(iy,\alpha) \in N$, we have \eqref{deriv41}, \eqref{deriv42}, and
$$
F(g(iy,\alpha),iy,\alpha) = 0 \,,
$$
where
$$
F(\zeta,iy,\alpha) := \zeta + \widetilde S_{\V}\Big({-}(1+iy\zeta-\tfrac12+\tfrac12\sqrt{1+4|\alpha|^2\zeta^2}\,)\Big)
$$
and the sign of the square-root is constant in $N$.
Note that $F$ is an analytic function with
\begin{multline}
\label{deriv43}
\frac{\partial F}{\partial \zeta}(g(iy,\alpha),iy,\alpha)
=
1 + \widetilde S_{\V}'\Big({-}(1+iyg(iy,\alpha)-\tfrac12+\tfrac12\sqrt{1+4|\alpha|^2g(iy,\alpha)^2}\,)\Big) \\
\times \left[ -iy-\frac{2|\alpha|^2g(iy,\alpha)}{\sqrt{1+4|\alpha|^2g(iy,\alpha)^2}} \right] \,.
\end{multline}
Moreover, comparing \eqref{deriv42} and \eqref{deriv43}, we see that
$$
\frac{\partial F}{\partial \zeta}(g(iy,\alpha),iy,\alpha)
=
\frac{ig(iy,\alpha)}{\frac{\partial g}{\partial y}(iy,\alpha)} \widetilde S_{\V}'\Big({-}(1+iyg(iy,\alpha)-\tfrac12+\tfrac12\sqrt{1+4|\alpha|^2g(iy,\alpha)^2}\,)\Big)
\ne 0 \,.
$$
It therefore follows from the implicit function theorem for real-analytic functions
that there exists a small neighborhood $\widetilde{N} \subset N$ of the point $(iy_0,\alpha_0)$
such that $g(iy,\alpha)$, the~solution to the equation $F(\zeta,iy,\alpha) = 0$, 
is analytic on~$\widetilde N$, with gradient
$$
\frac{\partial g}{\partial(y,u,v)} = - \left( \frac{\partial F}{\partial \zeta} \right)^{-1} \frac{\partial F}{\partial(y,u,v)} \,.
$$
Equation~\eqref{eq:derivatives} now follows by a straightforward calculation.

\paragraph{Existence of continuous extension.}
This follows from Lemma \ref{lem:limit}
and the subsequent Remark \ref{rem:limit}.

\paragraph{Proof of Assumption \ref{lipschitz}.}
Let $K$ be a compact set as in Assumption \ref{lipschitz},
and let $\alpha,\beta \in K$.
For fixed $C > 0$, consider the function $f(y) := \log(1+y^2/C^2)$.
Since $f'(y) = \frac{2y}{C^2+y^2}$, this function satisfies
the conditions of Equation \eqref{eq:newdifference},
and we obtain
$$
\left| \int \log(1+y^2/C^2) d\mu_{\V(\alpha)}(y) - \int \log(1+y^2/C^2) d\mu_{\V(\beta)}(y) \right|
\leq |\alpha-\beta| / C \,,
$$
from which Assumption \ref{lipschitz} follows immediately.
\end{proof}

\pagebreak[2]


\paragraph{Acknowledgement.}
We thank Peter J.\@ Forrester for pointing out some relevant references.


\begin{thebibliography}{12345}
\setlength{\itemsep}{0.2pt}

\bibitem{Akemann-Ipsen-Kieburg}
 Akemann, G.; Ipsen, J.; Kieburg, M.
 {\em Products of rectangular random matrices: Singular values and progressive scattering.}
 Phys. Rev. E, vol. 88, 2013, 052118. 

\bibitem{AGT:10}Alexeev, N.; G\"otze, F.; Tikhomirov, A. N.
 {\em Asymptotic distribution of singular values of powers of random matrices},
 Lithuanian math. J., vol. 50, no.~2, 2010, 121--132.

\bibitem{AGT:10a}Alexeev, N.; G\"otze, F.; Tikhomirov, A. N.
 {\em On the singular spectrum of powers and products of random matrices},
 Doklady Mathematics, vol. 82, no.~1, 2010, 505--507.

\bibitem{AGT:10b}Alexeev, N.; G\"otze, F.; Tikhomirov, A. N. 
 {\em On the asymptotic distribution of singular values of products of large rectangular random matrices.} 
 Preprint, arXiv:1012.2586

\bibitem{AGT:12}Alexeev, N.; G\"otze, F.; Tikhomirov, A. N.
 {\em On the asymptotic distribution of the singular values of powers of random matrices}. (Russian) 
 Zapiski Nauchn. Seminarov POMI, vol. 408, 2012, 9--43.

\bibitem{Perez}Arizmendi, O. E., P\'erez-Abreu, V. 
 {\em The $S$-transform of symmetric probability measures with unbounded supports.}
 Communicati\'on del CIMAT, No 1-08-20/19-11-2008 (PE-CIMAT).

\bibitem{Bai} 
 Bai, Z. D.
 {\em Circular law.}
 Ann. Probab. 25, no.~1, 1997, 494--529.

\bibitem{BS:10} Bai, Z.; Silverstein, J. W. 
 {\em Spectral analysis of large dimensional random matrices.}
 Second edition. Springer Series in Statistics. Springer, New York, 2010. xvi+551 pp. 

\bibitem{BMS:13} Belinschi, S. T., Mai, T., Speicher, R.
 {Analytic subordination theory of operator-valued free additive convolution 
    and the solution of a general random matrix problem}. 
 Preprint, arXiv:1303.3196.

\bibitem{Bentkus} Bentkus, V. 
 {\em A new approach to approximations in probability theory and operator theory.} (Russian) 
  Liet. Mat. Rink. 43, no.~4, 2003, 444--470; 
  translation in Lithuanian Math. J. 43, no.~4, 2003, 367--388.     

\bibitem{BercVoic} Bercovici, H., Voiculescu, D.
 {\em Free convolution of measures with unbounded support.}
 Indiana Univ. Math. J. 42, no.~3, 1993, 733--773.

\bibitem{Biane}
 Biane, P.
 {\em Processes with free increments.}
 Math. Z. 227, 1998, 143--174.

\bibitem{Bordenave} Bordenave, Ch.
 {\em On the spectrum of sum and product of non-Hermitian random matrices.}
 Electron. Commun. Probab. 16, 2011, 104--113.

\bibitem{BordChaf} Bordenave, Ch., Chafa\"{\i}, D. 
 {\em Around the circular law}. 
 Probab. Surv. 9, 2012, 1--89. 

\bibitem{Burda1} Burda, Z., Janik, R. A., Waclaw, B.
 {\em Spectrum of the product of independent random Gaussian matrices.}
 Phys. Rev. E, vol. 81, 2010, 041132.

\bibitem{Burda2} Burda, Z., Jarosz, A., Livan, G., Nowak, M. A., Swiech, A.
 {\em Eigenvalues and singular values of products of rectangular Gaussian random matrices.}
 Phys. Rev. E, vol. 82, 2010, 061114.
 
\bibitem{Chaterj} Chatterjee, S.
 {\em A generalization of the Lindeberg principle.} 
 Ann. Probab. 34, no.~6, 2006, 2061--2076.

\bibitem{CG} Chistyakov, G., G\"otze, F.
 {\em The arithmetic of distributions in free probability theory.}
 Cent. Eur. J. Math. 9, no.~5, 2011, 997--1050.
 
\bibitem{VoiculescuNica} Dykema, K. J., Nica, A., Voiculescu, D.
 {\em Free random variables. A noncommutative probability approach to free products with applications to random matrices, operator algebras and harmonic analysis on free groups.} 
 CRM Monograph Series, 1. American Mathematical Society, Providence, RI, 1992. vi+70 pp. 

\bibitem{Forrester}
 Forrester, P. J. 
 {\em Eigenvalue statistics for product complex Wishart matrices.}
 J.~Phys.~A: Math. Theor. 47, 2014, 345202. 
 arXiv:1401.2572.
 
\bibitem{ForresterLiu}
 Forrester, P. J., Liu, D.-Z.
 {\em Raney distributions and random matrix theory.}
 Preprint, arXiv:1404.5759.

\bibitem{Girko:84} Girko, V. L. {\em The circular law.} (Russian) 
 Teor. Veroyatnost. i Primenen. 29, no.~4, 1984, 669--679.

\bibitem{GT:04} G\"otze, F.; Tikhomirov, A. N.
 {\em Rate of convergence in probability to the Marchenko-Pastur law.} 
 Bernoulli 10, no.~3, 2004, 503--548. 

\bibitem{GT:10aop} G\"otze, F.; Tikhomirov, A. N.
 {\em The circular law for random matrices.} 
 Ann. Prob. 38, no.~4, 2010, 1444--1491.

\bibitem{GT:12} G\"otze, F.; Tikhomirov, A. N.
 {\em On the asymptotic spectrum of products of independent random matrices}. 
 Preprint, arXiv:1012.2710. 

\bibitem{Tikh:talk12} G\"otze, F.; Tikhomirov, A. N.
 {\em Limit theorems for products of large random matrices.}
 Talk given at the conference ``Random matrices and their applications'' in~Paris in October 2012,
 available at \verb|http://congres-math.univ-mlv.fr/| \verb|sites/congres-math.univ-mlv.fr/files/Tikhomirov.pdf|

\bibitem{Petz-1} Hiai, F.; Petz D.
 {\em Asymptotic Freeness Almost Everywhere for Random Matrices.}
 Acta. Sci. Math. (Szeged), 66, 2000, 801--826.

\bibitem{HornJ} Horn, R.; Johnson, Ch.
 {\em Topics in  Matrix analysis}.
 Cambridge University Press, 1991, 607 pp.

\bibitem{Kuijlaars-Zhang}
 Kuijlaars, A.; Zhang, L.
 {\em Singular values of products of Ginibre random matrices, 
 	multiple orthogonal polynomials and hard edge scaling limits.}
 Preprint, arXiv:1308.1003.	

\bibitem{PasLut} Lytova, A.; Pastur, L. 
 {\em Central limit theorem for linear eigenvalue statistics 
    of random matrices with independent entries.}
 Ann. Probab. 37, no. 5, 2009, 1778--1840.

\bibitem{MP:67} Marchenko, V.; Pastur, L.
 {\em The eigenvalue distribution in some ensembles of random matrices}.
 Math. USSR Sbornik, no. 1, 1967, 457--483.

\bibitem{Mays}
 Mays, A.
 \emph{A real quaternion spherical ensemble of random matrices.}
 Preprint, arXiv:1209.0888.

\bibitem{Mueller}
 M\"uller, R.
 {\em On the asymptotic eigenvalue distribution of concatenated vector-valued fading channels.}
 IEEE Trans. Inf. Theory 48, no.~7, 2002, 2086--2091.

\bibitem{Nica} Nica, A.
 {\em $R$-Transforms in Free Probability}. 
 Lectures in special semester "Free~probability theory and operator spaces." 
 IHP, Paris 1999.
 Available at \verb|http://www.math.uwaterloo.ca/~anica/NOTES/section11.pdf|

\bibitem{PanZhou}
 Pan, G.; Zhou, W.
 {\em Circular law, extreme singular values and potential theory.}
 J. Multivariate Anal. 101, no.~3, 2010, 645--656.

\bibitem{PasShcherb} Pastur, L.; Shcherbina, M. 
 {\em Eigenvalue distribution of large random matrices.} 
 Mathematical Surveys and Monographs, 171. American Mathematical Society, Providence, RI, 2011. xiv+632 pp.

\bibitem{SpeicherRao:2007} Rao, N. R.; Speicher, R. 
 {\em Multiplication of free random variables and the $S$-trans\-form: 
    the case of vanishing mean.} 
 Electron. Comm. Probab. 12, 2007, 248--258.

\bibitem{SR:10} O'Rourke, S.; Soshnikov, A. 
 {\em Product of independent non-Hermitian random matrices}. 
 Preprint, arXiv:1012.4497.

\bibitem{saff} Saff E. B.; Totik V.
 {\em Logarithmic potentials with external fields. Appendix B by Thomas Bloom.}
 Grundlehren der Mathematischen Wissenschaften [Fundamental Principles of Mathematical Sciences], 316. Springer-Verlag, Berlin, 1997. 

\bibitem{Speicher} Speicher, R.
 {\em Free probability theory.} 
 Preprint, arXiv:0911.0087.

\bibitem{TaoVu:10} Tao, T.; Vu, V.
 {\em Random matrices: universality of ESDs and the circular law. With an appendix by Manjunath Krishnapur.}
 Ann. Probab. 38, no. 5,  2010, 2023--2065. 
 
\bibitem{TT:13} Timushev D., Tikhomirov A. 
 {\em On the asymptotic distribution of singular values 
    of powers products of sparse random matrices.}
 Izvestia Komi Science Center of Ural Division of RAS, vol.13, 2013, 10--17.

\bibitem{Tikh:12} Tikhomirov, A. N.
 {\em On the asymptotics of the spectrum of the product of two rectangular random matrices.} (Russian) 
 Sibirsk. Mat. Zh. 52, no.~4, 2011, 936--954; 
 translation in Sib. Math. J. 52, no.~4, 2011, 747--762. 

 
\bibitem{Tikh:13} Tikhomirov, A. N.
 {\em Asymptotic distribution of the singular numbers for spherical ensemble matrices.} (Russian)
 Mat. Tr. 16, no.~2, 2013, 169--200.

\bibitem{Voiculescu:87} Voiculescu D. 
 {\em Multiplication of certain noncommuting random variables.} 
 J. Operator Theory 18, 1987, 223--235.

\bibitem{Voiculescu:98} Voiculescu, D. 
 {\em Lectures on free probability theory.} 
 In book: ``Lectures on Probability Theory and Statistics (Saint-Flour, 1998)'', 
 vol. 1738 of Lecture Notes in Math., 279--349.

\bibitem{Voiculescu:91} Voiculescu, D. 
 {\em Limit laws for Random matrices and free products}. 
 Invent. math., vol. 104, 1991, 201--220. 

\bibitem{Wigner:55} Wigner, E.~P.
 {\em Characteristic vectors of bordered matrices with infinite dimensions.} 
 Ann. of Math. (2) 62, 1955, 548--564. 

\end{thebibliography}
\end{document}